\documentclass[11pt]{amsart}
\usepackage{latexsym}

\usepackage{enumerate}

\makeatletter

\@namedef{subjclassname@2010}{

\textup{2010} Mathematics Subject Classification}
\subjclass[2010]{11N25,11N37 and 11N69}
\usepackage{amssymb, amsmath,amsthm, amsfonts}
\usepackage{comment}
\usepackage[linktocpage,bookmarksdepth=2]{hyperref}
\hypersetup{
     colorlinks,
     linkcolor=green,
     citecolor=green,
     urlcolor=blue
   }

\newtheorem{thm}{Theorem}[section]
\newtheorem{prop}[thm]{Proposition}

\newtheorem{cor}[thm]{Corollary}
\newtheorem{lem}[thm]{Lemma}
\theoremstyle{definition}
\newtheorem{rem}[thm]{Remark}

\newcommand{\ra}{\rightarrow}
\newcommand{\bk}{\backslash}
\newcommand{\sg}{\sigma}

\newcommand{\mb}{\mathbb}
\newcommand{\mf}{\mathfrak}

\newcommand{\C}{\mb{C}}
\newcommand{\N}{\mb{N}}
\newcommand{\Z}{\mb{Z}}
\newcommand{\G}{\mb{G}}
\newcommand{\Q}{\mb{Q}}
\newcommand{\F}{\mb{F}}
\newcommand{\mc}{\mathcal}
\newcommand{\Fc}{\mc{F}}
\newcommand{\Gc}{\mc{G}}
\newcommand{\Hc}{\mc{H}}
\newcommand{\Lc}{\mc{L}}

\newcommand{\ord}{\operatorname{ord}}
\newcommand{\SL}{\operatorname{SL}}

\newcommand{\geom}{\mathrm{geom}}
\newcommand{\arith}{\mathrm{arith}}
\newcommand{\rank}{\operatorname{rank}}
\newcommand{\tr}{\operatorname{tr}}
\newcommand{\Std}{\operatorname{Std}}
\newcommand{\mult}{\operatorname{mult}}
\newcommand{\Frob}{\operatorname{Frob}}
\newcommand{\Sing}{\operatorname{Sing}}
\newcommand{\Swan}{\operatorname{Swan}}
\newcommand{\Gal}{\operatorname{Gal}}

\newcommand{\llf}{\left\lfloor}

\newcommand{\e}{\varepsilon}
\newcommand{\eps}{\epsilon}

\newcommand{\rrf}{\right\rfloor}

\newcommand{\mbf}{\boldsymbol}
\newcommand{\asum}{\sideset{}{^{\ast}}\sum}
\newcommand{\ssum}{\sideset{}{^{\star}}\sum}

\renewcommand{\bar}{\overline}

\makeatletter
\let\@@pmod\pmod
\DeclareRobustCommand{\pmod}{\@ifstar\@pmods\@@pmod}
\def\@pmods#1{\mkern4mu({\operator@font mod}\mkern 6mu#1)}
\makeatother

\begin{document}
\title{Squarefree Integers in Arithmetic Progressions to Smooth Moduli}
\author{Alexander P. Mangerel}
\address{Centre de Recherches Math\'{e}matiques, Universit\'{e} de Montr\'{e}al, Pavillon Andr\'{e}-Aisenstadt, 2920 Chemin de la Tour, Montr\'{e}al, Qu\'{e}bec H3T 1N8} \address{
Department of Mathematical Sciences, Durham University, Upper Mountjoy Campus, Stockton Road, Durham DH1 3LE}
\email{smangerel@gmail.com}

\begin{abstract}
Let $\e > 0$ be sufficiently small and let $0 < \eta < 1/522$. We show that if $X$ is large enough in terms of $\e$ then for any squarefree integer $q \leq X^{196/261-\e}$ that is $X^{\eta}$-smooth one can obtain an asymptotic formula with power-saving error term for the number of squarefree integers in an arithmetic progression $a \pmod{q}$, with $(a,q) = 1$. In the case of squarefree, smooth moduli this improves upon previous work of Nunes, in which $196/261 = 0.75096...$ was replaced by $25/36 = 0.69\bar{4}$. This also establishes a level of distribution for a positive density set of moduli that improves upon a result of Hooley. We show more generally that one can break the $X^{3/4}$-barrier for a density 1 set of $X^{\eta}$-smooth moduli $q$ (without the squarefree condition). \\
Our proof appeals to the $q$-analogue of the van der Corput method of exponential sums, due to Heath-Brown, to reduce the task to estimating correlations of certain Kloosterman-type complete exponential sums modulo prime powers. In the prime case we obtain a power-saving bound via a cohomological treatment of these complete sums, while in the higher prime power case we establish savings of this kind using $p$-adic methods.
\end{abstract}

\maketitle
\setcounter{tocdepth}{1}
\tableofcontents

\section{Introduction}
It is a classical problem in analytic number theory to study the distribution of well-known sequences in arithmetic progressions. The example of interest to us here is the sequence of square-free integers, i.e., integers $n$ such that $p^2 \nmid n$ for all primes $p$. 
Writing the indicator function for this sequence as $\mu^2(n)$ for all $n \in \mb{N}$, where $\mu$ denotes the M\"{o}bius function\footnote{Recall that the M\"{o}bius function is the arithmetic function satisfying $\mu(n) = 0$ if $n$ is not squarefree, and otherwise $\mu(p_1\cdots p_k) = (-1)^k$ if $p_1,\ldots,p_k$ are distinct primes.}, it is known that for a given modulus $q$, the asymptotic equidistribution estimate 
\begin{align}\label{eq:asymp}
\sum_{n \leq X \atop n \equiv a \pmod{q}} \mu^2(n) = \frac{1}{\phi(q/(q,a))} \sum_{n \leq X \atop (n,q) = (a,q)} \mu^2(n) + o(X/q)
\end{align}
for the count of squarefree integers in a progression $a \pmod{q}$ holds if $X$ is sufficiently large relative to $q$. It is a challenging question to determine the optimal constant $\theta \in [0,1)$ such that \eqref{eq:asymp} holds as soon as $q \leq X^{\theta}$. It is widely believed that any $\theta < 1$ should be admissible, but this is far from proven in general. 
It has been known for some time (see \cite{Prach}) that any $\theta < 2/3$ is admissible. At present, the best result that is available for all moduli $q$ is that any $\theta < 25/36 = 0.69\bar{4}$ is admissible, which is a recent result of Nunes \cite{Nun}. \\
If we impose constraints on the modulus $q$ then one might expect that it is possible to improve this range of $\theta$. For instance, Hooley \cite{Hoo75} showed that $\theta$ can be taken in the interval $[2/3,3/4)$ provided that $q$ has a moderately large\footnote{For instance, when $\theta = 3/4-\e$ for $\e > 0$ small, $p$ must have size $X^{2/3-\e}$.} prime factor $p$, the size of which depends on $\theta$. As a consequence, he deduced that for a positive proportion of moduli $q \leq X^{3/4-\e}$ the asymptotic \eqref{eq:asymp} holds (the proportion depends on $\e$ and tends to $0$ as $\e \ra 0^+$). Recently, Liu, Shparlinski and Zhang \cite{LSZ} improved upon this result by proving a Bombieri-Vinogradov type theorem for squarefree integers, showing that the required asymptotic formula holds for all residue classes $(a,q) = 1$ for almost all $q \leq X^{3/4-\e}$.\\
%
In this paper, we shall improve upon Nunes' result for a different, natural collection of moduli, specifically those $q$ that are $X^{\eta}$-smooth (otherwise called $X^{\eta}$-friable), i.e., such that all prime factors $p$ of $q$ satisfy $p \leq X^{\eta}$ for $\eta > 0$ small. For many such moduli, including all squarefree $X^{\eta}$-smooth moduli, we shall in fact be able to improve upon the range of admissibility $\theta < 3/4$ from the works of Hooley and of Liu, Shparlinski and Zhang.\\
More precisely, our first main result is as follows. 
\begin{thm}\label{thm:Smooth}
Let $\e > 0$ be small and let $0 < \eta < 1/522$. Let $X$ be large in terms of $\eta,\e$, and let $q \leq X^{196/261-\e}$ be $X^{\eta}$-smooth and squarefree. Then there is a $\delta > 0$, depending only on $\eta$ and $\e$, such that for any $a \pmod{q}$ with $(a,q) \leq X^{\e}$,
$$
\sum_{n \leq X \atop n \equiv a \pmod{q}} \mu^2(n) = \frac{1}{\phi(q/(q,a))} \sum_{n \leq X \atop (n,q/(q,a)) = 1} \mu^2(n) + O_{\e}\left(\frac{X^{1-\delta}}{q}\right).
$$
\end{thm}
\begin{rem}
Our proof actually shows that for $q \leq X^{3/4-\e}$ one may take as the smoothness exponent any $\eta \leq 6/25$, whereas for $X^{3/4-\e} < q \leq X^{196/261-\e}$ the admissible range is $\eta < 1/522$; it is probable that this range in $\eta$ may be improved somewhat, though it was not our primary objective to obtain an optimal such range. 
\end{rem}

\noindent Our result has the following trivial corollaries. The first improves upon the range\footnote{We emphasize that $196/261 = 3/4 + 1/1044 > 3/4$.} of moduli that is accessible in the density result of Hooley quoted earlier.
\begin{cor} \label{cor:posProp}
For any $\e > 0$, a positive proportion of the moduli $q \leq X^{196/261-\e}$ satisfy
$$
\max_{(a,q)= 1} \left|\sum_{n \leq X \atop n \equiv a \pmod{q}} \mu^2(n) - \frac{1}{\phi(q)} \sum_{n \leq X \atop (n,q) = 1} \mu^2(n)\right| = o_{\e}(X/q).
$$
\end{cor}

\begin{cor}
Let $\e > 0$ be small, $0 < \eta < 1/522$ and let $X$ be sufficiently large in terms of $\eta$ and $\e$. Let $q$ be squarefree and $X^{\eta}$-smooth, and let $a$ be a coprime residue class modulo $q$. Then provided $X \geq q^{261/196+\e}$ there is a squarefree integer in the set $\{n \leq X : n \equiv a \pmod{q}\}$. 
\end{cor}
It is desirable to remove the assumption that $q$ is squarefree, so that the above results continue to hold in some form for $q$ merely $X^{\eta}$-smooth. While we are not able to derive this conclusion for all such moduli, we do show that almost all $X^{\eta}$-smooth moduli do have this property.
\begin{thm} \label{thm:notSfree}
There is a $\lambda > 0$ such that the following is true. Let $\eta > 0$, and set 
$$
\mc{Q}_{\lambda}(X) := \{q \leq X^{3/4+\lambda} : P^+(q) \leq X^{\eta}\}. 
$$
Then for all but $O_{\eta}(|\mc{Q}_{\lambda}(X)|/\log X)$ moduli $q \in \mc{Q}_{\lambda}(X)$, we have
$$
\max_{(a,q) = 1} \left|\sum_{n \leq X \atop n \equiv a \pmod{q}} \mu^2(n) - \frac{1}{\phi(q)}\sum_{n \leq X \atop (n,q) = 1} \mu^2(n)\right| = o_{\eta}(X/q).
$$
\end{thm}
\noindent Given $y \geq 2$ we will say that a positive integer $q$ is \emph{$y$-ultrasmooth} if, whenever $p^n||q$ we have $p^n \leq y$. What we will actually show is that the conclusion of Theorem \ref{thm:notSfree} applies to all elements of
$$
\mc{Q}_{\lambda}'(X) := \{q \in \mc{Q}_{\lambda}(X) : q \text{ is $X^{\eta}$-ultrasmooth}\}.
$$
It is easy to see that by the union bound,
\begin{align*}
|\mc{Q}_{\lambda}(X) \bk \mc{Q}_{\lambda}'(X)| &\leq \sum_{p \leq X^{\eta}} |\{q \in \mc{Q}_{\lambda}(X) : p^{\nu}|| q, p^{\nu} > X^{\eta}\}| \leq \sum_{p \leq X^{\eta}} \sum_{q \leq X^{3/4+\lambda-\eta}} 1_{P^+(q) \leq X^{\eta}} \\
&\ll_{\eta} \frac{X^{\eta}}{\log X} |\mc{Q}_{\lambda}(X)|X^{-\eta} \ll \frac{|\mc{Q}_{\lambda}(X)|}{\log X}
\end{align*} 
which implies the claimed bound on the exceptional set when $X \geq X_0(\eta)$.
\begin{rem}
To be precise, there are two reasons why we do not treat all smooth moduli in this paper. The first is that our method does not directly treat moduli $q$ such that $q$ has a large power of 2 or 3 as a divisor. It is likely that it could be modified to treat this case, but at the cost of adding a non-trivial amount of pages to this already long paper. \\
A more serious issue arises from the fact that in the course of the proof we need to be able to factor our modulus $q$ (or more precisely, a suitably chosen divisor of it) into \emph{coprime} parts with well-controlled sizes (here we allow perturbations of size $X^{\eta}$ only). If $q$ were merely $X^{\eta}$-smooth but not $X^{\eta}$-ultrasmooth, e.g., if $q$ were divisible by a prime power $p^{\nu} > X^{1/100}$, say, with $p \leq X^{\eta}$, then this prime power would have to arise in one of the factors, biasing its size in a manner that might be incompatible with the bounds we obtain. The $X^{\eta}$-ultrasmoothness condition is in place to prevent such a bias in size from occurring.
\end{rem}

\begin{rem}
It is worth noting that our proof actually shows that we may asymptotically estimate the number of squarefree integers in a progression modulo $q$ when $q \leq X^{3/4-\e}$ which is $X^{\eta}$-smooth, without further assumptions (see Section \ref{subsec:3of4} for a proof). However, this by itself is not sufficient to improve upon Hooley's positive density result.
\end{rem}

\subsection{Proof Strategy}
To prove Theorems \ref{thm:Smooth} and \ref{thm:notSfree}, we will use a method of Heath-Brown that was used by Irving \cite{Irv} to study the distribution of the divisor function in arithmetic progressions. To motivate our argument, it is useful to see where obstructions occur in the classical treatment of counting squarefree integers in progressions, as found in \cite{Prach}. \\
Assume in what follows that $(a,q) = 1$, for convenience. Using the classical identity 
$$
\mu^2(n) = \sum_{kl^2 = n} \mu(l),
$$ 
and decoupling the parameters $k$ and $l$ by localizing them in short intervals (as we do in Section 2) it can be shown that the error term in \eqref{eq:asymp} is controlled by averaged incomplete exponential sums\footnote{Henceforth, given $x \in \mb{R}$ and $q \in \mb{N}$ we shall write $e(x) := e^{2\pi i x}$ and $e_q(x) := e(x/q)$.}
\begin{align} \label{eq:motiv}
\asum_{k \pmod{q}} \frac{1}{k} \left|\sum_{d \in I \atop (d,q) =1} e_q\left(ka\bar{d}^2\right)\right|,
\end{align}
where $I$ is some interval of size $< q$.  Using the completion method, one is led to bound averages of ``quadratic Kloosterman sums'' 
$$
K_2(A,B;q) := \asum_{x \pmod{q}} e_q\left(A\bar{x}^2 + Bx\right),
$$
where $A \in (\mb{Z}/q\mb{Z})^{\times}$ and $B \in \mb{Z}/q\mb{Z}$. Using the Chinese remainder theorem, we may of course factor $K_2$ as a product of complete sums modulo prime powers $p^n||q$. Each such factor can be estimated pointwise by $O(p^{n/2})$, either using: 
\begin{enumerate}[(i)]
\item a trivial application of the Bombieri--Dwork--Weil bound (see Lemma \ref{lem:WeilK2}) when $n = 1$, or 
\item the $p$-adic method of stationary phase for $n \geq 2$ (see Lemma \ref{lem:expK2}). 
\end{enumerate}
The resulting bound $O_{\e}(X^{\e}q^{1/2})$ for the complete sum modulo $q$ is of size $o(X/q)$ as required, as long as $q \leq X^{2/3-\e}$. \\
One way to go beyond the $X^{2/3}$ barrier is to try to exploit the averaging in $k$ in \eqref{eq:motiv}, rather than employing a pointwise bound. Indeed, when $q$ factors sufficiently nicely, the $q$-analogue of the van der Corput method of exponential sums, developed by Ringrose \cite{Ringf} and Heath-Brown \cite[Theorem 2]{HB}, enables one to reduce the problem above to estimating correlations of complete exponential sums (to be discussed momentarily). Using (a variant of) a result of Fouvry--Ganguly--Kowalski--Michel \cite{FGKM14} to estimate such correlations, Irving is able to treat $X^{\e}$-smooth and squarefree moduli of size $q \leq X^{2/3+1/246-\e}$. Aside from the fact that the sums $K_2$ entering the picture differ in behaviour from genuine Kloosterman sums (so that the results of \cite{FGKM14} do not apply), this strategy by itself is insufficient in the case of squarefree integers, even falling short of Nunes' result (see Subsection \ref{subsubsec:BdsAvg} for further details). \\
To do better, we incorporate an additional idea. The key difficulty in understanding the distribution of squarefree integers in progressions is to estimate the count of points on the curve $xy^2 \equiv a \pmod{q}$, for $x$ and $y$ lying in certain ranges. Clearly, if $I,J \subset \mb{Z}/q\mb{Z}$ are intervals and $q'$ is any divisor of $q$ then
$$
|\{(x,y) \in I \times J : xy^2 \equiv a \pmod{q}\}| \leq |\{(x,y) \in I \times J : xy^2 \equiv a \pmod{q'}\}|;
$$
moreover, observe that the analogous bound for \eqref{eq:motiv} with $q'$ (of suitable size) in place of $q$, i.e., $O(X^{\e}(q')^{1/2})$, can be of size $o(X/q)$ for larger choices of $q$ than $X^{2/3-\e}$. There is therefore an advantage in working with exponential sums modulo a suitably sized divisor of $q$, whenever such a divisor can be found; this was observed as well in \cite{Hoo75}, but its implementation differs from what we do here. As we shall show below, the fact that $q$ is $X^{\eta}$-smooth means we can find divisors of $q$ of any prescribed size (up to factors of size $X^{\eta}$), in particular the size required to use the trick above. We therefore end up applying Irving's method to treat Kloosterman-type sums modulo $q'$ instead, which when combined with the analysis of correlations described above results in a gain not over the range $q \leq X^{2/3-\e}$, but over a larger range of $q$ instead.\\
Let us describe more precisely a key feature of our argument. A crucial step in Irving's method (again for the divisor function) involves giving non-trivial estimates for the correlations 
$$
\sum_{b \in I} \text{Kl}(a,b+h_1;p)\cdots \text{Kl}(a,b+h_k;p),
$$
where $p \mid q$, $h_1,\ldots,h_k \in \mb{F}_p$, $a \in \mb{F}_p^{\times}$, $I \subset \mb{F}_p$ is an interval and
$$
\text{Kl}(A,B;p) := \sum_{x \in \mb{F}_p^{\times}} e_p\left(A\bar{x} + Bx\right) \text{ for $A,B \in \mb{F}_p$}
$$
is the classical Kloosterman sum modulo $p$. In our circumstances, we treat $K_2$ sums to both prime and prime power moduli, each case of which requires a separate analysis. \\
In the prime case, the corresponding correlation sum that we need to treat is of the form\footnote{We emphasize that unlike classical Kloosterman sums, the $K_2$ sums are not real-valued in general.}
$$
\sum_{b \in I} K_2(a, b+h_1;p)\cdots K_2(a,b+h_k;p) \bar{K}_2(a,b+h'_1;p) \cdots \bar{K}_2(a,b+h'_l;p),
$$
where $k+l \geq 1$. We treat (completions of) such sums, which are the subject of Theorem \ref{thm:boundCorrCompl}, in Section \ref{sec:corPrime} of this paper, using cohomological methods, in particular the sheaf-theoretic Fourier transform of Deligne. \\
Specifically, we view $K_2$ as the $\ell$-adic Fourier transform of the trace function of an Artin--Schreier sheaf ($\ell \neq p$ being an auxiliary prime), which is itself a trace function, pointwise pure of weight 1. Treating its correlations amounts to identifying cases in which tensor products of the underlying Galois representations are, or are not, geometrically trivial, a task facilitated by the Goursat--Kolchin--Ribet criterion (see \cite{FKMSumProducts} for an array of example applications of this method).\\
In the prime power case, our work is simplified (in terms of the required theoretical preliminaries, but not in the amount of technical details) by the fact that these $K_2$ sums can be explicitly computed using the $p$-adic stationary phase method. For instance, when $n \geq 2$ and $p>3$ is prime we have (see Lemma \ref{lem:expK2})
$$
K_2(A,B;p^n) = p^{n/2} \left(\frac{3A}{p^n}\right) \e_{p,n} \sum_{y \pmod{p^{\llf n/2\rrf}} \atop y^3 \equiv \bar{2A}B \pmod{p^{\llf n/2\rrf}}} e_{p^n}(3Ay^2),
$$
for any $A \in (\mb{Z}/p^n\mb{Z})^{\times}$ and $B \in \mb{Z}/p^n\mb{Z}$, where $\e_{p,n} \in S^1$. The correlation problem then revolves around bounding exponential sums over variables from a variety determined by the polynomials $\{X^3-\bar{2a}(b+h_i)\}_i$ and $\{X^3-\bar{2a}(b+h_j')\}_j$, as $b$ varies. \\
An important work treating such correlations, with $K_2$ sums again replaced by Kloosterman sums, was undertaken by Ricotta and Royer \cite{RicRoy}. They used their estimates to establish distribution theorems for Kloosterman sums modulo $p^n$, as $p \ra \infty$. However, their method is efficient mainly when $n$ is fixed, as it relies on treating the correlation sum (via explicit formulae for Kloosterman sums to prime powers) as an exponential sum whose phase function turns out to be a polynomial modulo $p^n$ of degree essentially as large as $n$, and then applying the standard van der Corput--Weyl method. We also employ this strategy, but instead use Vinogradov's method (and the recent solution to the Vinogradov main conjecture, due to Bourgain--Demeter--Guth \cite{BDG}, and independently Wooley \cite{Woo}) in place of van der Corput's method, which leads to a stronger result when $n$ is of bounded size as $p$ grows. \\
In our circumstances we will also require a treatment that is effective for moduli $p^n$ with rather large values of $n$ and $p$ possibly fixed. Fortunately, very recent work of Mili\'{c}evi\'{c} and Zhang \cite{MilZhan} introduced a method that suits this situation, in which one iteratively applies the stationary phase method to recover from the correlation sum a sum over a ``generically trivial'' variety, up to small error (at least when $n$ is large enough), rather than using Weyl sum estimates. A combination of the arguments\footnote{Unfortunately, the bounds for prime power moduli $p^n$ with $n \geq 2$ are rather poorer than the bounds for prime moduli. This is due, in part, to the lack of rigid algebraic data. 
As a result of these weaker conclusions, we have chosen to leave the exponent $\delta$ in Theorem \ref{thm:notSfree}, which is necessarily weaker than what is obtainable in Theorem \ref{thm:Smooth}, inexplicit.} in each of these regimes will suit our needs. 
\section*{Acknowledgments}
The author is grateful to the anonymous referees for a careful reading of the paper, for several corrections to the content of Section \ref{sec:corPrime}, and for many recommendations that helped in improving the exposition.
The author warmly thanks Corentin Perret-Gentil for invaluable discussions and suggested references that were crucial to the proof of Theorem \ref{thm:boundCorrCompl},  as well as many helpful comments contributing to better readability. The author would also like to thank Aled Walker for a useful suggestion leading to an improvement to the arguments in the proof of Theorem \ref{thm:expsumPrimPow}. This paper began while the author was visiting the California Institute of Technology in February 2020. The author would like to thank Caltech for the excellent working conditions, as well as Maksym Radziwi\l\l \  for the invitation and for bringing the work \cite{Nun} to his attention. 

\section{Setting up the Key Estimate}\label{sec:setupEst}

\subsection{First reductions}
Let $\e,\eta > 0$ be sufficiently small, let $X$ be large relative to $\e$ and $\eta$, and let $X^{2/3-\e} \leq q \leq X^{9/10-\e}$ be $X^{\eta}$-smooth. Given $a$ a residue class modulo $q$, an arithmetic function $g : \mb{N} \ra \mb{C}$ and a set $E\subset \N$, define
$$
\Delta_{g}(E;q,a) := \sum_{n \in E \atop n \equiv a \pmod*{q}} g(n) - \frac{1}{\phi(q/(a,q))} \sum_{n \in E \atop (n,q) = (a,q)} g(n).
$$
We shall also use the shorthand $\Delta_g(X;q,a):=\Delta_g([1,X]\cap\N;q,a)$. For the remainder of this section we will assume that $(a,q) = 1$. \\
Take $g = \mu^2$, the indicator function for the squarefree integers.  Using the classical identity $\mu^2(n) = \sum_{md^2 = n} \mu(d)$ we obtain
$$
\Delta_{\mu^2}(X;q,a) = \sum_{d \leq \sqrt{X}\atop (d,q)=1} \mu(d) \Delta_1(X/d^2; q, a\bar{d}^2),
$$
where $\bar{d}$ denotes the residue class modulo $q$ with $d\bar{d} \equiv 1 \pmod{q}$. Let $\delta \in (0,1/20)$ and let 
$$
X^{\delta+\e} \leq V_0 \leq X^{1-\delta}/q
$$ 
be a parameter to be chosen. For $m\le Y$ and $1 \leq b \leq m$, note the trivial bound 
$$
\Delta_1(Y;m,b) = \left \lfloor \frac{Y-b}{m} \right\rfloor - \frac{1}{\phi(m)} \left(\frac{\phi(m)}{m} Y + O(\tau(m))\right) = O(1),
$$
whence it follows that
$$
\Delta_{\mu^2}(X;q,a) = \sum_{V_0 < d \leq \sqrt{X}\atop (d,q)=1} \mu(d)\Delta_1(X/d^2;q,a\bar{d}^2) + O\left(V_0\right).
$$
It will be convenient in what follows later to remove the coefficient $\mu(d)$. Decomposing dyadically the sum in $d$ and applying the triangle inequality, we find that there is a $V_0 < V \leq \sqrt{X}$ such that
\begin{align*}
  \Delta_{\mu^2}(X;q,a) &\ll (\log X) \left|\sum_{d \sim V\atop (d,q)=1} \mu(d)\Delta_1(X/d^2;q,a\bar{d}^2)\right| + V_0 \\
                        &\leq (\log X) \sum_{d \sim V\atop (d,q)=1} \left|\Delta_1(X/d^2;q,a\bar{d}^2)\right| + V_0.  
\end{align*}
Next, we further subdivide the range of $m \leq X/d^2$ into dyadic subintervals, leading to the existence of $U$ satisfying $UV^2 \leq X$, such that
$$
\Delta_{\mu^2}(X;q,a) \ll (\log X)^2 \sum_{d \sim V \atop (d,q) = 1} \left|\sum_{\substack{m \sim U \\ md^2 \leq X \\ m \equiv a \bar{d}^2 \pmod*{q}}} 1- \frac{1}{\phi(q)} \sum_{\substack{m \sim U \\ md^2 \leq X \\ (m,q) = 1}} 1 \right| + V_0.
$$
To remove the condition $md^2 \leq X$ we split $(U,2U]$ and $(V,2V]$ into subintervals of respective lengths $UV_0^{-1}$ and $VV_0^{-1}$, of which there are $\ll V_0^{2}$ in total. We thus find that there are $U_1 \in (U,2U]$, $V_1 \in (V,2V]$ with $U_1V_1^2 \leq 8X$ such that
$$
\Delta_{\mu^2}(X;q,a) \ll V_0^2 (\log X)^2 \sum_{d \in I(V_1) \atop (d,q) = 1} \left|\Delta_1(I(U_1); q,a\bar{d}^2)\right| + V_0 + \mc{E},
$$
where
\[I(U_1) = (U_1, U_1 + UV_0^{-1}],\text{ and}\quad I(V_1) = (V_1,V_1 + VV_0^{-1}]\]
and $\mc{E}$ counts the number of pairs $(m,d)$ such that: \\
\begin{enumerate}[(i)]
\item $md^2 > X$
\item $md^2 \equiv a \pmod{q}$ and 
\item $d \in I(V_1')$, $m \in I(U_1')$ with $U_1' \in (U,2U]$, $V_1' \in (V,2V]$ satisfying $U_1'(V_1')^2 \leq X$.
\end{enumerate}
We easily see that
\begin{align*}
  X < md^2 &\leq (U_1' + U'V_0^{-1})(V'_1 + V'V_0^{-1})^2\\
           &\leq U'_1(V'_1)^2 + U'V'_1V_0^{-1} + 3U'_1V'_1V'V_0^{-1} + 3U'V'V'_1V_0^{-2}\\
           &\leq X + O(XV_0^{-1}).
\end{align*}
As $X /V_0 \geq q^{1+\delta}$, by Shiu's theorem \cite{Shiu} the contribution from these pairs is
$$
\mc{E}\ll \sum_{X < n \leq X + O(X/V_0) \atop n \equiv a \pmod*{q}} \tau(n) + \frac{1}{\phi(q)} \sum_{X  < n \leq X + O(X/V_0) \atop (n,q) = 1} \tau(n) \ll_{\delta} \frac{X\log X}{qV_0}.
$$
We thus find that
\begin{equation}
  \label{eq:DeltamuDelta1}
  \Delta_{\mu^2}(X;q,a) \ll_{\delta} V_0^2(\log X)^2 \sum_{d \in I(V_1) \atop (d,q) = 1} \left|\Delta_1(I(U_1);q,a\bar{d}^2)\right| + V_0 + X(\log X)/(qV_0),
\end{equation}
again with $U_1V_1^2 \leq 8X$. Note that we may assume that $V_1 > X^{1-\delta-\e}/(qV_0)$, since otherwise we immediately obtain
\begin{align}\label{eq:assumpV1}
\Delta_{\mu^2}(X;q,a) \ll_{\e,\delta} V_0^2 X^{\e} |I(V_1)| + V_0 + X(\log X)/(qV_0) \ll X^{1-\delta}/q.
\end{align}
Let $\tilde{q}$ be a divisor of $q$ to be chosen later (the smoothness assumption on $q$ will be useful in this selection). Using the fact that when $U_1V_1^2 \leq 8X$,
$$
\sum_{d \in I(V_1)} \frac{1}{\phi(q')} \sum_{m \in I(U_1)} 1_{(m,q') = 1} = \frac{|I(V_1)|}{\phi(q')}\left(\frac{\phi(q')}{q'}|I(U_1)| + O(\tau(q'))\right) \ll \frac{X}{q'V_1V_0^2} + X^{\e}\frac{V_1}{q'V_0}
$$
for $q' \in \{q,\tilde{q}\}$, we obtain
\begin{align*}
\sum_{d \in I(V_1) \atop (d,q) = 1} \left|\Delta_1(I(U_1);q,a\bar{d}^2)\right| &\leq \sum_{d \in I(V_1) \atop (d,q) = 1} \left(\sum_{m \in I(U_1) \atop m \equiv a \bar{d}^2 \pmod*{q}} 1 + \frac{1}{\phi(q)} \sum_{m \in I(U_1) \atop (m,q) = 1} 1\right) \\
&\leq \sum_{d \in I(V_1) \atop (d,\tilde{q}) = 1} \left(\sum_{m \in I(U_1) \atop m \equiv a \bar{d}^2 \pmod*{\tilde{q}}} 1 + \frac{1}{\phi(\tilde{q})} \sum_{m \in I(U_1) \atop (m,\tilde{q}) = 1} 1\right) + O\left(\frac{X}{\tilde{q}V_0^2V_1} + X^{\e}\frac{V_1}{\tilde{q}V_0}\right) \\
&= \sum_{d \in I(V_1) \atop (d,\tilde{q}) = 1} \Delta_1( I(U_1);\tilde{q},a\bar{d}^2) + O\left(\frac{X}{\tilde{q}V_1V_0^2} + \frac{X^{\e}V_1}{V_0\tilde{q}}\right).
\end{align*}
Hence, from \eqref{eq:DeltamuDelta1} and $V_1 \ll X^{1/2}$,
\begin{multline*}
\Delta_{\mu^2}(X;q,a) \ll_{\e} V_0^2(\log X)^2 \left|\sum_{d \in I(V_1) \atop (d,\tilde{q}) = 1} \Delta_1(I(U_1);\tilde{q},a\bar{d}^2)\right|
+ V_0 + \frac{X^{1+\e}}{qV_0} + \frac{X^{1+\e}}{\tilde{q}V_1} +\frac{ V_0X^{1/2+\e}}{\tilde{q}}
\end{multline*}
for $\tilde{q}\mid q$, $X^{\delta + \e} \leq V_0 \leq X^{1-\delta}/q$ and $V_1 \geq \max\{V_0, X^{1-\delta-\e}/(qV_0)\}$. \\
Having decoupled the variables $d$ and $m$ and removed the weight $\mu(d)$, we now introduce additive characters into the fold. By orthogonality, we have
\begin{multline*}
\sum_{d \in I(V_1) \atop (d,\tilde{q}) = 1} \left(\sum_{m \in I(U_1) \atop m\equiv a \bar{d}^2 \pmod*{\tilde{q}}} 1- \frac{1}{\phi(\tilde{q})} \sum_{m \in I(U_1) \atop (m,\tilde{q}) = 1} 1\right)
= \frac{1}{\tilde{q}} \sum_{k \pmod*{\tilde{q}}} \left(\sum_{d \in I(V_1) \atop (d,\tilde{q}) = 1} e_{\tilde{q}}(-ka\bar{d}^2)\right) \left(\sum_{m \in I(U_1)} e_{\tilde q}(km)\right)\\
- \frac{1}{\phi(\tilde{q})} \sum_{d \in I(V_1) \atop (d,\tilde{q}) = 1} \sum_{m \in I(U_1) \atop (m,\tilde{q}) = 1} 1.
\end{multline*}
By the sieve, 
$$
\sum_{m \in I(U_1) \atop (m,\tilde{q}) = 1} 1 = \frac{\phi(\tilde{q})}{\tilde{q}}|I(U_1)| + O(\tau(\tilde{q})),
$$
the main term of which cancels the sum with $k = 0$ above, and so we obtain
\begin{align*}
  \Delta_{\mu^2}(X;q,a) &\ll_{\e} \frac{V_0^2(\log X)^2}{\tilde{q}}\sum_{k \pmod*{\tilde{q}} \atop k \neq 0}\left|\sum_{d \in I(V_1) \atop (d,\tilde{q}) = 1} e_{\tilde{q}}(-ka\bar{d}^2)\sum_{m \in I(U_1)} e_{\tilde{q}}(km)\right| \\
  &+V_0+ \frac{X^{1+\e}}{qV_0} + \frac{X^{1+\e}}{\tilde{q}V_1}+ \frac{V_0X^{1/2+\e}}{\tilde{q}} \\
                        &\ll_{\e} \frac{V_0^2(\log X)^2}{\tilde{q}}\sum_{1 \leq |k| \leq \tilde{q}/2}\left|\sum_{d \in I(V_1) \atop (d,\tilde{q}) = 1} e_{\tilde q}(ka\bar{d}^2)\right|\left|\sum_{m \in I(U_1)} e_{\tilde q}(-km)\right| \\
  &+ V_0\left(1 + \frac{X^{1/2+\e}}{\tilde{q}}\right)  + \frac{X^{1+\e}}{\tilde{q}V_1} + \frac{X^{1+\e}}{qV_0}.
\end{align*}
Applying the geometric series estimate
$$
\sum_{m \in I(U_1)} e_{\tilde q}(km) \ll \min\{|I(U_1)|, \|k/\tilde{q}\|^{-1}\} \ll \tilde{q}/k
$$
for $1 \leq k \leq \tilde{q}/2$, we thus obtain
$$
\Delta_{\mu^2}(X;q,a) \ll_{\e} V_0^2(\log X)^2 \sum_{1 \leq k \leq \tilde{q}/2}\frac{1}{k}\left|\sum_{d \in I(V_1) \atop (d,\tilde{q}) = 1} e_{\tilde q}(ka\bar{d}^2)\right| + \frac{X^{1+\e}}{\tilde{q}V_1} + \frac{X\log X}{qV_0} + V_0.
$$
\subsection{Bounding incomplete exponential sums on average}
Our main objective from here on is to get a suitable estimate on average over $k$ for
$$
S_{\tilde{q},a}(k;V_1) := \sum_{d \in I(V_1) \atop (d,\tilde{q}) = 1} e_{\tilde q}\left(ka\bar{d}^2\right).
$$
To simplify matters further, we separate the range of $k$ according to $(k,\tilde{q}) = f$, giving
\begin{align}
\sum_{1\leq k \leq \tilde{q}-1} \frac{|S_{\tilde{q},a}(k;V_1)|}{k} = \sum_{f|\tilde{q} \atop f < \tilde{q}} \frac{1}{f}\quad \asum_{k' \pmod*{\tilde{q}/f}} \frac{|S_{\tilde{q}/f,a}(k';V_1)|}{k'}. \label{eq:fsum}
\end{align}
Fix $f\mid\tilde{q}$ with $f < \tilde{q}$ for the moment, and put $q' := \tilde{q}/f$. Completing the sum, we obtain
\begin{align}
S_{q',a}(k';V_1) &= \asum_{l \pmod*{q'}} e_{q'}(k'a\bar{l}^2)  \sum_{d \in I(V_1) \atop d \equiv l \pmod*{q'}} 1 \nonumber\\
&= \frac{1}{q'}\sum_{r \pmod*{q'}}\quad \asum_{l \pmod*{q'}} e_{q'}(k'a\bar{l}^2 + rl)  \sum_{d \in I(V_1)} e_{q'}(-rd) \nonumber \\
&= \frac{1}{q'}\sum_{r=1}^{q'} e_{q'}(-rV_1) K_2(k'a,r;q') g_{q'}(r), \label{eq:withGq}
\end{align}
where for $Q \geq 2$ we have set
$$
K_2(A,B;Q) := \asum_{x \pmod*{Q}} e_Q\left(A\bar{x}^2 + Bx\right) \quad (A\in(\Z/Q\Z)^\times, B \in \Z/Q\Z),
$$
the complete exponential sum defined in the introduction,
as well as 
\begin{align} \label{eq:gqExp}
  g_{q'}(r) &:= e_{q'}(rV_1)\sum_{d \in I(V_1)} e_{q'}(-rd) = \sum_{1 \leq d \leq V_1/V_0} e_{q'}(-rd) \ll \min\{V_1/V_0,q'/r\}.
\end{align}
We would like to use partial summation to remove the weight $g_{q'}$ in \eqref{eq:withGq}; however, the long sum in $r$ will make this inefficient in the sequel if we do not split the interval into shorter segments. To this end, let $1 \leq K \leq q'-1$ be a parameter that we will choose later. We split
\begin{align*}
  S_{q',a}(k';V_1) = \frac{1}{q'} \sum_{1 \leq m \leq q'/K} \sum_{K(m-1) < r \leq Km} e_{q'}(-rV_1) K_2(k'a,r;q') g_{q'}(r).
\end{align*}
Given $1 \leq m \leq q'/K$, set
$$
\kappa(m; k'a, q') := \max_{1 \leq R \leq K} \left|\sum_{r = K(m-1)+1}^{K(m-1) + R} e_{q'}(-rV_1)K_2(k'a,r;q')\right|.
$$
Applying partial summation to estimate the derivative $g_{q'}'$ of $g_{q'}$, we get
$$
|g_{q'}(r+1)-g_{q'}(r)| = \left|\int_r^{r+1} g'_{q'}(t) dt\right| \leq \max_{r \leq t < r+1} |g'_{q'}(t)|\ll \frac{V_1}{q'V_0}\min\{V_1/V_0,q'/r\}.
$$ 
Combining this, \eqref{eq:gqExp} and partial summation once again, we obtain
\begin{align}
  S_{q',a}(k';V_1) &\ll \frac{1}{q'}\sum_{1 \leq m \leq q'/K} \kappa(m; k'a, q')\cdot K \max_{(m-1)K < r \leq mK} |g_{q'}(r+1)-g_{q'}(r)| \nonumber \\
                   &\ll \frac{V_1}{q'V_0}\sum_{1 \leq m \leq q'/K}\frac{\kappa(m; k'a, q')}{m} + K\left(\frac{V_1}{q'V_0}\right)^2|K_2(k'a,0;q')|. \label{eq:SqaBd}
\end{align}
We may control the terms in \eqref{eq:fsum} with large $f$ directly using a square-root cancelling bound for $K_2(A,B,Q)$, which we derive below.

\subsubsection{Point-wise bounds}

\begin{lem}[$p$-adic stationary phase method]\label{lem:expK2}
For any $n \geq 2$, $a \in (\mb{Z}/p^n\mb{Z})^{\times}$ and $b \in \mb{Z}/p^n\mb{Z}$, we have
\begin{align}\label{eq:expK2}
K_2(a,b;p^{n}) = \left(\frac{3a}{p^n}\right) \e_{p,n}p^{n/2} \asum_{ \substack{y \pmod{p^{\llf n/2\rrf}} \\ y^3 \equiv \bar{2a}b \pmod{p^{\llf n/2\rrf}}}} e_{p^{n}}(3ay^2),
\end{align}
where $\e_{p,n} = 1$ if $n$ is even and $\e_{p,n} = i^{(p-1)^2/4}$ if $n$ is odd. 
\end{lem}
\begin{proof}
Applying Lemmas 12.2 and 12.3 of \cite{IK}, we have that if $A \in (\mb{Z}/p^n\mb{Z})^{\times}$ and $B \in \mb{Z}/p^n\mb{Z}$ then
\begin{align*}
K_2(A,B;p^n) = p^{n/2} \asum_{ \substack{y \pmod{p^{\llf n/2\rrf}} \\ y^3 \equiv \bar{2A}B \pmod{p^{\llf n/2\rrf}}}} e_{p^n}(Ay^2 + B\bar{y}) \theta_{p^n}(y;A,B),
\end{align*}
where for $y$ satisfying $y^3 \equiv \bar{2A}B \pmod{p^{\llf n/2\rrf}}$ we have set
$$
\theta_{p^n}(y;A,B) := \begin{cases} 1 &\text{ if $2 \mid n$} \\ p^{-1/2}\sum_{z \pmod{p}} e_p \left(3Az^2 + z\left((2Ay-B\bar{y}^2)/p^{(n-1)/2}\right)\right) &\text{ if $2 \nmid n$.} \end{cases}
$$
A key point here is that the set of critical points, i.e., solutions to $y^3 \equiv \bar{2A}B \pmod{p^{\llf n/2\rrf}}$ is invariant under translations by $p^{n-\llf n/2\rrf} \mb{Z}/p^n\mb{Z}$ (see e.g., \cite[Lemma 1]{MilZhan}); in particular, by choosing a lift of such critical points to solutions to $y^3 \equiv \bar{2A}B \pmod{p^n}$ via Hensel's lemma, we may rewrite 
$$
Ay^2 + B\bar{y} \equiv 3A y^2 \pmod{p^n}.
$$
When $n = 2m$, with $m \geq 1$, we simply have (with $A = a$ and $B = b$)
$$
K_2(a,b;p^{2m}) = p^{m} \asum_{ \substack{y \pmod{p^{m}} \\ y^3 \equiv \bar{2a}b \pmod{p^{m}}}} e_{p^{2m}}(3ay^2),
$$
which implies the claim in this case. On the other hand, completing the square and using the explicit computation of Gauss sums modulo $p$ when $n = 2m+1$, we get
\begin{align*}
\theta_{p^{2m+1}}(y;A,B) &= p^{-1/2} \sum_{z \pmod{p}} e_p\left(3A\left(z + \bar{6A}\frac{2Ay-B\bar{y}^2}{p^{m}}\right)^2 - \bar{12A} \frac{(2Ay-B\bar{y}^2)^2}{p^{2m}}\right) \\
&= \left(\frac{3A}{p}\right) i^{(p-1)^2/4} e_{p^{2m+1}}\left(-\bar{3}Ay^2(1-\bar{2A}B\bar{y}^3)^2\right).
\end{align*}
Again, using the $p^{m+1}\mb{Z}/p^{2m+1}\mb{Z}$-translation invariance of the solutions to $y^3 \equiv \bar{2A}B \pmod{p^m}$, at critical points the exponential here is simply 1, and we obtain (when $A = a$ and $B = b$)
$$
K_2(a,b;p^{2m+1})  = \left(\frac{3a}{p}\right) i^{(p-1)^2/4} p^{(2m+1)/2} \asum_{ \substack{y \pmod{p^{m}} \\ y^3 \equiv \bar{2a}b \pmod{p^{m}}}} e_{p^{2m+1}}(3ay^2).
$$
The claim is proved.
\end{proof}

\begin{lem}\label{lem:WeilK2}
  Let $Q \geq 2$. Then
  \[\max_{\substack{A\in(\Z/Q\mb{Z})^\times \\ B\in \Z/Q\mb{Z}}} |K_2(A,B;Q)| \ll_{\e} Q^{1/2+\e}.\]  
\end{lem}
\begin{proof}
We reduce to the case of prime power moduli using the Chinese remainder theorem. Indeed, if $Q = m_1m_2$ where $m_1$ and $m_2$ are coprime, then selecting $r_1,r_2 \in \mb{Z}$ such that $m_1r_1+m_2r_2 = 1$, we have
\begin{align} \label{eq:factK2}
K_2(A,B;Q) = K_2(r_2A,r_2B;m_1)K_2(r_1A,r_1B;m_2).
\end{align}
Applying this inductively over the prime power divisors of $Q$ and taking maxima over $A$ and $B$, we obtain
$$
\max_{\substack{A\in(\Z/Q\mb{Z})^\times \\ B\in \Z/Q\mb{Z}}} |K_2(A,B;Q)| \leq \prod_{p^n || Q} \max_{\substack{A\in(\mb{Z}/p^n\mb{Z})^\times \\ B\in \mb{Z}/p^n\mb{Z}}} |K_2(A,B;p^n)|.
$$
Now, observe that for each $p || Q$ prime and $A,B \in \mb{Z}/p\mb{Z}$ we can represent
$$
K_2(A,B;p) = \sum_{x \in \F_p^{\times}} e_p\left(A\bar{x}^2 + Bx\right) = \sum_{x \in \F_p^{\times}} e_p(f_{A,B}(x)),
$$
where $f_{A,B}(x) = A/x^2 + Bx\in\F_p(x)$. By the Bombieri--Dwork--Weil bound \cite[Theorem 6]{Bomb66} (see also \cite[Section 3.5]{DelEC}),
\begin{equation}
  \label{eq:DworkBombieriWeil}
  \max_{A\in\F_p^\times, B\in\F_p} |K_2(A,B;p)| \leq  C'\sqrt{p}.
\end{equation}
for $C'>0$ an absolute constant. On the other hand, if $p^n||Q$ with $n \geq 2$ then by Lemma \ref{lem:expK2} we have
$$
|K_2(A,B;p^n)| \leq p^{n/2}|\{y \pmod{p^{\llf n/2 \rrf}} : y^3 \equiv \bar{2A}B \pmod{p^{\llf n/2\rrf}}\}|.
$$
The solutions to $y^3 \equiv \bar{2A}B \pmod{p}$ lift uniquely to solutions modulo $p^{\llf n/2\rrf}$ by Hensel's lemma, so the cardinality in the previous line is $\leq 3$. In particular, 
$$
\max_{\substack{A \in (\mb{Z}/p^n\mb{Z})^{\times} \\ B \in \mb{Z}/p^n\mb{Z}}} |K_2(A,B;p^n)| \leq Cp^{n/2}
$$ for $C := \max\{C',3\}$. \\
We therefore conclude that
\[\max_{\substack{A\in(\Z/Q\Z)^\times \\ B\in \Z/Q\Z}} |K_2(A,B;Q)| \leq C^{\omega(Q)} Q^{1/2} \ll_{\e} Q^{1/2+\e}\]
as claimed.
\end{proof}

Combining Lemma \ref{lem:WeilK2} with \eqref{eq:withGq} and $V_1 \ll X^{1/2}$ gives
\begin{align*}
  S_{q',a}(k';V_1) &\ll_{\e} (\log q)\max_{r \pmod*{q'} \atop r\neq 0} |K_2(k'a,r;q')| + \frac{|g_{q'}(0)||K_2(k'a,0;q')|}{q'}
                   &\ll_{\e} (q')^{1/2+\e} + X^{\e}(X/q')^{1/2}V_0^{-1}.  
\end{align*}
Let $Z \geq 2$ be a parameter to be chosen later. Applying this with $q' = \tilde q/f$ with $f > Z$ in particular, it follows immediately that
\begin{align*}
  \sum_{f\mid q \atop f > Z} \frac{1}{f}\quad \sum_{1 \leq k \leq \tilde{q}/(2f)} \frac{|S_{\tilde{q}/f,a}(k';V_1)|}{k'} &\ll_{\e} \tilde{q}^{\e} \sum_{f\mid \tilde{q} \atop f > Z}\frac{1}{f} \left((\tilde{q}/f)^{1/2+\e} + X^{\e}(X/\tilde{q})^{1/2}f^{1/2}V_0^{-1}\right)\\
  &\ll X^{\e}\left(\tilde{q}^{1/2} Z^{-3/2} + Z^{-1/2}(X/\tilde{q})^{1/2}V_0^{-1}\right),
\end{align*}
which, in combination with \eqref{eq:SqaBd}, thus gives
\begin{align} 
\Delta_{\mu^2}(X;q,a) &\ll_{\e} X^{\e} \frac{V_0V_1}{\tilde{q}}\sum_{\substack{f\mid\tilde{q} \\ f \leq Z}}\quad \sum_{1\leq k' \leq \tilde{q}/(2f)} \sum_{m = 1}^{\tilde{q}/(fK)} \frac{\kappa(m; k'a,\tilde{q}/f)}{k'm} \label{eq:impWeil}\\
&+ V_0\left(1+X^{\e}\left(\frac{X}{Z\tilde{q}}\right)^{1/2}\right) + X^{\e}\left(\frac{V_0^2\tilde{q}^{1/2}}{Z^{3/2}}+ \frac{X}{qV_0} + \frac{X}{\tilde{q}V_1}+ K V_1^2 \left(\frac{Z}{\tilde{q}}\right)^{3/2}\right), \nonumber
\end{align}
provided $K \leq \tilde{q}/Z$.

\subsubsection{Bounds on average} \label{subsubsec:BdsAvg}
It turns out (see Section \ref{subsec:3of4}) that merely applying Lemma \ref{lem:WeilK2} directly to $K_2$ (after replacing $q$ by $\tilde{q}$ as we have done above) results in a power-saving upper bound for $\Delta_{\mu^2}(X;q,a)$ for any $q \leq X^{3/4-\e}$, provided $\tilde{q}$ can be chosen with appropriate size. This is independent of squarefreeness considerations, and is modeled after Hooley's idea in \cite{Hoo75}. \\
Ignoring the effect of the sum over divisors $f$ (which will have little influence in the sequel), in order to do better we will need to find a power savings in $X$ over the pointwise estimate from Lemma \ref{lem:WeilK2}, i.e.,
$$
\max_{\substack{1 \leq m \leq \tilde{q}/K-1}}\quad \max_{\substack{1 \leq k' \leq \tilde{q}-1}} |\kappa(m;k'a,\tilde{q})| \leq K \max_{A \in (\Z/\tilde{q}\Z)^{\times} \atop B \in \Z/\tilde{q}\Z} |K_2(A,B;\tilde{q})| \ll_{\e} K\tilde{q}^{1/2+\e},
$$
by utilizing the averaging of $K_2$ sums implicit in the definition of $\kappa(m;k'a,\tilde{q})$. To this end, we will apply the $q$-van der Corput method of Heath-Brown, as formulated by Irving in \cite{Irv}. The following is proved \emph{mutatis mutandis} by the arguments in \cite{Irv}.
\begin{prop}[{\cite[Lemma 4.3]{Irv}}] \label{prop:qvdC}
Let $K,L \geq 1$, and $M \in \Z$. Suppose $Q \geq 4$ factors as $Q = Q_0 \cdots Q_L$, with $Q_j \geq 2$ for each $0 \leq j \leq L$. Let $J$ be an interval with $|J| \leq K$, and set
$$
T(b,M) := \sum_{k \in J} e_Q(-Mk) K_2(b,k;Q),
$$
where $b$ is a coprime residue class modulo $Q$. If $K \geq \max\{Q_1,\ldots,Q_L\}$ then
$$
|T(b,M)| \ll_{\e,L} Q^{1/2+\e}K\left(\sum_{j = 1}^L \left(\frac{Q_{L-j+1}}{K}\right)^{2^{L-j}} + \frac{Q}{K^{(L+1)}Q_0^{2^{L-1}+1}} \sum_{1 \leq |h_1| \leq K/Q_1} \cdots \sum_{1 \leq |h_L| \leq K/Q_L} |T(\mbf{h})|\right)^{2^{-L}}
$$
where for each $\mbf{h} \in \Z^L$ with $1 \leq |h_j| \leq K/Q_j$ there is an interval $J(\mbf{h})$ of size $\leq K$ and $b'$ coprime to $q$ such that
$$
T(\mbf{h}) := \sum_{k \in J(\mbf{h})} \prod_{I \subseteq \{1,\ldots,L\}} \mc{C}^{|I|} K_2\left(b',k+\sum_{i \in I} Q_ih_i,Q_0\right),
$$
$\mc{C}(z) := \bar{z}$ being the complex conjugation map.
\end{prop}

Proposition \ref{prop:qvdC} indicates that our main point of focus for the remainder of the argument will be the estimation of $|T(\mbf{h})|$, for $\mbf{h}$ satisfying $Q_j|h_j| \in [1,K]$ for all $1 \leq j \leq L$. We will estimate these terms pointwise in $\mbf{h}$, the key point being that, outside of a sparse set of $\mbf{h}$, we will obtain significant cancellation. This will result in the following. 
\begin{prop} \label{prop:ThEst}
Adopt the notation of Proposition \ref{prop:qvdC}. \\
i) Assume that $Q = Q_0 \cdots Q_L$ is squarefree. Then
$$
\sum_{1 \leq |h_1| \leq K/Q_1} \cdots \sum_{1 \leq |h_L| \leq K/Q_L} |T(\mbf{h})| \ll_{\e,L}\frac{K^L}{Q} Q_0^{2^{L-1}+3/2+\e}\left(\frac{K}{Q_0} + 1\right). 
$$
ii) Then there is a $\delta' = \delta'(L) \in (0,2^{-2^L}]$ such that the following holds. Suppose $Q = Q_0\cdots Q_L$ is such that $(Q_i,Q_j) = 1$ for all $0 \leq i < j \leq L$, and $(Q_0,6) = 1$. Assume also that $K/Q_j > Q_0^{2\delta'}$ for all $p^{\nu} || Q_0$ and all $1 \leq j \leq L$. Then
$$
\sum_{1 \leq |h_1| \leq K/Q_1} \cdots \sum_{1 \leq |h_L| \leq K/Q_L} |T(\mbf{h})| \ll_{\e,L}\frac{K^L}{Q} Q_0^{2^{L-1}+2-\delta'+\e}.
$$
\end{prop}

Proposition \ref{prop:ThEst} will be proved in the next two sections. As a first step, we shall replace $T(\mbf{h})$ by analogous complete sums \eqref{eq:T2} modulo prime powers with an additional additive phase using the following.
\begin{lem} \label{lem:toPrime}
  Let $I \subset \mb{Z}/Q_0 \mb{Z}$ be an interval, and let $A_0 \in (\mb{Z}/Q_0\mb{Z})^{\times}$, $B_0\in\Z/Q_0\Z$. Let $N,M \geq 0$ with $N+M \geq 1$, and let $\mbf h\in\Z^N$, $\mbf h'\in\Z^M$. Write $Q_0 = \prod_{1 \leq j \leq k} p_j^{\alpha_j}$ for distinct primes $p_j$ and $k = \omega(Q_0)$.
  Then there exist $\mbf{A},\mbf{B}\in\prod_{1\le j\le k} \mb{Z}/p_j^{\alpha_j}\mb{Z}$ such that
\begin{align*}
&\sum_{B \in I} \prod_{1 \leq i \leq N} K_2(A_0,B_0+h_i;Q_0) \prod_{1 \leq j \leq M} \bar{K}_2(A,B+h_j';Q_0) \\
&\ll \sum_{C \pmod*{Q_0}} \min\left\{\frac{|I|}{Q_0},\frac{1}{Q_0\|C/Q_0\|}\right\} \prod_{j= 1}^{k} \left|T(A_j,B_j,C,\mbf h,\mbf h';p_j^{\alpha_j})\right|,
\end{align*}
where for each $Q\mid Q_0$ we set
\begin{multline}
  \label{eq:T2}
  T(A,B,C, \mbf h,\mbf h';Q):=\sum_{b \pmod*{Q}} e_{Q}\left(CBb\right)\\
   \cdot \prod_{1 \leq i \leq N}K_2(A,b+h_{i};Q) \prod_{1 \leq j\leq M} \bar{K}_2(A,b+h'_{j};Q).
\end{multline}
\end{lem}
\begin{proof}
For $B \in \mb{Z}/Q_0\mb{Z}$ put
$$
g(B) := \prod_{1 \leq i \leq N} K_2(A,B+h_i;Q_0) \prod_{1 \leq j \leq M} \bar{K_2}(A,B+h_j';Q_0).
$$
Completing the sum over $B$ modulo $Q_0$, the left-hand side is
\begin{align*}
&\sum_{B \pmod*{Q_0}} 1_I(B) g(B) = \frac{1}{Q_0} \sum_{C \pmod{Q_0}} \left(\sum_{D \in I} e_{Q_0}(-CD)\right) \left(\sum_{B \pmod*{Q_0}} g(B)e_{Q_0}(CB)\right) \\
&\ll \frac{1}{Q_0}\sum_{C \pmod{Q_0}} \min\{|I|,\|C/Q_0\|^{-1}\} \left|T(A,B,C,\mbf h,\mbf h';Q_0)\right|.
\end{align*}
It thus suffices to show the existence of $A_j,B_j \pmod{p_j^{\alpha_j}}$ for each $1 \leq j \leq k$, such that
\[T(A,B,C,\mbf h,\mbf h';Q_0)=\prod_{1\le j\le k}T(A_j,B_j,C,\mbf h,\mbf h';p_j^{\alpha_j}).\]
We prove this by induction on $k$, the number of distinct prime factors of $Q_0$. When $k = 1$ there is nothing to prove. Assume this works for any $Q_0$ having $k$ distinct prime factors, and now suppose that $Q_0$ has $k+1$ such factors. Write $Q_0 = Q_0'p^{\alpha}$, where $p\nmid Q_0'$ and $\alpha \geq 1$. Let $r,s \in \mb{Z}$ be chosen such that $rQ_0' + sp^{\alpha} = 1$. By the Chinese remainder theorem, every $B \pmod{Q_0}$ can be written uniquely as $B = uQ_0' + vp^{\alpha}$, where $0 \leq u \leq p^{\alpha}-1$ and $0 \leq v \leq Q_0'-1$. Thus,
\begin{multline*}
  T(A,B,C,\mbf h,\mbf h';Q_0)=\sum_{u \pmod*{p^{\alpha}}} e_{p^{\alpha}}\left(Cu\right) \sum_{v \pmod*{Q_0'}} e_{Q_0'}\left(Cv\right)\\
 \cdot \prod_{1 \leq i \leq N}K_2(A,uQ_0'+vp^{\alpha}+h_{i};p^{\alpha}Q_0') \prod_{1 \leq j\leq M} \bar{K}_2(A,uQ_0'+vp^{\alpha}+h'_{j};p^{\alpha}Q_0').
\end{multline*}
Applying \eqref{eq:factK2} and the symmetry $K_2(\alpha,\gamma \beta; q') = K_2(\gamma^2\alpha,\beta;q')$ for any $\gamma \in (\mb{Z}/q'\mb{Z})^{\times}$, we see that the products are
\begin{multline*}
\left(\prod_{i = 1}^N K_2(s^3A,uQ_0' + h_i;p^{\alpha}) \prod_{j = 1}^M \bar{K}_2(s^3A,uQ_0'+h_j';p^{\alpha})\right) \\
\cdot \left(\prod_{i = 1}^N K_2(r^3A,vp^{\alpha} + h_i;Q_0') \prod_{j = 1}^M \bar{K}_2(r^3A,vp^{\alpha}+h_j'; Q_0')\right).
\end{multline*}
Making the change of variables $u \mapsto uQ_0'$ and $v \mapsto vp^{\alpha}$, we thus get
\[T(A,B,C,\mbf h,\mbf h';Q_0)=T(s^3A,\bar Q_0',C,\mbf h,\mbf h';p^{\alpha})\cdot T(r^3A,\bar p^{\alpha},C,\mbf h,\mbf h';Q_0'),\]
where the inverses are taken modulo $p^{\alpha}$ and $Q_0'$ respectively.

By induction, we can factor the second bracketed term similarly into $k$ products, and the claim follows.
\end{proof}

In Irving's work \cite{Irv}, where squarefree moduli specifically are treated, an application of a (variant of a) result of Fouvry, Ganguly, Kowalski and Michel on correlations of Kloosterman sums (\cite[Proposition 3.2]{FGKM14}) is used to control the complete sums to prime moduli arising in the factorization in Lemma \ref{lem:toPrime}. To achieve the same goal, we will prove in full detail a similar result for correlations of the complete exponential sums $K_2$ in the next section. In Section \ref{sec:corPrimePow}, we treat the same problem for prime power moduli, using rather different techniques.

\section{Correlations of $K_2$ Sums to Prime Moduli: Cohomological Methods} \label{sec:corPrime}
The goal of this section is to provide an estimate for sums like \eqref{eq:T2}. We will prove, in full detail, the following analogue of \cite[Proposition 3.2]{FGKM14} and \cite[Section 4.3]{Irv} (the latter of which cites private communications for the corresponding result for Kloosterman sums).
This will be the main input to the proof of Proposition \ref{prop:ThEst} i).\\
%
Throughout this section, fix a prime $p$. Given $N,M \geq 1$, and tuples $\mbf{h} \in \mb{F}_p^N$ and $\mbf{h}' \in \mb{F}_p^M$ we define 
$$
T = T_{\mbf{h},\mbf{h}'} := \{h_1,\ldots,h_N,h_1',\ldots,h_M'\},
$$
and for each $\tau \in T$ define
\begin{align*}
\mu(\tau) &= \mu_{\mbf{h}}(\tau) := |\{1 \leq j \leq N : h_j = \tau\}| \\
\nu(\tau) &= \nu_{\mbf{h}'}(\tau) := |\{1 \leq j \leq M : h_j' = \tau\}|.
\end{align*}
\begin{thm}\label{thm:boundCorrCompl}
  For $A\in\F_p^\times$, $\psi\in\widehat\F_p$ a possibly trivial additive character, and $h_1,\ldots,h_N$,  $h_1',\ldots, h_M' \in \F_p$ (where $N+M\ge 1$), we have
  \begin{equation}
    \label{eq:boundCorrCompl}
    \sum_{B \in \F_p} \psi(B)\prod_{i=1}^N K_2(A,B+h_i;p)\prod_{j=1}^M \bar K_2(A,B+h'_j;p)\ll (N+M)3^{N+M}p^{\frac{N+M+1}{2}}
  \end{equation}
unless $\psi$ is trivial and $\mu(\tau) \equiv \nu(\tau) \pmod{3}$ for all $\tau \in T$. 
The implied constant is absolute (e.g. does not depend on $N$ or $M$).
\end{thm}
\begin{rem}
We can rewrite the left-hand side of \eqref{eq:boundCorrCompl} as
$$
\sum_{B \in \mb{F}_p} \psi(B) \prod_{\tau \in T} K_2(A,B+\tau;p)^{\mu(\tau)} \bar{K}_2(A,B+\tau;p)^{\nu(\tau)}.
$$
Note that when $\psi$ is trivial and $\mu(\tau) = \nu(\tau)$ for all $\tau \in T$, the summands are all non-negative. We should therefore not expect to find any improvement over the trivial bound in general (outside of the possibility that the absolute values of $p^{-1/2}K_2(A,B+\tau;p)$ are small, which is atypical for large $p$ given that these normalized sums become equidistributed according to Haar measure on $\text{SU}_3(\mb{C})$ as $B$ varies modulo $p$ and $p \ra \infty$). 
It will transpire from the proof of Theorem \ref{thm:boundCorrCompl} that our characterization of when savings are available for the above correlation sums is essentially sharp.
\end{rem}

To prove Theorem \ref{thm:boundCorrCompl} we will employ the method surveyed in \cite{FKMSumProducts}. Highlights of this method include:
\begin{itemize}
\item the interpretation of $K_2$ as the trace function of an $\ell$-adic sheaf, pure of weight 1, and the computation of its geometric monodromy group (following Katz \cite{KatzESDE});
\item the determination of geometric isomorphisms between shifts of these sheaves to find the monodromy of a product sheaf via the Goursat--Kolchin--Ribet criterion; and
\item the application of the Grothendieck--Lefschetz trace formula and the Grothendieck--Ogg--Shafarevich Euler--Poincar\'{e} formula.
\end{itemize}
We will use the statements from \cite{PGGaussDistr16} that also track the dependencies in the number of factors (i.e., $N$ and $M$), for completeness.
For the sake of concision, we refer the reader to these (and the references therein) for additional definitions, notations and other details.

\subsection{Cohomological Interpretation of $K_2$ and \eqref{eq:boundCorrCompl}}

In the language of \cite{DelEC}, the pointwise bound on $K_2$ modulo primes from Lemma \ref{lem:WeilK2} (see \eqref{eq:DworkBombieriWeil}) can be seen as the outcome of applying the Grothendieck--Lefschetz trace formula and Deligne's Riemann hypothesis over finite fields \cite{Del2} to the one-dimensional Artin--Schreier $\ell$-adic sheaf $\Lc_{e(f_{A,B}/p)}$ on $\G_m\times\F_p$:
\begin{align*}
  \left|\sum_{x\in\F_p^\times} e(f_{A,B}(x)/p)\right|&=\left|\sum_{x\in\F_p^\times} \iota \tr \left(\Frob_{x, p}\mid \Lc_{e(f_{A,B}/p)}\right)\right|\\
                                                     &=\left|\sum_{i=0}^2 (-1)^i\iota\tr\left(\Frob_p\mid H^i_c(\G_m\times\overline\F_p, \Lc_{e(f_{A,B}/p)})\right)\right|\\
                                        &\le \sqrt{p}\cdot \dim H^1_c\left(\G_m\times\overline\F_p, \Lc_{e(f_{A,B}/p)}\right),
\end{align*}
where $\ell\neq p$ is an auxiliary prime, $\iota: \overline\Q_\ell\to\C$ is a compatible\footnote{By this we mean that $\iota$ must take $\ell$-adic additive characters $\lambda : \mb{F}_p \ra \overline{\mb{Q}}_{\ell}$, implicit in the definition of the $\ell$-adic Fourier transform defined below, to usual additive characters $\iota \circ \lambda : \mb{F}_p \ra \mb{C}$, e.g., $\iota(\lambda(x)) := e(x/p)$ for each $x \in \mb{F}_p$.} embedding, $\Frob_{x, p}\in \pi_1^\geom(\G_m,\overline\eta)$ is the geometric Frobenius class at $x\in\G_m(\F_p)$ for $\overline\eta$ a geometric generic point, and $\Frob_p\in\Gal(\overline\F_p/\F_p)$ is the geometric Frobenius (acting on the cohomology groups). The dimension of the first cohomology group is bounded independently from $p$ by the Grothendieck--Ogg--Shafarevich formula, giving the constant $C'$ in \eqref{eq:DworkBombieriWeil} above.

To handle \emph{sums of} $K_2$ sums in Theorem \ref{thm:boundCorrCompl}, we adopt a different perspective, and consider $K_2$ as an $\ell$-adic trace function itself, using the $\ell$-adic Fourier transform of Deligne. Some properties of the sheaf thus produced are given in the following lemma.

\begin{lem}\label{lem:Gc}
  Assume that $p$ is an odd prime, fix an auxiliary prime $\ell\neq p$, a compatible embedding $\iota: \overline\Q_\ell\to\C$, and let $A\in\F_p^\times$. There exists a middle-extension $\overline\Q_\ell$-sheaf of $\overline\Q_\ell$-modules $\Gc_A$ that is lisse on $\G_m\times\F_p$ and pointwise pure of weight $0$, as well as $\gamma \in\C$ of modulus one, such that for every $B\in\G_m(\F_p)$,
  \begin{equation}
    \label{eq:Gtf}
    \iota\tr\left(\Frob_{B,p}\mid \Gc_A\right)=\frac{-\gamma K_2(A,B;p)}{\sqrt{p}}.
  \end{equation}
  Moreover:
  \begin{enumerate}
  \item $\Gc_A$ has rank 3, with a unique $\infty$-break at $2/3$;
  \item\label{item:SwanG} $\Swan_\infty(\Gc_A)=2$;
  \item\label{item:tameG} $\Gc_A$ is tamely ramified at $0$, and its local monodromy is a unipotent pseudoreflection;
  \item\label{item:monodromy} if $p > 7$, the arithmetic and geometric monodromy groups $G_\arith$, $G_\geom$ (that is, the Zariski closures of the images of the arithmetic/geometric fundamental groups by the corresponding representation composed with $\iota$) are equal and isomorphic to $\SL_3(\C)$.
  \end{enumerate}
\end{lem}
\begin{proof}
  A sheaf $\Gc_A'$ satisfying all of the above properties but \eqref{item:monodromy} is given by the $\ell$-adic Fourier transform of the Artin--Schreier sheaf $\Lc_{e(f(X)/p)}$, where $f(X)=A/X^2\in\F_p(X)$. The resulting trace function (the left-hand side of \eqref{eq:Gtf}) is indeed the discrete Fourier transform of $e(f(x)/p)$, namely $K_2$. The statements of the lemma are all contained in \cite[Section 7.3, Section 7.12 ($\SL$-Example(3)), Theorem 7.12.3.1]{KatzESDE}, with $h=f$, $\alpha=0$ and $n_0=\ord_0(f)=2$. In particular, the rank is $1+n_0=3$.

  Regarding the last property, since $f(x)+f(-x)$ is nonconstant the geometric monodromy group is isomorphic to $\SL_3(\C)$ for $p > 7$ by the theorem from ibid. By \cite[Section 7.12 ($\SL$-Example(3))]{KatzESDE}, $\det\Gc'_A$ is geometrically trivial, whence arithmetically isomorphic to $\beta\otimes\overline\Q_\ell$, where $\beta$ is a $p$-Weil number of weight $0$. Therefore, setting $\Gc_A=\beta^{-1/3}\otimes\Gc'_A$ yields
  \[\SL_3(\C)=G_\geom(\Gc_A)=G_\geom(\Gc'_A)\le G_\arith(\Gc_A)\le\SL_3(\C),\]
  with $\Gc_A$ still verifying the previous properties. The theorem then holds with the twist $\gamma=\iota\beta^{-1/3}\in\C$.
\end{proof}
Note that the purity and dimensionality statements encompass the bound \eqref{eq:DworkBombieriWeil} from Lemma \ref{lem:WeilK2}:
\begin{align} \label{eq:WeilV2}
|K_2(A,B;p)|\le \dim(\Gc_A)\sqrt{p}=3\sqrt{p}
\end{align}
(including at $B=0$, when $K_2(A,0;p)$ is, up to a unimodular factor, quadratic Gauss sum with modulus exactly $\sqrt{p}$).

The bound in \eqref{eq:boundCorrCompl} is trivial if $p \leq 7$, so we may assume in what follows that $p > 7$. Using Lemma \ref{lem:Gc}, we may bound the modulus of the left-hand side of \eqref{eq:boundCorrCompl} as
\begin{align}\label{eq:firstCoh}
\ll p^{\frac{N+M}{2}}\left(\left|\sum_{B\in U(\F_p)} \iota \tr\left(\Frob_{B,p}\mid \Hc_{A,\mbf{h},\psi}\right)\right|+3^{N+M}(N+M)\right),
\end{align}
setting
\begin{align*}
  \Hc_{A,\mbf{h},\psi}:=\Lc_\psi\otimes \left(\bigotimes_{\tau \in T_{\mbf{h},\mbf{h}'}} [+\tau]^*\Gc_A^{\otimes \mu(\tau)}\otimes [+\tau]^*D(\Gc_A)^{\otimes \nu(\tau)} \right),
\end{align*}
and letting $U(\mb{F}_p)$ denote its set of lisse points. Here, the tensor product $\otimes$ and dual $D$ are understood at the level of the corresponding $\ell$-adic representations (as in \cite{KatzGKM,KatzESDE}). The second term in brackets in \eqref{eq:firstCoh} arises by applying \eqref{eq:WeilV2} to the set of ramification points of the tensor product, of which there are $\leq N+M$. \\
The trace formula will now be applied to \eqref{eq:firstCoh}.

\subsection{Applying the trace formula}

By the Grothendieck--Lefschetz trace formula and the Grothendieck--Ogg--Shafarevich formula (see the references in \cite[Theorem 2.5]{PGGaussDistr16}),
\begin{equation}
  \label{eq:tfHc}
  \sum_{B\in U(\F_p)} \tr\left(\Frob_{B,p}\mid \Hc_{A,\mbf{h},\psi}\right)=p\cdot \tr \left(\Frob_p \mid (\Hc_{A,\mbf{h},\psi})_{\pi_{1,p}^\geom(U, \overline\eta)}\right) + O \left(E(\Hc_{A,\mbf{h},\psi})\sqrt{p}\right)
\end{equation}
with an absolute implied constant, where
\begin{align*}
  E(\Hc_{A,\mbf{h},\psi})&=\rank(\Hc_{A,\mbf{h},\psi}) \left(|\Sing(\Hc_{A,\mbf{h},\psi})| + \sum_{x\in \Sing(\Hc_{A,\mbf{h},\psi})} \Swan_x(\Hc_{A,\mbf{h},\psi})\right)\\
                 &\le 3^{N+M}\left(N+M+(N+M)\cdot 2+1\right)\le 4(N+M) 3^{N+M},
\end{align*}
the inequality coming from Lemma \ref{lem:Gc} and the ramification properties of the tensor product recalled in \cite[Lemma 1.3]{KatzGKM}. Let
\[\Hc^+_{A,\mbf{h},\psi}=\Lc_\psi\oplus\left(\bigoplus_{\tau \in T_{\mbf{h},\mbf{h}'}} [+\tau]^*\Gc_A\right)\]
and note that
\begin{align*}
  G_\geom \left(\Hc^+_{A,\mbf{h},\psi}\right)&\le G_\arith \left(\Hc^+_{A,\mbf{h},\psi}\right)\ni\Frob_p\\
                                                         &\le G_\geom(\Lc_{\psi})\times G_\geom(\Gc_A)^{H}
                                                           =\begin{cases}
                                                             \Z/p\times\SL_3(\C)^H : \psi\neq 0\\
                                                             \SL_3(\C)^H: \psi = 0,
                                                             \end{cases}
\end{align*}
where $H :=|T_{\mbf{h},\mbf{h}'}|$. A very convenient case arises when these inclusions are equalities: in this event $\Frob_p$ acts trivially on the coinvariant space in \eqref{eq:tfHc}, and the right-hand side becomes
\begin{align*}
  &p\cdot \dim(\Hc_{A,\mbf{h},\psi})_{G} + O \left(E(\Hc_{A,\mbf{h},\psi})\sqrt{p}\right)\\
  &=\delta_{\psi=0}\prod_{\tau \in T}\mult_1 \left(\Std^{\otimes \mu(\tau)}\otimes D(\Std)^{\otimes \nu(\tau)}\right) + O((N+M)3^{N+M}p^{1/2}),
\end{align*}
by Schur's Lemma (see \cite[Lemma 4.6, Proposition 4.8]{PGThe}), where $\mult_1$ is the multiplicity of the trivial representation on a representation $\rho$. 
Finally, the multiplicities are zero if, and only if, $3|(\mu(\tau) - \nu(\tau))$ for all $\tau \in T_{\mbf{h},\mbf{h}'}$: this is precisely the content of \cite[Prop. 4.5(2)]{KowRic}, in the case $N = 3$ (in the notation there, $\mult_1\left(\Std^{\otimes \mu(\tau)} \otimes D(\Std)^{\otimes \nu(\tau)}\right) = A_{\nu(\tau),\mu(\tau)}$, for each $\tau$).

Therefore, it only remains to prove:
\begin{lem}
  The arithmetic and geometric monodromy groups of $\Hc_{A,\mbf{h},\psi}$ coincide and are isomorphic to $\Z/p\times\SL_3(\C)^H$ if $\psi$ is non trivial, and $\SL_3(\C)^H$ otherwise, where $H = |T_{\mbf{h},\mbf{h}'}|$.
\end{lem}
\begin{proof}
  Let us first consider the case where $\psi$ is trivial. By the Goursat--Kolchin--Ribet criterion \cite[Section 1.8]{KatzESDE}, since the pair $(\SL_3(\C),\Std)_{i={1,2}}$ is Goursat-adapted (in the language given there), it suffices to show that there exists no geometric isomorphism of the form
  \[[+h]^*\Gc_A\cong [+h']^*\Gc_A\otimes\Lc,\quad h\neq h',\]
  for a one-dimensional sheaf $\Lc$. Without loss of generality, $h'=0$. Let us assume by contradiction that $h\neq 0$. Since the left-hand side is ramified exactly at $-h$ and $\infty$, $\Lc$ must be ramified at $0$ and $-h$ (and possibly at $\infty$). It follows that
  \[\Gc_A^{I_0}\cong([+h]^*\Gc_A)^{I_{-h}}\cong (\Gc_A\otimes\Lc)^{I_{-h}}\cong\Gc_A\otimes\Lc^{I_{-h}}=0.\]
  However, this contradicts Lemma \ref{lem:Gc}(3), which states that the stalk at $0$ of $\Gc_A$ is a rank 2 
  pseudoreflection.


  If $\psi$ is non-trivial, let us write $\Hc^+_{A,\mbf{h}, \psi}=\Lc_\psi\oplus \Fc$, where we may assume by the above that the arithmetic and geometric monodromy groups of $\Fc$ coincide and are isomorphic to $\SL_3(\C)^H$. We have
  \[G_\geom \left(\Hc^+_{A,\mbf{h}, \psi}\right)\le G_\arith \left(\Hc^+_{A,\mbf{h}, \psi}\right)\le \F_p\times\SL_3(\C)^H,\]
  and both $G_\geom$ surjects onto both $\F_p$ and $\SL_3(\C)^H$. The surjection onto $\F_p$ implies that $G_\geom$ has at least $p$ connected components, of which the component at the identity contains $\SL_3(\C)^H$. It follows that $G_\geom = \F_p \times \SL_3(\C)^H$, as claimed.  
\end{proof}
This finishes altogether the proof of Theorem \ref{thm:boundCorrCompl}. \qed

\section{Correlations of $K_2$-sums to Prime Power Moduli: Stationary Phase Methods} \label{sec:corPrimePow}
\noindent In this section, we complement the results of the previous section, which applied to correlations of $K_2$ sums to prime moduli, with a treatment of $K_2$ sums to higher prime power moduli. \\
Before stating Theorem \ref{thm:expsumPrimPow} below, we need a few elements of notation.
Fix $n \geq 2$ and $p > 3$. Let $N,M \geq 0$ with $K := N + M \geq 1$, and suppose $\mbf{h} \in \mb{Z}^N$ and $\mbf{h}' \in \mb{Z}^M$. Similarly to the previous section, given $\tau \in \mb{Z}/p^n\mb{Z}$ we define 
\begin{align*}
\mu(\tau) = \mu_{\mbf{h}}(\tau) &:= |\{1 \leq j \leq N : h_j \equiv \tau \pmod{p^n}\}| \\
\nu(\tau) = \nu_{\mbf{h}'}(\tau) &:= |\{1 \leq j \leq M : h_j' \equiv \tau \pmod{p^n}\}|, 
\end{align*}
and we define 
$$
T = T_{\mu,\nu} := \{\tau \in \mb{Z}/p^n\mb{Z} : \mu(\tau) + \nu(\tau) \geq 1\}.
$$ 
We will prove the following estimates.
\begin{thm}\label{thm:expsumPrimPow}
Let $p > 3$ be prime and let $n \geq 2$. Let $N,M \geq 0$ with $N+M \geq 1$, and let $\mbf{h} \in \mb{Z}^N$ and $\mbf{h}' \in \mb{Z}^M$. Let $c \in \mb{Z}/p^n\mb{Z}$, $a \in (\mb{Z}/p^n\mb{Z})^{\times}$, and let $\mu = \mu_{\mbf{h}}$, $\nu = \nu_{\mbf{h}'}$ and $T = T_{\mbf{h},\mbf{h}'}$ be defined as above. Further, put
$$
\rho = 1_{p \leq 3|T|/2-1} + \left\lceil \frac{\log(20(N+M)^3)}{\log p} \right \rceil,
$$
and if $n > (N+M)^32^{N+M}$, assume in addition that
$$
\min\{|\tau-\tau'|_p : \tau,\tau' \in T, \tau \neq \tau'\} \geq p^{-2|T|^{-2}(\llf 2^{-|T|}n\rrf - \rho)},
$$
where $|x|_p$ denotes the $p$-adic absolute value of $x \in \mb{Q}$. Then
\begin{align*}
&\sum_{b \pmod{p^n}} e_{p^n}(cb) \prod_{1 \leq i \leq N} K_2(a,b+h_i; p^n) \prod_{1 \leq j \leq M} \bar{K}_2(a,b+h_j';p^n) \\
&\ll_{N,M,\e} p^{(N+M+2)n/2} 
\begin{cases} 
p^{-\frac{1}{n^2(n-1)^2}+\e} &\text{ if $n \leq (N+M)^32^{N+M}$} \\
p^{-n2^{-N-M}+1} 
&\text{ if $n > (N+M)^32^{N+M}$,}
\end{cases} 
\end{align*}
unless either: 
\begin{enumerate}[(i)]
\item $p \equiv 2 \pmod{3}$, $c = 0$ and $\mu(\tau) = \nu(\tau)$ for all $\tau \in T$,
\item $p \equiv 1 \pmod{3}$, $c = 0$ and $3|(\mu(\tau)-\nu(\tau))$ for all $\tau \in T$, or 
\item $p \equiv 1 \pmod{3}$, $p^{n-1}||c$ and $3|(\mu(\tau)-\nu(\tau))$; in this case, the bound is $O_{N,M}(p^{(N+M+2)n/2-1/2})$.
\end{enumerate}
\end{thm}
\begin{rem}
The condition $3|(\mu(\tau)-\nu(\tau))$ in (ii) of the above statement is the same condition that arose in connection with the trivial bound in Theorem \ref{thm:boundCorrCompl}. This condition, along with (i) and (iii), arise in this context for the following reason. Modulo prime powers $p^n$ with $n \geq 2$, a complete sum $K_2(a,b;p^n)$ may be explicitly evaluated as a sum over the set of critical points of a certain $p$-adic phase function. The correlations of the $K_2$ sums therefore expands as a linear combination of exponential sums modulo $p^n$, each with a fixed frequency. The trivial bound cannot be improved whenever one of these frequencies is zero, and this occurs precisely in the degenerate cases listed in the statement of Theorem \ref{thm:expsumPrimPow} (this is treated, in Proposition \ref{prop:epsZero} below). 
\end{rem}

The remainder of the section is devoted to proving this theorem.
\subsection{Preparation}
Fix $a \in (\mb{Z}/p^n\mb{Z})^{\times}$, $c \in \mb{Z}/p^n\mb{Z}$, $\mbf{h}\in \mb{Z}^N$ and $\mbf{h}' \in \mb{Z}^M$, and define $\mu,\nu$ and $T$ as above. Set
$$
S_{p^n}(\mbf{h},\mbf{h}';c,a) := \sum_{b \pmod{p^n}} e_{p^n}(cb) \prod_{1 \leq i \leq N} K_2(a,b+h_i; p^n) \prod_{1 \leq j \leq M} \bar{K}_2(a,b+h_j';p^n).
$$
We can rewrite this expression as 
\begin{align*}
S_{p^n}(\mbf{h},\mbf{h}';c,a) &= \sum_{d \pmod{p}}  \sum_{b \pmod{p^n} \atop b \equiv d \pmod{p}} e_{p^n}(cb)\prod_{\tau \in T} K_2(a,b+\tau;p^n)^{\mu(\tau)} \bar{K}_2(a,b+\tau;p^n)^{\nu(\tau)} \\
&=: \sum_{d \pmod{p}} \mc{S}_{p^n}(c,\mu,\nu;a,d).
\end{align*}

We deduce from Lemma \ref{lem:expK2} that unless $4a^2(b+\tau) \equiv 4a^2(d+\tau)  \in (\mb{Z}/p\mb{Z})^{\times 3}$ (i.e., the cube of a residue class prime to $p$) for all $\tau \in T$ we get that $\mc{S}_{p^n}(c,\mu,\nu;a,d) = 0$. We will henceforth assume the condition
$$
(\Diamond) : \ \ 4a^2 (d+\tau) \in (\mb{Z}/p\mb{Z})^{\times 3} \text{ for all $\tau \in T$,}
$$
which, as $T$ and $a$ are fixed, depends only on $d$. \\
Now, depending on whether $p \equiv 1 \pmod{3}$ or $p \equiv 2 \pmod{3}$, if $d$ satisfies $(\Diamond)$ then there are either:
\begin{itemize}
\item exactly 3 critical points, when $p \equiv 1 \pmod{3}$, or 
\item exactly one critical point, when $p\equiv 2 \pmod{3}$. 
\end{itemize}
Letting $u_0$ be a primitive cube root modulo $p^{\llf n/2\rrf}$ (by Hensel's lemma, this is necessarily a lift of a primitive cube root modulo $p$), we may define a fixed branch of the cube root $r \mapsto s(r)$ such that $s(r)^3 \equiv r \pmod{p}$ whenever $r \in (\mb{Z}/p\mb{Z})^{\times 3}$, and lift this branch to a branch of cube root modulo $p^{\llf n/2\rrf}$ as well; then, if $x^3 \equiv r \pmod{p^{\llf n/2\rrf}}$ we obtain that 
$$
x \equiv s(r)u_0^j \pmod{p^m}  \text{ for some } 0 \leq j \leq d_p-1, 
$$
setting $d_p = 1$ if $p \equiv 2 \pmod{3}$ and $d_p = 3$ if $p \equiv 1 \pmod{3}$. 
In this way, we can write (in the notation of Lemma \ref{lem:expK2})
\begin{align*}
K_2(a,b+\tau;p^{n}) = \left(\frac{3a}{p^n}\right) \e_{p,n}p^{n/2} \sum_{0 \leq j \leq d_p-1} e_{p^{n}}(3as(\bar{2a}(b+\tau))^2u_0^{j}).
\end{align*}
Fix $d$ satisfying $(\Diamond)$. For ease of notation, in the sequel we will write 
$$
\mc{S}^{\ast}_{p^n}(c,\mu,\nu;a,d) = \left(\frac{3a}{p^n}\right)^{N+M} \bar{\e_{p,n}}^{N+M}\mc{S}_{p^n}(c,\mu,\nu;a,d),
$$
and study $\mc{S}_{p^n}^{\ast}$ as $d$ varies. \\
For each $\tau \in T$ set 
\begin{align*}
U(\tau) := \{0,\ldots,d_p-1\}^{\mu(\tau)}, \quad \quad \quad \quad V(\tau) := \{0,\ldots,d_p-1\}^{\nu(\tau)},
\end{align*}
as well as
\begin{align*}
\mc{U} := \prod_{\tau \in T} U(\tau),  \quad \quad \quad \quad
\mc{V} := \prod_{\tau \in T} V(\tau).
\end{align*}
With these notations, we can write
\begin{align*}
&\mc{S}^{\ast}_{p^{n}}(c,\mu,\nu,a,d) \\
&= p^{(N+M)n/2} \sum_{\substack{\mbf{j} = (\mbf{j}(\tau))_{\tau \in T} \in \mc{U} \\ \mbf{j}' = (\mbf{j}'(\tau))_{\tau \in T} \in \mc{V}}} \sum_{b \pmod{p^{n}} \atop b \equiv d \pmod{p}}e_{p^{n}}\left(bc + 3a \sum_{\tau \in T} s(\bar{2a}(b+\tau))^2\left(\sum_{1 \leq i \leq \mu(\tau)} u_0^{j_i(\tau)} - \sum_{1 \leq i' \leq \nu(\tau)} u_0^{j'_{i'}(\tau)}\right)\right).
\end{align*}
For notational convenience, we simplify the above expression further as follows. Given $\mbf{\eps} = (\eps_{\tau})_{\tau \in T} \in (\mb{Z}/p^{n}\mb{Z})^{|T|}$ and $b \in \mb{Z}/p^n\mb{Z}$, define
\begin{align}\label{eq:fTeDef}
f_{T,\mbf{\eps}}(b) := bc + 3a\sum_{\tau \in T} \eps_{\tau} s(\bar{2a}(b+\tau))^2.
\end{align}
Provided $d$ satisfies $(\Diamond)$, we may thus write
\begin{align*}
\mc{S}^{\ast}_{p^n}(c,\mu,\nu;a,d) = p^{(N+M)n/2} \sum_{\mbf{\eps} \in (\mb{Z}/p^{n} \mb{Z})^{|T|}} \phi(\mbf{\eps}) \sum_{b \pmod{p^{n}} \atop b \equiv d \pmod{p}} e_{p^{n}}(f_{T,\mbf{\eps}}(b)),
\end{align*}
where we have set
$$
\phi(\mbf{\eps}) := \left|\left\{(\mbf{j},\mbf{j}') \in \mc{U} \times \mc{V} : \eps_{\tau} = \sum_i u_0^{j_i(\tau)} - \sum_{i'} u_0^{j'_{i'}(\tau)} \text{ for all $\tau \in T$} \right\}\right| \geq 0,
$$
for each $\mbf{\eps} \in (\mb{Z}/p^n\mb{Z})^{|T|}$. Note that 
$$
\sum_{\mbf{\eps} \in (\mb{Z}/p^n \mb{Z})^{|T|}} \phi(\mbf{\eps}) = |\mc{U}||\mc{V}| \ll 3^{N+M},
$$
and if $p \equiv 2 \pmod{3}$ then $\phi(\mbf{\eps}) = 1$ if $\eps_{\tau} = \mu(\tau)-\nu(\tau)$ for all $\tau \in T$, and 0 otherwise. \\
We now separate the $\mbf{\eps} = \mbf{0}$ term from the remaining $\mbf{\eps}$, so that\footnote{We use the notation $\sideset{}{^{\Diamond}}\sum_{d \pmod{p}}$ to denote a sum over $d \pmod{p}$ satisfying $(\Diamond)$ (for $a$ and $T$ fixed).}
\begin{align*}
\sideset{}{^{\Diamond}}\sum_{d \pmod{p}}\mc{S}^{\ast}_{p^n}(c,\mu,\nu;a,d) &= \sideset{}{^{\Diamond}}\sum_{d \pmod{p}} \mc{S}^{\ast, \neq \mbf{0}}_{p^n}(c,\mu,\nu;a,d) + \phi(\mbf{0}) p^{(N+M)n/2}  \sideset{}{^{\Diamond}}\sum_{d \pmod{p}} \sum_{b \pmod{p^n} \atop b \equiv d \pmod{p}} e_{p^n}(f_{T,\mbf{0}}(b)) \\
&=: \sideset{}{^{\Diamond}}\sum_{d \pmod{p}} \mc{S}^{\ast, \neq \mbf{0}}_{p^n}(c,\mu,\nu;a,d) + \sideset{}{^{\Diamond}}\sum_{d \pmod{p}} \mc{S}^{\ast, \mbf{0}}_{p^n}(c,\mu,\nu;a,d).
\end{align*}

\subsection{The $\mbf{\eps} = \mbf{0}$ Terms}
The contribution to $S_{p^n}(\mbf{h},\mbf{h}';c,a)$ from $\mbf{\eps} = \mbf{0}$ can be estimated as follows.
\begin{prop}\label{prop:epsZero}
With the above notation, we have
\begin{align*}
\sideset{}{^{\Diamond}}\sum_{d \pmod{p}} \mc{S}^{\ast, \mbf{0}}_{p^n}(c,\mu,\nu;a,d) = 0
\end{align*}
unless $p^{n-1}|c$ and $3|(\mu(\tau)-\nu(\tau))$ for all $\tau \in T$. In this latter case, the following non-trivial bounds hold: 
\begin{align*}
\sideset{}{^{\Diamond}}\sum_{d \pmod{p}} \mc{S}^{\ast, \mbf{0}}_{p^n}(c,\mu,\nu;a,d) \ll_{N,M} p^{(N+M+2)n/2} \cdot
\begin{cases} 
p^{-1/2} &: p \equiv 1 \pmod*{3}, c \neq 0, 3|(\mu(\tau)-\nu(\tau)) \forall \tau \in T \\
p^{-1} &: p \equiv 2 \pmod*{3}, c \neq 0, \mu(\tau) = \nu(\tau) \forall \tau \in T.
\end{cases}
\end{align*}
%
%
\end{prop}
To prove this result, we need several lemmas.
\begin{lem} \label{lem:pval}
Let $p \equiv 1 \pmod{3}$ be prime, and let $n \geq 2$ and $D \geq 1$. Let $\alpha_1,\alpha_2,\alpha_3 \in \mb{Z}$ with $\max\{|\alpha_1|,|\alpha_2|,|\alpha_3|\} \leq D$. Assume that $p^{n/2} > 20D^3$ and let $1 \leq u_0 \leq p^{\llf n/2 \rrf}-1$ be a primitive cube root modulo $p^{\llf n/2 \rrf}$. Then either $\alpha_1 = \alpha_2 = \alpha_3$ or else
$$
\nu_p(\alpha_1 + \alpha_2u_0 + \alpha_3 u_0^2) \leq \left \lceil \frac{\log(20D^3)}{\log p}\right \rceil.
$$
\end{lem}
\begin{proof}
Set $r := \left\lceil\frac{\log(20D^3)}{\log p}\right\rceil \leq n/2$. Assume the claim is false, so that $p^r|(\alpha_1+\alpha_2u_0 + \alpha_3 u_0^2)$ and $\alpha_1,\alpha_2,\alpha_3$ are not all the same. 
As $p \equiv 1 \pmod{3}$ and $u_0$ is a primitive cube root, we must have $u_0 \not\equiv 1 \pmod{p^{\llf n/2\rrf}}$. By assumption at least one of $\alpha_1,\alpha_2,\alpha_3$ does not vanish, and up to multiplication by $u_0^j$ for $j \in \{0,1,2\}$ (which does not change the $p$-adic valuation), we may suppose $\alpha_3$ does not. 
As $u_0^3 \equiv 1 \pmod{p^r}$, we have
\begin{align*}
\alpha_1 + \alpha_2 u_0 + \alpha_3 u_0^2 &= (\alpha_1-\alpha_3) + (\alpha_2-\alpha_3)u_0 + \alpha_3(1+u_0+u_0^2) \\
&\equiv (\alpha_1-\alpha_3) + (\alpha_2-\alpha_3)u_0 \pmod{p^r}.
\end{align*}
Now, on one hand this implies that 
\begin{equation}\label{eq:firstu0}
\alpha_1-\alpha_3 \equiv -u_0(\alpha_2-\alpha_3) \pmod{p^r}.
\end{equation}
On the other, taking cubes, we obtain that 
$$
(\alpha_1-\alpha_3)^3 + (\alpha_2-\alpha_3)^3 \equiv 0 \pmod{p^r}.
$$ 
However, as
$$
|(\alpha_1-\alpha_3)^3 + (\alpha_2-\alpha_3)^3| \leq 2\cdot (2D)^3 < 20D^3 \leq p^r,
$$ 
it follows that $(\alpha_1-\alpha_3)^3 = -(\alpha_2-\alpha_3)^3$, and therefore that $\alpha_1-\alpha_3 = -(\alpha_2-\alpha_3)$. Plugging this into the earlier congruence implies that either $\alpha_1-\alpha_3 = \alpha_2-\alpha_3 = 0$, or else $u_0 \equiv 1 \pmod{p^r}$. Since the $\alpha_j$ are not all the same by assumption, the first of these is impossible. To see that the second fails as well, note that if $u_0 \not \equiv 1 \pmod{p^{\llf n/2 \rrf}}$ but $u_0 = 1 +mp^r$ for some $m \in \mb{Z}$, then of course $p^{\llf n/2 \rrf -r} \nmid m$. On the other hand, since $u_0$ is a cube root of unity modulo $p^{\llf n/2\rrf}$,
$$
u_0^3 = (1+mp^r)^3 = 1+mp^r(3+3p^rm+m^2p^{2r}) \equiv 1 \pmod{p^{\llf n/2\rrf}},
$$
which is only possible for $p \neq 3$ if $\nu_p(m) \geq \llf n/2 \rrf - r$, a contradiction. The claim follows.
\end{proof}

\begin{lem} \label{lem:zeroTerm}
Let $T \subseteq \mb{Z}/p^{n}\mb{Z}$, and suppose $\phi(\mbf{0}) \neq 0$. Assume furthermore that $p^n > 20(N+M)^3$.\\
a) If $p\equiv 2 \pmod{3}$ then $\mu(\tau) = \nu(\tau)$ for all $\tau \in T$.\\
b) If $p \equiv 1 \pmod{3}$ then $\mu(\tau) \equiv \nu(\tau) \pmod{3}$ for all $\tau \in T$. 
\end{lem}
\begin{proof}
In each case, $\phi(\mbf{0}) \neq 0$ if, and only if, there exist tuples $\mbf{j} \in \mc{U}$, $\mbf{j}' \in \mc{V}$ such that 
$$
0 = \sum_{1 \leq i \leq \mu(\tau)} u_0^{j_i(\tau)} - \sum_{1 \leq i' \leq \nu(\tau)} u_0^{j'_{i'}(\tau)} \text{ for all $\tau \in T$.}
$$
a) When $p \equiv 2 \pmod{3}$, we may (trivially) take $u_0 = 1$ and $j_i(\tau) = j'_{i'}(\tau) = 0$ for all $i,i'$ and $\tau$, which leads immediately to the conclusion $\mu(\tau) = \nu(\tau)$ for all $\tau \in T$.\\
b) When $p \equiv 1 \pmod{3}$, we write the above expression as
$$
0 = \left(m_{\tau} + n_{\tau} u_0 + p_{\tau}u_0^2\right) - \left(m_{\tau}' + n_{\tau}'u_0 + p_{\tau}'u_0^2\right) = (m_{\tau} - m_{\tau}') + (n_{\tau}-n_{\tau}')u_0 + (p_{\tau} - p_{\tau}')u_0^2,
$$
where $m_{\tau},n_{\tau},p_{\tau}$ and $m_{\tau}',n_{\tau}',p_{\tau}'$ are non-negative integers satisying 
\begin{align*}
&m_{\tau} + n_{\tau}+p_{\tau} = \mu(\tau)
&m_{\tau}' + n_{\tau}'+p_{\tau}' = \nu(\tau),
\end{align*}
for all $\tau \in T$. By Lemma \ref{lem:pval}, it follows that $m_{\tau} - m_{\tau}' = n_{\tau} - n_{\tau}' = p_{\tau} - p_{\tau}' = \ell$, say. We conclude that
$$
\mu(\tau)-\nu(\tau) = \left(m_{\tau} + n_{\tau} + p_{\tau}\right) - \left(m_{\tau}' + n_{\tau}' + p_{\tau}'\right) = 3\ell,
$$
which implies the claim.
\end{proof}
\begin{lem} \label{lem:DiamBd}
Let $p > 3$ be prime, and let $C \in \mb{Z}/p\mb{Z}$. Let $A \in (\mb{Z}/p\mb{Z})^{\times}$ and let $\tilde{T} \subseteq \mb{Z}/p\mb{Z}$. \\
a) If $p \equiv 2 \pmod{3}$ then 
$$
\sum_{\substack{d \pmod{p} \\ 4A^2(d+\tau) \in (\mb{Z}/p\mb{Z})^{\times 3} \\ \forall \tau \in \tilde{T}}} e_p(dC) = p1_{C \equiv 0 \pmod{p}} + O(|\tilde{T}|).
$$
b) If $p \equiv 1 \pmod{3}$ then 
$$
\sum_{\substack{d \pmod{p} \\ 4A^2(d+\tau) \in (\mb{Z}/p\mb{Z})^{\times 3} \\ \forall \tau \in \tilde{T}}}e_p(dC) = 3^{-|\tilde{T}|}p1_{C \equiv 0 \pmod{p}} + O\left(|\tilde{T}|^2p^{1/2}\right).
$$
\end{lem}
\begin{proof}
a) When $p \equiv 2 \pmod{3}$ every residue class modulo $p$ is a cube. Thus, if $d$ satisfies the condition in the sum on the left-hand side then as $p > 3$ this is equivalent to $\prod_{\tau \in \tilde{T}} (d+\tau) \in (\mb{Z}/p\mb{Z})^{\times}$, which is satisfied for all but $O(|\tilde{T}|)$ residue classes $d$ modulo $p$. Thus, the sum in question is simply
$$
\sum_{d \pmod{p}} e_p(dC) + O(|\tilde{T}|) = p1_{C \equiv 0 \pmod{p}} + O(|\tilde{T}|),
$$
as required. \\
b) Let $\Xi_3(p) := \{\chi \pmod{p} : \chi^3 = \chi_0\}$, where $\chi_0$ denotes the trivial multiplicative character modulo $p$. Since the set of multiplicative characters modulo $p$ is a cyclic group of order $p-1$ and $3|(p-1)$, we can write
$$
\Xi_3(p) = \{\xi^j : j \in \{-1,0,1\}\}, \text{ where } \xi := \chi_1^{(p-1)/3}
$$
for some fixed generator $\chi_1$ for the group of characters mod $p$.
We note that for any $b \in \mb{Z}/p\mb{Z}$,
$$
\sum_{-1 \leq j \leq 1} \xi^j(b) = \begin{cases} 3 &\text{ if $b \in (\mb{Z}/p\mb{Z})^{\times 3}$} \\ 0 &\text{ otherwise}\end{cases},
$$
and so the exponential sum in question is
\begin{align} \label{eq:allJs}
&\sum_{d \pmod{p}} e_p(dC) \prod_{\tau \in \tilde{T}} 1_{4A^2(d+\tau) \in (\mb{Z}/p\mb{Z})^{\times 3}} \nonumber \\
&= 3^{-|\tilde{T}|} \sum_{\mbf{j} \in \{-1,0,1\}^{|\tilde{T}|}} \xi(4A^2)^{t(\mbf{j})} \sum_{d \pmod{p}} \xi\left(\prod_{\tau \in \tilde{T}} (d+\tau)^{j_{\tau}}\right)e_p(dC),
\end{align}
where we have written $t(\mbf{j}) := \sum_{\tau \in \tilde{T}} j_{\tau}$. \\
When $\mbf{j} = \mbf{0}$ we get
\begin{align}\label{eq:Jzero}
&3^{-|\tilde{T}|} \sum_{d \pmod{p}} \chi_0\left(\prod_{\tau \in \tilde{T}} (d+\tau)\right) e_p(dC) \nonumber\\
&= 3^{-|\tilde{T}|} \sum_{d \pmod{p}} e_p(dC) + O(|\tilde{T}|3^{-|\tilde{T}|}) = 3^{-|\tilde{T}|}\left(p1_{C \equiv 0 \pmod{p}} + O(|\tilde{T}|)\right). 
\end{align}
For $\mbf{j} \neq \mbf{0}$, we define
\begin{align*}
g_{\mbf{j}}(d) := \prod_{\tau \in \tilde{T} \atop j_{\tau} = 1} (d+\tau),  \ \ \ \ 
h_{\mbf{j}}(d) &:= \prod_{\tau \in \tilde{T} \atop j_{\tau} = -1} (d+\tau),
\end{align*}
and what remains to be estimated is the expression
$$
3^{-|\tilde{T}|} \sum_{\substack{\mbf{j} \in \{-1,0,1\}^{|\tilde{T}|} \\ \mbf{j} \neq \mbf{0}}} \sum_{\substack{d \pmod{p} \\ p\nmid (d+\tau) \forall \tau \in T}} \xi(g_{\mbf{j}}(d) \bar{h}_{\mbf{j}}(d)) e_p(dC) \ll \max_{\mbf{j} \neq \mbf{0}} \left|\sum_{\substack{d \pmod{p} \\ p\nmid (d+\tau) \forall \tau \in T}} \xi(g_{\mbf{j}}(d) \bar{h}_{\mbf{j}}(d)) e_p(dC)\right|.
$$
We interpret the latter sum as a sum over $\mb{F}_p$, and estimate the sum using the language of $\ell$-adic trace functions, for an auxiliary prime $\ell \neq p$. For each $\mbf{j} \neq \mbf{0}$ let $\mc{K}_{\mbf{j}}$ denote the Kummer sheaf attached to $\xi \circ (g_{\mbf{j}}/h_{\mbf{j}})$, and let $\mc{L}_C$ be the Artin--Schreier sheaf attached $d \mapsto e_p(dC)$. Then the sum we must estimate is
$$
\sum_{d \in U_{\mbf{j}}(\mb{F}_p)} t_{\mc{K}_{\mbf{j}}}(d) t_{\mc{L}_C}(d) = \sum_{d \in U_{\mbf{j}}(\mb{F}_p)} t_{\mc{F}}(d),
$$
where $\mc{F} := \mc{K}_{\mbf{j}} \otimes \mc{L}_C$ and $U_{\mbf{j}}(\mb{F}_p) := \{d \in \mb{F}_p : d \neq -\tau \ \forall \tau \in T\}$. By Corollary 2.31 of \cite{PGThe}, this is 
$$
p1_{\mc{F} \text{ geom. trivial}} + O(\text{cond}(\mc{F})^2p^{1/2}).
$$
When $\mbf{j} \neq \mbf{0}$ we claim that $\mc{F}$ is not geometrically trivial, i.e., that $\mc{K}_{\mbf{j}}$ and $D(\mc{L}_{C})$ are not geometrically isomorphic. Indeed, on one hand if $C \neq 0$ then $\mc{L}_C$ (and thus $D(\mc{L}_C)$) has a lone wild ramification point at $\infty$ (with $\text{Swan}_{\infty}(\mc{L}_C) = 1$), while in contrast $\mc{K}_{\mbf{j}}$ has only tame ramification points at the zeros of $g_{\mbf{j}}h_{\mbf{j}}$ (and in particular, $\text{Swan}_{\infty}(\mc{K}_{\mbf{j}}) = 0$). If $C = 0$ then because $g_{\mbf{j}}(d)h_{\mbf{j}}(d)$ has only distinct roots and is non-constant, it is not the cube of a polynomial. Thus, $\mc{K}_{\mbf{j}}$ is ramified in at least one point and hence not geometrically trivial in this case as well. Finally, 
$$
\text{cond}(\mc{F}) = \text{cond}(\mc{K}_{\mbf{j}}) \text{cond}(\mc{L}_C) \leq 3(\deg{g_{\mbf{j}}} + \deg{h_{\mbf{j}}}+1) \ll |\tilde{T}|,
$$
so that indeed 
$$
\sum_{\substack{d \pmod{p} \\ p \nmid g_{\mbf{j}}(d)h_{\mbf{j}}(d) \\ p\nmid (d+\tau) \forall \tau \in T}} \xi(g_{\mbf{j}}(d) \bar{h}_{\mbf{j}}(d)) e_p(dC) \ll |\tilde{T}|^2p^{1/2},
$$
for all $\mbf{j} \neq \mbf{0}$. Inserting this and \eqref{eq:Jzero} into \eqref{eq:allJs} and summing over $\mbf{j}$, we reach the claim.
\end{proof}
\begin{proof}[Proof of Proposition \ref{prop:epsZero}]
Suppose $\mbf{\eps} = \mbf{0}$. Since $f_{T,\mbf{0}}(b) = cb$, we have
\begin{align*}
\sideset{}{^{\Diamond}}\sum_{d \pmod{p}} \mc{S}^{\ast, \mbf{0}}_{p^n}(c,\mu,\nu;a,d) &= \phi(\mbf{0}) p^{(N+M)n/2} \sideset{}{^{\Diamond}}\sum_{d \pmod{p}} \sum_{t \pmod{p^{n-1}}} e_{p^{n}}(c(pt + d)) \\
&= \phi(\mbf{0})p^{(N+M+2)n/2-1} 1_{p^{n-1}|c} \sideset{}{^{\Diamond}}\sum_{d \pmod{p}} e_p(dc'),
\end{align*}
where $c' = 0$ if $p^{n-1} \nmid c$, and otherwise $c \equiv c'p^{n-1} \pmod{p^{n}}$. By Lemma \ref{lem:DiamBd}, we have
$$
\sideset{}{^{\Diamond}}\sum_{d \pmod{p}} e_p(dc') =
3^{-|\tilde{T}|1_{p\equiv 1 \pmod{3}}} p1_{c' \equiv 0 \pmod{p}} + O_{|\tilde{T}|}(p^{1/2}1_{p \equiv 1 \pmod{3}} + 1_{p \equiv 2 \pmod{3}}),
$$
where $\tilde{T} := \{\tau \pmod{p} : \tau \in T\}$. As $|\tilde{T}| \leq |T| \leq N+M$,
the above expression is thus
$$
\ll_{N,M} p^{(N+M+2)n/2} \left(1_{c \equiv 0 \pmod{p^{n}}} + 1_{p^{n-1}|c}\left(p^{-1/2} 1_{p \equiv 1 \pmod{3}} + p^{-1}1_{c = 0}1_{p \equiv 2 \pmod{3}}\right)\right).
$$
By Lemma \ref{lem:zeroTerm}, this gives
\begin{align*}
\sideset{}{^{\Diamond}}\sum_{d \pmod{p}} \mc{S}^{\ast, \mbf{0}}_{p^n}(c,\mu,\nu;a,d) &\ll_{N,M} p^{(N+M+2)n/2} 1_{p \equiv 1 \pmod{3}} \left(1_{c \equiv 0 \pmod{p^{n}}} + p^{-1/2}1_{p^{n-1}|c}\right)  \prod_{\tau \in T} 1_{3|(\mu(\tau)-\nu(\tau))} \\
&+ p^{(N+M+2)n/2} 1_{p \equiv 2 \pmod{3}}\left( 1_{c \equiv 0 \pmod{p^{n}}} + p^{-1}1_{p^{n-1}|c}\right) \prod_{\tau \in T} 1_{\mu(\tau) = \nu(\tau)},
\end{align*}
which implies the claim.
\end{proof}
In the next two subsections, we treat the estimation of 
$$
\sideset{}{^{\Diamond}}\sum_{d \pmod{p}} \mc{S}_{p^n}^{\ast,\neq \mbf{0}}(c,\mu,\nu;a,d)
$$
in two ways: first, we will prove an estimate that is efficient when $n$ is large, and subsequently a different estimate that is most efficient when $n$ is not large (but $p$ is). 

\subsection{Bounds for Large $n$}
The main result of this subsection is the following.
\begin{prop}\label{prop:largeN}
Let $p>3$ be prime. Assume that $n \geq 2^{N+M}(N+M)^3$. Then we have
$$
\sideset{}{^{\Diamond}} \sum_{d \pmod{p}} \mc{S}_{p^n}^{\ast, \neq \mbf{0}}(c,\mu,\nu;a,d) \ll_{N,M} p^{(N+M+2-2\delta_1)n/2}
$$
with $\delta_1 := 2^{-N-M}$, unless there are at least two distinct $\tau,\tau' \in T$ such that
$$
\tau \equiv \tau' \pmod{p^{r_p(n)}},
$$
where we define
$$
r_p(n) := \llf \frac{2}{(N+M)(N+M-1)} \left(\llf 2^{-N-M} n\rrf-1_{p \leq 3(N+M)/2-1}- \left\lceil \frac{\log(20(N+M)^3)}{\log p} \right\rceil\right)\rrf.
$$
\end{prop}
The above proposition will be proved using a method of Mili\'{c}evi\'{c} and Zhang \cite[Section 4]{MilZhan}, a consequence of which is the following proposition concerning exponential sums with argument function $f_{T,\mbf{\eps}}$ (see \eqref{eq:fTeDef} for the definition).
\begin{prop} \label{prop:ftepsbd}
There exist constants $\delta_i = \delta_i(|T|) > 0$ for $i = 1,2,3$, and $\rho = \rho(N,M,|T|) > 0$ such that the following holds. \\
For any positive integer $n \geq 2$ such that $p^{n/2} > 20(N+M)^3$, any $d \in \mb{Z}/p\mb{Z}$ and any non-zero $\mbf{\eps} \in (\mb{Z}/p^{n}\mb{Z})^{|T|}$ with $\phi(\mbf{\eps}) \neq 0$, either: 
\begin{itemize} 
\item the estimate
$$
\sum_{b \in \mb{Z}/p^{n}\mb{Z} \atop b \equiv d \pmod{p}} e_{p^{n}}(f_{T,\mbf{\eps}}(b)) \ll_{|T|} p^{(1-\delta_1)n}
$$ 
holds, or else
\item there are at least two distinct $\tau,\tau' \in T$ such that $\eps_{\tau},\eps_{\tau'} \not \equiv 0 \pmod{p^{n}}$ and 
$$
\tau \equiv \tau' \pmod{p^{\llf \delta_3(\llf \delta_2n\rrf - \rho)\rrf}}.
$$
\end{itemize}
In particular, the values $\delta_1 = \delta_2 = 2^{-|T|}$, $\delta_3 = \binom{|T|}{2}^{-1}$ and 
$$
\rho = 1_{p \leq 3|T|/2-1} + \left \lceil \frac{\log(20(N+M)^3)}{\log p}\right \rceil,
$$
are admissible.
\end{prop}

We need the following simple lemma about $p$-adic valuations of generalized binomial coefficients.
\begin{lem} \label{lem:nuPbinom}
Let $p > 3$. Then for any $k \geq 1$, 
$$
0 \leq \max_{0 \leq j \leq k} \nu_p\left(\binom{2/3}{j}\right) \leq 1_{p \leq 3k/2-1}.
$$
\end{lem}
\begin{proof}
Notice that 
$$
\binom{2/3}{j} = (j!)^{-1}\prod_{0 \leq l \leq j-1}(2/3-l) = \frac{(-1)^{j-1}2}{3^{j}j!} \prod_{1 \leq l \leq j-1}(3l-2),
$$
and so as $p > 3$,
$$
\nu_p\left(\binom{2/3}{j}\right) = \nu_p\left(\prod_{1 \leq l \leq j-1} (3l-2)\right) - \nu_p(j!).
$$
Note that 
$$
\max_{0 \leq j \leq k} \nu_p\left(\binom{2/3}{j}\right) \geq \nu_p\left(\binom{2/3}{0}\right) \geq 0.
$$ 
By Legendre's formula for $p$-adic valuations of factorials, it suffices to show that
$$
\nu_p\left(\prod_{1 \leq l \leq j-1} (3l-2)\right) \leq \sum_{r \geq 1} \llf \frac{j}{p^r} \rrf + 1_{p \leq 3j/2-1},
$$
for each $1 \leq j \leq k$. To check this inequality, we observe first that if $1 \leq l_0 \leq j$ is minimal such that $p|(3l_0-2)$ and $l > l_0$ also has this property then $p|(l-l_0)$. Clearly, by minimality we must have $\nu_p(3l_0-2) = 1$: if $p \equiv 1 \pmod{3}$ then $l_0$ satisfies $p = 3l_0-2$; if $p \equiv 2 \pmod{3}$ then $2p = 3l_0-2$. In particular, if $p > 3j/2-1$ then $l_0$ does not exist. \\
Consider next when $p \leq 3j/2-1$. Then we have
$$
\nu_p\left(\prod_{1 \leq l \leq j-1}(3l-2)\right) = \sum_{r \geq 1} |\{1 \leq l \leq j-1 : l \equiv 2\bar{3} \pmod{p^r}\}| \leq \sum_{r \geq 1} \llf \frac{j-a_r}{p^r}\rrf,
$$
where $0 \leq a_r \leq p^r-1$ is a minimal representative of the residue class $2\bar{3} \pmod{p^r}$. We clearly have $\llf (j-a_r)/p^r\rrf \leq \llf j/p^r \rrf$, so that, summing over $r$, we obtain the desired bound. 
%
%
%
\end{proof}

The following further observations will be key. If $b \in (\mb{Z}/p^n\mb{Z})^{\times}$, $\kappa \in \mb{N}$ and $t \in \mb{Z}/p^n \mb{Z}$, we note that
\begin{align*}
s(\bar{2a}(b+\tau+p^{\kappa}t))^2 &= s(\bar{2a}(b+\tau)(1+p^{\kappa}t/(b+\tau)))^2 \\
&= s(\bar{2a}(b+\tau))^2 \sum_{l \geq 0} \binom{2/3}{l} (p^{\kappa} t)^l (b+\tau)^{-l} = \sum_{l \geq 0} \binom{2/3}{l} (p^{\kappa} t)^l s(\bar{2a}(b+\tau))^{2-3l} \\
&\equiv s(\bar{2a}(b+\tau))^2 + \frac{2}{3}s(\bar{2a}(b+\tau))^{-1} p^{\kappa} t \pmod{p^{\min\{n,2\kappa\}}},
\end{align*}
where the second equality is owed to the convergence in the $p$-adic topology of the power series
\begin{equation}\label{eq:ppowSer}
(1 + x)^{2/3} = \sum_{l \geq 0} \binom{2/3}{l} x^l,
\end{equation}
as long as $|x|_p < 1$. It follows from this that
\begin{equation} \label{eq:TaylorExp}
f_{T,\mbf{\eps}}(b+p^{\kappa}t) \equiv \left(bc + 3a\sum_{\tau \in T} \eps_{\tau} s(\bar{2a}(b+\tau))^2\right) + p^{\kappa}t \left(c + 2a\sum_{\tau \in T} \eps_{\tau} s(\bar{2a}(b+\tau))^{-1}\right) \pmod{p^{\min\{2\kappa,n\}}}.
\end{equation}
Analogously, we also define the ``derivative'' functions
$$
f_{T,\mbf{\eps}}^{(j)}(b) := b^{1-j}c1_{j \in \{0,1\}} + 3a \binom{2/3}{j} \sum_{\tau \in T} \eps_{\tau} s(\bar{2a}(b+\tau))^{2-3j}.
$$
\begin{proof}[Proof of Proposition \ref{prop:ftepsbd}]
We begin by noting that if $|\{\tau \in T : \eps_{\tau} \not \equiv 0 \pmod{p^{n}}\}| \leq 1$ then since $\mbf{\eps}$ is non-zero the $p$-adic stationary phase method (see e.g., Lemma 1 (1) of \cite{MilZhan}) implies that for some $\tau_0 \in T$,
\begin{align*}
\sum_{b \in \mb{Z}/p^{n}\mb{Z} \atop b \equiv d \pmod{d}} e_{p^{n}}(f_{T,\mbf{\eps}}(b)) &= \sum_{b \in \mb{Z}/p^{n}\mb{Z} \atop b \equiv d \pmod{d}} e_{p^{n}}(bc + \eps_{\tau_0} s(\bar{2a}(b+\tau_0))^2) \\
&= p^{n/2} \sum_{\substack{b \in \mb{Z}/p^{n}\mb{Z} \\ b \equiv d\pmod{p} \\ \eps_{\tau_0}s(\bar{2a}(b+\tau_0))^{2} \equiv - c \pmod{p^{\llf n/2\rrf}}}} e_{p^{n}}(bc + \eps_{\tau_0}s(\bar{2a}(b+\tau_0))^2) \\
&\ll p^{n/2},
\end{align*}
since the inner sum has at most a single summand. As $N+M \geq 1$ this satisfies the first claim, so henceforth we may assume that $|\{\tau \in T : \eps_{\tau} \not \equiv 0 \pmod{p^{n}}\}| \geq 2$. \\
The second condition is now vacuously satisfied if $0 < \delta_2n < 1$, so we may assume that $\delta_2 n \geq 1$. We put $X := \{b \pmod{p^{n}} : b \equiv d \pmod{p}\}$, noting that this is $p^{\llf \delta_2n \rrf} \mb{Z}/p^{n}\mb{Z}$-invariant. Putting
$$
\mc{E}_{|T|} := \left\{b \in \mb{Z}/p^{n}\mb{Z} : f^{(j)}(b) \equiv 0 \pmod{p^{\llf \delta_2 n\rrf}} \text{ for all $1 \leq j \leq |T|$}\right\},
$$
the proof of Proposition 8 of \cite{MilZhan} shows that 
\begin{equation}\label{eq:restoCrit}
\sum_{b \in X} e_{p^{n}}(f_{T,\mbf{\eps}}(b)) = \sum_{b \in X \cap \mc{E}_{|T|}} e_{p^{n}}(f_{T,\mbf{\eps}}(b)) + O_{|T|}\left(p^{(1-\delta_1)n}\right),
\end{equation}
as long as (writing $f_{T,\mbf{\eps}}^{(0)} = f_{T,\mbf{\eps}}$) the relations
$$
f^{(j)}_{T,\mbf{\eps}}(b+p^{\kappa}t) \equiv f^{(j)}_{T,\mbf{\eps}}(b) + p^{\kappa} t f^{(j+1)}_{T,\mbf{\eps}}(b) \pmod{p^{\min\{2\kappa,n\}}}
$$
hold for $0 \leq j \leq |T|-1$; this is guaranteed by a calculation analogous to \eqref{eq:TaylorExp}. The proof of Proposition 9 there shows that if 
$$
\rho_0 := \max_{1 \leq j \leq |T|} \nu_p \left( \binom{2/3}{j} \right) + \min_{\tau \in T} \nu_p(\eps_{\tau})
$$ 
then $\mc{E}_T = \emptyset$ whenever
$$
\min\{|\tau-\tau|_p : \tau, \tau' \in T, \tau \neq \tau', \eps_{\tau},\eps_{\tau'} \not \equiv 0 \pmod{p^n}\} \geq p^{-\delta_3(\llf \delta_2n\rrf - \rho_0)},
$$
or equivalently, whenever 
$$
\tau \not \equiv \tau' \pmod{p^{\llf \delta_3(\llf \delta_2 n\rrf-\rho)\rrf}} \text{ for all distinct $\tau,\tau'$ with non-zero $\eps_{\tau},\eps_{\tau'}.$}
$$ 
Let us assume the latter condition. Then \eqref{eq:restoCrit} yields the estimate
$$
\sum_{b \in X} e_{p^{n}}(f_{T,\mbf{\eps}}(b)) \ll_{|T|} p^{(1-\delta_1)n},
$$
and it remains to check that the required constraints on the constants hold. As noted above the proof of Proposition 8 in \cite{MilZhan}, we may take $\delta_1 = \delta_2 = 2^{-|T|}$ and $\delta_3 = \binom{|T|}{2}^{-1}$. Now, if $p \equiv 2 \pmod{3}$ then $\eps_{\tau}$ can be identified with an integer in $[-N,M]$, and we deduce that $\nu_p(\eps_{\tau}) \leq \left\lceil \log(2(N+M))/\log p \right \rceil$, which is acceptable. Thus, consider when $p \equiv 1 \pmod{3}$. We may then write 
$$
\eps_{\tau} = \alpha_{\tau} + \beta_{\tau} u_0 + \gamma_{\tau}u_0^2,
$$
and as $\phi(\mbf{\eps}) \neq 0$ such a representation exists with $|\alpha_{\tau}|,|\beta_{\tau}|,|\gamma_{\tau}| \leq N+M$ for all $\tau \in T$. Since $\mbf{\eps}$ is non-zero we may find $\tau' \in T$ such that $\eps_{\tau'} \not \equiv 0 \pmod{p^n}$, and therefore $\alpha_{\tau'},\beta_{\tau'}$ and $\gamma_{\tau'}$ are not the same. By Lemma \ref{lem:pval} we get 
$$
\min_{\tau \in T} \nu_p(\eps_{\tau}) \leq \nu_p(\eps_{\tau'}) \leq \left\lceil \frac{\log(20(N+M)^3)}{\log p}\right \rceil.
$$ 
Finally, by Lemma \ref{lem:nuPbinom}, we have
$$
\max_{1 \leq j \leq |T|} \nu_p\left(\binom{2/3}{j}\right) \leq 1_{p \leq 3|T|/2-1}.
$$
We thus obtain
$$
\rho_0 \leq 1_{p \leq 3|T|/2-1} + \left\lceil \frac{\log(20(N+M)^3)}{\log p} \right\rceil =: \rho.
$$
The claim then follows by replacing $\rho_0$ with its upper bound $\rho$ in the second alternative of the proposition, which relaxes the condition there.
\end{proof}

\begin{proof}[Proof of Proposition \ref{prop:largeN}]
Note that by the assumed lower bound for $n$, we must have $p^{n/2} > 20(N+M)^3$. 
We may partition the non-zero $\mbf{\eps}$ into the sets 
\begin{align*}
\mc{A}_T &:= \left\{\mbf{\eps} \in (\mb{Z}/p^{n}\mb{Z})^{|T|} \bk \{\mbf{0}\} : \min_{\substack{\tau \neq \tau' \\ \eps_{\tau},\eps_{\tau'} \not\equiv 0 \pmod{p^{n}}}} |\tau-\tau'|_p \geq p^{-\delta_3(\llf \delta_2 n \rrf - \rho)}\right\}\\
\mc{B}_T &:= \left((\mb{Z}/p^{n}\mb{Z})^{|T|} \bk \{\mbf{0}\}\right) \bk \mc{A}_T.
\end{align*}
Noting that $|T| \leq N+M$, we apply Proposition \ref{prop:ftepsbd} to bound 
$$
p^{(N+M)n/2} \sideset{}{^{\Diamond}} \sum_{d \pmod{p}} \sum_{\mbf{\eps} \in \mc{A}_T} \phi(\mbf{\eps}) \sum_{b \pmod{p^{n}} \atop b \equiv d \pmod{p}} e_{p^{n}}\left(f_{T,\mbf{\eps}}(b)\right) \ll_{N,M} p^{(N+M)n/2 + 1} \cdot p^{(1-\delta_1)n} \ll p^{(N+M+2-2\delta_1)n/2 + 1}.
$$
On the other hand, if $\mc{B}_T \neq \emptyset$ then the second alternative of the proposition holds. This implies the claim.
\end{proof}

\subsection{Bounds for Small $n$ and Large $p$}
The above estimates are useful for $n$ sufficiently large in terms of $N,M$. In this subsection we provide an estimate which is more efficient for $n \ll_{N,M} 1$, but with $p \gg_{N,M} 1$. 
As before, we write
$$
\sideset{}{^{\Diamond}}\sum_{d \pmod{p}} S^{\ast, \neq \mbf{0}}_{p^n}(c,\mu,\nu;a,d) = \sideset{}{^{\Diamond}}\sum_{d \pmod{p}} \sum_{\mbf{\eps} \in (\mb{Z}/p^{n}\mb{Z})^{|T|} \atop \mbf{\eps} \neq \mbf{0}} \phi(\mbf{\eps}) \sum_{b \pmod{p^{n}} \atop b \equiv d \pmod{p}} e_{p^{n}}(f_{T,\mbf{\eps}}(b)),
$$
where $T = T_{\mbf{h},\mbf{h}'}$ is as above. The estimate we prove in this case is as follows.
\begin{prop}\label{prop:RicRoyEst}
Let $\mbf{\eps} \neq \mbf{0}$. Assume that $n \leq (N+M)^32^{N+M}$. Then
$$
\sideset{}{^{\Diamond}} \sum_{d \pmod{p}} \sum_{b \pmod{p^n} \atop b \equiv d \pmod{p}} e_{p^n}(f_{T,\mbf{\eps}}(b)) \ll_{\e,N,M} p^{n-\frac{1}{n^2(n-1)^2}+\e}.
$$
\end{prop}
\begin{rem}
Note that the upper bound is trivially satisfied for $p = O_{N,M}(1)$, by choosing a suitable constant. We make no attempt to specify this dependence in the sequel, though this could be done in principle provided that one had an effective bound in the critical case $s = \frac{1}{2}k(k+1)$ in Vinogradov's mean value theorem (see Theorem \ref{thm:dec}). For some work in this direction for $s$ slightly larger, see \cite{Stei}.
\end{rem}
In preparation for the proof of Proposition \ref{prop:RicRoyEst}, note that from the identity \eqref{eq:ppowSer} for $|x|_p < 1$, given $b = d + pt$ with $t \in \mb{Z}/p^{n-1}\mb{Z}$ we can write
\begin{align*}
f_{T,\mbf{\eps}}(d+pt) &= (d+pt)c + \sum_{\tau \in T} \eps_{\tau} s(\bar{2a}(d+\tau + pt))^2 \equiv \sum_{l = 0}^{n-1} t^l a_l(\mbf{\eps},d) \pmod{p^{n}} \\
&=: P_{\mbf{\eps},d}(t) \pmod{p^{n}},
\end{align*}
where we have put
$$
a_l(\mbf{\eps},d) := p^l \left(\binom{2/3}{l}(2a)^l \sum_{\tau \in T} \eps_{\tau}s(\bar{2a}(d+\tau))^{2-3l} + cd^{1-l}1_{l \in \{0,1\}}\right).
$$
Note that $P_{\mbf{\eps},d}(t)$ is a polynomial in $t$ modulo $p^{n-1}$. To evaluate the exponential sum over $b$ (and thus over $t \pmod{p^{n-1}}$), we will roughly speaking split the set of $d$ according to the degree (modulo $p^n$) of $P_{\mbf{\eps},d}(t)$ and use bounds for Weyl sums of that degree. As we will not be able to extract cancellation when $P_{\mbf{\eps},d}$ has degree 0 or 1, we will need to check that the number of $d$ (satisfying $(\Diamond)$) for which this happens is small. To this end, we
need the following lemma, which is a modification of Proposition 4.8 of \cite{RicRoy}.
\begin{lem} \label{lem:RicRoy}
Let $w \in \mb{Z}/p\mb{Z}$ and $\tilde{T} \subset \mb{Z}/p\mb{Z}$. Let also $\tilde{\mbf{\eps}} \in (\mb{Z}/p\mb{Z})^{|\tilde{T}|} \bk \{\mbf{0}\}$. Then for $j = 1,2$,
$$
|\{d \pmod{p} : 4a^2(d+\tau) \in (\mb{Z}/p\mb{Z})^{\times 3} \text{ $\forall \ \tau \in \tilde{T}$ and } \sum_{\tau \in \tilde{T}} \tilde{\eps}_{\tau} s(\bar{2a}(d+\tau))^{2-3j} \equiv w \pmod{p}\}|
$$
is $O_{|\tilde{T}|}(1)$.
\end{lem}
\begin{proof}
The proof when $j = 2$ is completely similar to that of $j = 1$, so we focus only on the latter case. \\
As in the proof of Lemma \ref{lem:DiamBd}, we translate the problem to $\mb{F}_p$, so that e.g., $\tilde{T}$ is identified with a subset of $\mb{F}_p$. Put $\mb{F} := \mb{F}_p$ if $p \equiv 1 \pmod{3}$, and $\mb{F} := \mb{F}_p[X]/(X^2+X+1)\mb{F}_p[X]$ when $p \equiv 2 \pmod{3}$. Define also
$$
N_{\tilde{T}}(w) := |\{d \in \mb{F} : 4a^2(d+\tau) \in \mb{F}^{\times 3} \text{ $\forall \ \tau \in \tilde{T}$ and } \sum_{\tau \in \tilde{T}} \tilde{\eps}_{\tau} s((2a)^{-1}(d+\tau))^{-1} = w\}|.
$$
When $p \equiv 1 \pmod{3}$, $N_{\tilde{T}}(w)$ is precisely the count on the left-hand side in the statement of the lemma; when $p \equiv 2 \pmod{3}$, $N_{\tilde{T}}(w)$ is an upper bound for the desired quantity, since $\mb{F}_p \subseteq \mb{F}$.
In what follows we let $U_0$ denote a primitive cube root of unity in $\mb{F}$; naturally, this satisfies $1+U_0 + U_0^2 = 0$. \\
Let $\mbf{a} \in \mb{F}_p^{|\tilde{T}|}$. Consider the product
$$
Q_w(\mbf{a}) := \prod_{\mbf{j} \in \{-1,0,1\}^{|\tilde{T}|}} \left(w - \sum_{\tau \in \tilde{T}} U_0^{j_{\tau}} a_{\tau}\right),
$$
which is a polynomial of total degree $\leq 3^{|\tilde{T}|}$ in the variables $(a_{\tau})_{\tau \in \tilde{T}} \in \mb{F}_p^{|\tilde{T}|}$. Note that $Q_w(\mbf{a}) = Q_w((a_{\tau}U_0^{j_{\tau}})_{\tau \in \tilde{T}})$ for all $\mbf{j} \in \{-1,0,1\}^{|\tilde{T}|}$. This implies that we can find a polynomial $\tilde{Q}_w$, defined over $\mb{F}$, such that\footnote{This can be seen e.g., by noting that 
$$
3Q_w(\mbf{a}) = Q_w((a_{\tau})_{\tau}) + Q_w((a_{\tau} U_0^{1_{\tau'=\tau}})_{\tau}) + Q_w((a_{\tau}U_0^{21_{\tau'=\tau}})_{\tau}),
$$
then expanding the product and noting that only terms in $a_{\tau'}^3$ survive, for each $\tau' \in T$.} $Q_w(\mbf{a}) = \tilde{Q}_w((a_{\tau}^3)_{\tau \in \tilde{T}})$. In particular, we can write
$$
Q_w(\mbf{a}) = \sum_{r_1,\ldots,r_{|\tilde{T}|} \geq 0 \atop r_1 + \cdots + r_{|\tilde{T}|} \leq 3^{|\tilde{T}|-1}} b_{r_1,\ldots,r_{|\tilde{T}|}}(w) a_{\tau_1}^{3r_1} \cdots a_{\tau_{|\tilde{T}|}}^{3r_{|\tilde{T}|}},
$$
where $\{\tau_1,\ldots,\tau_{|\tilde{T}|}\}$ is an enumeration of $\tilde{T}$. Now, define
\begin{align*}
\tilde{R}_w(Y) &:= \left(\prod_{\tau \in \tilde{T}}(Y+\tau)^{3^{|\tilde{T}|-1}}\right) \cdot Q_w((\tilde{\eps}_{\tau}s((2a)^{-1}(Y+\tau))^{-1})_{\tau \in \tilde{T}}) \\
&= \sum_{r_1,\ldots,r_{|\tilde{T}|} \geq 0 \atop r_1 + \cdots + r_{|\tilde{T}|} \leq 3^{|\tilde{T}|-1}} b_{r_1,\ldots,r_{|\tilde{T}|}}(w) \prod_{1 \leq j \leq |\tilde{T}|} (2a \tilde{\eps}_{\tau_j}^3)^{r_j} \cdot \prod_{1 \leq j \leq |\tilde{T}|} (Y+\tau_j)^{3^{|\tilde{T}|-1}-r_j},
\end{align*}
which is a polynomial in $Y$ of degree $\leq |\tilde{T}|3^{|\tilde{T}|-1} \ll_{|\tilde{T}|} 1$ over $\mb{F}$. Every $d \in \mb{F}$ counted by $N_{\tilde{T}}(w)$ is a root of $\tilde{R}_w(Y)$ over $\mb{F}$, since it is a root of $Q_w((\eps_{\tau}s((2a)^{-1}(Y+\tau))^{-1})_{\tau \in \tilde{T}})$. It follows that, provided $\tilde{R}_w(Y)$ is a non-zero polynomial, we get $N_{\tilde{T}}(w) \ll_{|\tilde{T}|} 1$ as required. It therefore suffices to show that $\tilde{R}_w(Y)$ is non-zero. \\
Assume first that $w \neq 0$. The leading coefficient (in $Y$) of $R_w(Y)$ is $b_{0,\ldots,0}(w) = w^{3^{|\tilde{T}|}} \neq 0$, so that $\tilde{R}_w(Y)$ is necessarily non-zero in this case. \\
Suppose next that $w = 0$. Note that by hypothesis there is a $\tau_1 \in \tilde{T}$ such that $\tilde{\eps}_{\tau_1} \neq 0$. 
Setting $Y = -\tau_1$ and expanding the product we see that only the term with $r_1 = 3^{|\tilde{T}|-1}$ (and $r_j = 0$ for $j \neq 1$) survives, this term arising as the monomial in $a_{\tau_1}^{3^{|\tilde{T}|}}$ in $Q_0(\mbf{a})$. Explicitly, this term has coefficient
$$
[a_{\tau_1}^{3^{|\tilde{T}|}}] Q_0(\mbf{a}) = \prod_{\mbf{j} \in \{-1,0,1\}} (-U_0^{j_{\tau_1}}) = -1.
$$
It follows that $b_{3^{|\tilde{T}|-1},0,\ldots,0} = -(2a\tilde{\eps}_{\tau_1}^3)^{3^{|\tilde{T}|-1}}$, and thus
$$
\tilde{R}_0(-\tau_1) = -((2a)\tilde{\eps}_{\tau_1}^3)^{3^{|\tilde{T}|-1}} \in \mb{F}^{\times}.
$$
The claim thus follows when $w \neq 0$ as well. 
\end{proof}
For larger degree polynomials we need bounds for Weyl sums. In contrast to the work in \cite{RicRoy}, where Weyl differencing is used to obtain cancellation, we will instead apply Vinogradov's method in order to obtain a stronger Weyl sum estimate. For this, we recall the Vinogradov main conjecture, proved in the groundbreaking work of Bourgain, Demeter and Guth \cite{BDG} (and independently in the work of Wooley \cite{Woo}).
\begin{thm}[Theorem 1.1 of \cite{BDG}]
\label{thm:dec}
Let $k \geq 1$, $P \geq 1$. Given $\mbf{x} \in [0,1]^k$ put
$$
f_k(\mbf{x},P) := \sum_{1 \leq n \leq P} e(x_1n + x_2n^2 + \cdots + x_kn^k).
$$
Furthermore, for $s \in \mb{N}$ define
$$
J_{s,k}(P) := \int_{[0,1]^k} |f_k(\mbf{x},P)|^{2s} d\mbf{x}.
$$
Then 
$$
J_{s,k}(P) \ll_{\e} P^{\e}\left(P^s + P^{2s-k(k+1)/2}\right).
$$
\end{thm}
\begin{proof}[Proof of Proposition \ref{prop:RicRoyEst}]
\noindent Recall that
$$
P_{\mbf{\eps},d}(t) = \sum_{0 \leq j \leq n-1} a_j(\mbf{\eps},d)t^j \text{ for } 0 \leq t \leq p^{n-1}-1.
$$
We wish to estimate
$$
\sideset{}{^{\Diamond}}\sum_{d\pmod{p}} \sum_{t \pmod*{p^{n-1}}} e_{p^n}\left(P_{\mbf{\eps},d}(t)\right).
$$ 
Let $R \geq 1$, $N:= \frac{1}{2}p^{n-1}$. Also, put 
$$
P_{\mbf{\eps},d,t}(z) := P_{\mbf{\eps},d}(t+z) = \sum_{0 \leq i \leq n-1} z^i\sum_{i \leq j \leq n-1} \binom{j}{i} t^{j-i} a_j(\mbf{\eps},d) = \sum_{0 \leq i \leq n-1} A_i(\mbf{\eps},d,t)z^i,
$$
where we have defined
$$
A_i(\mbf{\eps},d,t) := p^i \sum_{i \leq j \leq n-1} (pt)^{j-i} \binom{j}{i}\left(\binom{2/3}{i} (2a)^j\sum_{\tau \in T} \eps_{\tau}s(\bar{2a}(d+\tau))^{2-3j} + cd^{1-j}1_{j \in \{0,1\}}\right).
$$
We define
$$
F_{\mbf{\eps}}(d) := R^{-2} \sum_{t \pmod{p^n}} S_{\mbf{\eps},d,t}(R) := R^{-2}\sum_{|t| \leq N} \sum_{1 \leq y,z \leq R} e_{p^n}(P_{\mbf{\eps},d,t}(yz)) + O(R^2).
$$
Set $\ell := n(n-1)/2$. We invoke Vinogradov's method, as exposed in Section 8.5 of \cite{IK}. Given $\alpha \in \mb{R}$ and $Y \geq 1$, define
$$
D(\alpha,Y) := Y^{-2} \sum_{|m| \leq Y} \left|\sum_{|n| \leq Y} e(\alpha mn)\right|.
$$
Now, for each $|t| \leq \frac{1}{2}p^{n-1}$, we obtain (cf. (8.76) there)
\begin{align}\label{eq:VinApp}
|S_{\mbf{\eps},d,t}(R)| \leq \left(\ell^{2(n-1)}R^{4\ell(\ell-1)+n(n-1)} J_{\ell,n-1}^2(R)\Delta(t)\right)^{\frac{1}{2\ell^2}},
\end{align}
where we have written
$$
\Delta(t) := \prod_{1 \leq h \leq n-1} D(A_h(\mbf{\eps},d,t)/p^n,\ell R^h).
$$
Our goal is to extract savings over the trivial bound $\Delta(t) \ll 1$. To this end, we consider $D(\alpha,Y)$ for $\alpha$ a $p$-adic rational. Put $\alpha = a/p^s$, for $(a,p) = 1$ and $s \geq 1$. We obtain
\begin{align*}
D(a/p^s,Y) &\ll Y^{-2} \sum_{|m| \leq Y} \min\{Y,\|am/p^s\|^{-1}\} \\
&= Y^{-2}\sum_{0 \leq r \leq s} \sum_{1 \leq |u| \leq \max\{1,\frac{1}{2}p^{s-r}\} \atop p \nmid u}  \sum_{|m| \leq Y \atop am \equiv up^r \pmod{p^s}} \min\{Y,p^{s-r}/|u|\} \\
&= Y^{-2}\sum_{0 \leq r \leq s-1} p^{s-r} \sum_{1 \leq |u| \leq \frac{1}{2}p^{s-r} \atop p \nmid u} \frac{1}{|u|} \sum_{|m| \leq Y \atop am \equiv up^r \pmod{p^s}} 1 + Y^{-1} \sum_{|m| \leq Y \atop p^s|m} 1 \\
& \ll Y^{-1}(p^s/Y+1)\log(p^s) + p^{-s},
\end{align*}
in this case. We specialize $Y = \ell R^h$ and 
$$
\alpha = A_h(\mbf{\eps},d,t)/p^n = \tilde{A}_h/p^{n-\theta_h}, 
$$
where $\theta_h = \theta_h(\mbf{\eps},d,t) := \nu_p(A_h(\mbf{\eps},d,t))$ and $p \nmid \tilde{A}_h$, for each $1 \leq h \leq n-1$. Note that $\theta_h(\mbf{\eps},d,t) \geq h$, with equality if, and only if, we have
$$
(2a)^h\binom{2/3}{h} \sum_{\tau \in T} \eps_{\tau} s(\bar{2a}(d+\tau))^{2-3h} \not \equiv -c1_{h = 1} \pmod{p},
$$
this condition being independent of $t$.\\
For $\mbf{\eps},d$ and $t$ fixed, let 
$$
\mf{D} = \mf{D}(\mbf{\eps},d,t) := \max\{1 \leq j \leq n-1 : \theta_j(\mbf{\eps},d,t) < n\},
$$ 
if this maximum is defined, and let $\mf{D}(\mbf{\eps},d,t) = 0$ otherwise. If $\mf{D} \geq 1$ then we obtain
$$
\Delta(t) \ll_{n} R^{-\mf{D}}(p^{n-\mf{D}}R^{-\mf{D}} + 1)\log p + p^{\theta_{\mf{D}}-n}.
$$
Suppose now that $\mf{D} \geq 2$. Applying Theorem \ref{thm:dec}, we obtain from \eqref{eq:VinApp} that
\begin{align*}
|S_{\mbf{\eps},d,t}(R)| &\ll_{\e,n} \left(R^{4\ell(\ell-1)+n(n-1)+\e}\left(R^{2\ell} + R^{4\ell - n(n-1)}\right)\left(R^{-\mf{D}}(p^{n-\mf{D}}R^{-\mf{D}}+1) + p^{\theta_{\mf{D}}-n}\right)\right)^{\frac{1}{2\ell^2}} \\
&\ll R^{2+\e}\left(R^{-\mf{D}/2\ell^2}((p^{n-\mf{D}}/R^{\mf{D}})^{1/2\ell^2}+1) + p^{-(n-\theta_{\mf{D}})/2\ell^2}\right),
\end{align*}
recalling that $2\ell = n(n-1)$. Taking $R = (N/\sqrt{p})^{1/2} \asymp p^{n/2-3/4}$, we deduce that
\begin{align*}
F_{\mbf{\eps}}(d) &\ll_{\e,d,t} |\{t \pmod{p^{n-1}} : \exists d : 4a^2(d+\tau) \in (\mb{Z}/p\mb{Z})^{\times 3} \forall \tau \in T \text{ and }\mc{D}(\mbf{\eps},d,t) \in \{0,1\}\}| \\
&+p^{\e} \sum_{|t| \leq N \atop \mf{D} = \mf{D}(\mbf{\eps},d,t) \geq 2} \left((N/\sqrt{p})^{-\mf{D}/4\ell^2}+(p^{n-\mf{D}}(N/\sqrt{p})^{-\mf{D}})^{1/2\ell^2} + p^{-(n-\theta_{\mf{D}})/2\ell^2}\right) + (N/\sqrt{p})^2 \\
&\ll_{\e} |\{t \pmod{p^{n-1}} : \exists d : 4a^2(d+\tau) \in (\mb{Z}/p\mb{Z})^{\times 3} \forall \tau \in T \text{ and }\mc{D}(\mbf{\eps},d,t) \in \{0,1\}\}| \\
&+p^{n-1-\frac{1}{n^2(n-1)^2}+\e}.
\end{align*}
In this way, we obtain
\begin{align*}
\sideset{}{^{\Diamond}} \sum_{d \pmod{p}} F_{\mbf{\eps}}(d) &\ll_{\e,n} \sum_{t \pmod{p^{n-1}}} |\{d \pmod{p} : 4a^2(d+\tau) \in (\mb{Z}/p\mb{Z})^{\times 3} \forall \tau \in T \text{ and } \mc{D}(\mbf{\eps},d,t) \in \{0,1\}\}| \\
&+p^{n-\frac{1}{n^2(n-1)^2} + \e}.
\end{align*}
It remains to treat the contribution from pairs $(t,d)$ for which $\mc{D}(\mbf{\eps},d,t) \in \{0,1\}$. By construction, if $\mc{D}(\mbf{\eps},d,t) = 0$ then $\theta_1(\mbf{\eps},d,t) \geq n > 1$. Similarly if $\mc{D}(\mbf{\eps},d,t) = 1$ \emph{and} $n \geq 3$ then $\theta_2(d,t) \geq n > 2$. In either of these cases, provided $p$ is large enough in terms of $n$ we get
$$
\sum_{\tau \in T} \eps_{\tau} s(\bar{2a}(d+\tau))^{2-3(\mc{D}+1)} \equiv 0 \pmod{p}.
$$
Applying Lemma \ref{lem:RicRoy}, we find that for each $t \pmod{p^{n-1}}$ the number of $d$ satisfying $(\Diamond)$ such that this latter congruence holds is $\ll_{N,M} 1$. Since $n \ll_{N,M} 1$, we deduce that
$$
\sideset{}{^{\Diamond}} \sum_{d \pmod{p}} F_{\mbf{\eps}}(d) \ll_{\e,N,M} p^{n-\frac{1}{n^2(n-1)^2}+\e} + p^{n-1} \ll p^{n-\frac{1}{n^2(n-1)^2}+\e}.
$$
It remains to consider those $d$ for which $\mc{D}(\mbf{\eps},d,t) = 1$ and $n = 2$. In this case, we know that $p^2 \nmid A_1(\mbf{\eps},d,t)$, and moreover that $A_1(\mbf{\eps},d,t) = a_1(\mbf{\eps},d)$, i.e., $A_1(\mbf{\eps},d,t)$ is independent of $t$. The sum we wish to bound is therefore
$$
\sideset{}{^{\Diamond}} \sum_{d \pmod{p} \atop p^2\nmid a_1(\mbf{\eps},d)} e_{p^2}(a_0(\mbf{\eps},d)) \sum_{t \pmod{p}} e_p\left(t(a_1(\mbf{\eps},d)/p)\right) = 0,
$$
so the bound required is satisfied in this case as well.
%
\end{proof}

\section{Applying the $K_2$ Correlations Bounds: Proof of Propositon \ref{prop:ThEst}}
In this section we will prove Proposition \ref{prop:ThEst}. Suppose $Q \geq 2$ factors as $Q = Q_0 \cdots Q_L$, where the $Q_i$ are mutually coprime. Let $p^{\nu}||Q_0$ with $\nu \geq 1$ and $p > 3$ prime. Let $C \in \mb{Z}/p^{\nu}\mb{Z}$ and $\mbf{h} \in \mb{Z}^L$, and define
\begin{multline*}
  M_{p^{\nu}}(C, \mbf{h}) := \max_{A,B \in (\mb{Z}/p^{\nu}\mb{Z})^{\times}}\left|\sum_{b \pmod*{p^{\nu}}} e_{p^{\nu}}(CBb)\prod_{I \subseteq \{1,\ldots,L\} \atop |I| \equiv 0 \pmod{2}} K_2(A,b+H_I;p^{\nu})\right.\\
  \left.  \cdot \prod_{J \subseteq \{1,\ldots,L\} \atop |J| \equiv 1 \pmod{2}} \bar{K}_2(A,b+H_J;p^{\nu})\right|,
\end{multline*}
where for each $I \subseteq \{1,\ldots,L\}$ we set $H_I := \sum_{i \in I} Q_ih_i$. 
Writing $N$ to denote the number of subsets of $\{1,\ldots,L\}$ of even cardinality and $M$ the number of subsets with odd cardinality, the total number of $K_2$ factors in the sum is clearly $N + M = 2^L$.\\
Define now $T_{\mbf{h}} := \{H_I : I \subseteq \{1,\ldots,L\}\}$, as well as 
\begin{align*}
\mu_{\mbf{h}}(\tau) &:= |\{I \subseteq \{1,\ldots,L\} : |I| \equiv 0 \pmod{2}, H_I \equiv \tau \pmod{p^{\nu}}\}| \\
\nu_{\mbf{h}}(\tau) &:= |\{I \subseteq \{1,\ldots,L\} : |I| \equiv 1 \pmod{2}, H_I \equiv \tau \pmod{p^{\nu}}\}|.
\end{align*}
Further, define 
$$
\mc{T}_{p^{\nu}} := \{\mbf{h} \in \mb{Z}^L : \mu_{\mbf{h}}(\tau) \equiv \nu_{\mbf{h}}(\tau) \pmod{3} \text{ for all } \tau \in T_{\mbf{h}}\}.
$$
The following lemma, which is analogous to \cite[Lemma 4.5]{Irv}, allows us to control how frequently either $\mbf{h} \in \mc{T}_{p^{\nu}}$ or $\mu_{\mbf{h}}(\tau) \equiv \nu_{\mbf{h}}(\tau) \pmod{3}$. 
\begin{lem} \label{lem:combo}
Let $p > 3$, $p\mid Q_0$ and let $\mbf{h} \in \mb{Z}^{L}$. 
If $\mu_{\mbf{h}}(\tau) \equiv \nu_{\mbf{h}}(\tau) \pmod{3}$ for all $\tau \in T_{\mbf{h}}$ then there is $1 \leq i \leq L$ such that $p|h_i$. In particular, if $\mbf{h} \in \mc{T}_{p^{\nu}}$ then $p\mid \prod_{1 \leq i\leq L} h_i$. \\
\end{lem}

\begin{proof}
If $\mbf{h} \in \mc{T}_{p^{\nu}}$ then by definition $\mu_{\mbf{h}}(\tau) \equiv \nu_{\mbf{h}}(\tau) \pmod{3}$ for all $\tau \in T$, and the second assertion immediately follows from the first. Thus, it suffices to prove that $p|\prod_{1 \leq i \leq L} h_i$ whenever $3|(\mu_{\mbf{h}}(\tau)-\nu_{\mbf{h}}(\tau))$ for all $\tau \in T$. \\
We denote
$$
E(\mbf{h}) := \prod_{1 \leq i \leq L} \left(1-e_p\left(Q_ih_i\right)\right).
$$ 
Let $\Phi_p(z) := \sum_{0 \leq j \leq p-1} z^j = \prod_{a \in (\mb{Z}/p\mb{Z})^{\times}} (z-e_p(a))$ the cyclotomic polynomial of order $p$. Assume for the sake of contradiction that $p \nmid h_i \forall i$. Then
$$
\prod_{a \in (\mb{Z}/p\mb{Z})^{\times}} E(a\mbf{h}) = \prod_{1 \leq i \leq L} \prod_{a \in (\mb{Z}/p\mb{Z})^{\times}} (1-e_p(ah_iQ_i)) = \prod_{b \in (\mb{Z}/p\mb{Z})^{\times}} (1-e_p(b))^L = \Phi_p(1)^L = p^L,
$$
using the fact $(Q_i,Q_0) = 1$ and thus that $p \nmid a h_iQ_i $ whenever $p \nmid a $, for all $1 \leq i \leq L$. \\
On the other hand, we have
$$
E(a\mbf{h}) = \prod_{1 \leq i \leq L} (1-e_p(ah_iQ_i)) = \sum_{I \subseteq \{1,\ldots,L\}} (-1)^{|I|} e_p(aH_I) = \sum_{\tau \in T} (\mu_{\mbf{h}}(\tau)-\nu_{\mbf{h}}(\tau))e_p(a\tau),
$$
and since $3|(\mu_{\mbf{h}}(\tau)-\nu_{\mbf{h}}(\tau))$ for each $\tau \in T$ we can find $m_{\tau} \in \mb{Z}$ such that
$$
E(a\mbf{h}) = 3\sum_{\tau \in T} m_{\tau} e_p(a\tau).
$$
It follows that
\begin{align*}
p^L &= \prod_{a \in (\mb{Z}/p\mb{Z})^{\times}} E(a\mbf{h}) = 3^{p-1}\prod_{\sg \in \text{Gal}(\mb{Q}(e_p(1))/\mb{Q})} \sg\left(\sum_{\tau \in T} m_{\tau} e_p(\tau)\right) \\
&= 3^{p-1} N_{\mb{Q}(e_p(1))/\mb{Q}}\left(\sum_{\tau \in T} m_{\tau} e_p(\tau)\right).
\end{align*}
Note that the bracketed sum over $\tau \in T$ on the right-hand side is an algebraic integer in $\mb{Q}(e_p(1))$, so its norm is a rational integer. This implies that $3 \mid p^L$, which is false since $p > 3$. This contradiction implies that there must be an $i$ for which $p \mid h_i$, as claimed.
\end{proof}

Using Lemma \ref{lem:combo}, we can finally give the following bounds on the correlation sums $M_{p^{\nu}}(C,\mbf{h})$. 
\begin{prop} \label{prop:bdSumm}
Let $L \geq 2$. Fix parameters $Y = Y(L) = 2^{3L+5}$ and $Z = Z(L) = 2^{3L}2^{2^L}$. Let $p^{\nu} || Q_0$, and let $C \in \mb{Z}/p^{\nu} \mb{Z}$. Assume additionally that:
\begin{equation}\label{eq:pAdAVCond}
\min_{I,J \subseteq \{1,\ldots,L\}} \{|H_I-H_J|_p : H_I \neq H_J\} \geq \begin{cases} p^{-r(\nu)} \text{ if $p > Y$ and $\nu > Z$} \\
p^{-\llf r(\nu)/2\rrf} \text{ if $3 < p \leq Y$, $\nu > Z$,}
\end{cases}
\end{equation}
where $r(\nu) := 2^{1-2L}\llf \nu 2^{-2^L}\rrf - 1$.
Then we have
\begin{align*}
\frac{M_{p^{\nu}}(C,\mbf{h})}{p^{\nu(2^{L-1}+1)}} &\ll_{\e,L} (p^{-1/2} + 1_{C = 0})1_{p^{\nu-1}|C}1_{p|\prod_i h_i} \\
&+ \begin{cases} 
p^{-1/2} &\text{ if $\nu = 1$} \\
p^{-\nu/2^{2^L}}  
&\text{ if $p > Y$ and $\nu > Z$} \\
p^{-\nu/2^{2^L}}  
&\text{ if $3 < p \leq Y$ and $\nu > Z$} \\
p^{-1/\nu^2(\nu-1)^2 + \e} &\text{ if $p > Y$ and $2 \leq \nu \leq Z$}.
\end{cases}
\end{align*}
\end{prop}

\begin{proof}
Write $Q_0 = \mf{q}_1\mf{q}_2\mf{q}_3\mf{q}_4\mf{q}_5$, where the $\mf{q}_j$ satisfy $(\mf{q}_i,\mf{q}_j) = 1$ for $i \neq j$, according to the following rules: 
\begin{enumerate}[(i)]
\item $p||Q_0 \Rightarrow p|\mf{q}_1$ 
\item $p^{\nu} || Q_0, p > Y$ and $\nu > Z \Rightarrow p^{\nu}|\mf{q}_2$ 
\item $p^{\nu} || Q_0, p \leq Y$ and $\nu > Z \Rightarrow p^{\nu} | \mf{q}_3$
\item $p^{\nu} || Q_0, p > Y$ and $\nu \leq Z \Rightarrow p^{\nu}|\mf{q}_4$
\item $\mf{q}_5 = Q_0/(\mf{q}_1\mf{q}_2\mf{q}_3\mf{q}_4)$.
\end{enumerate}
We summarize (in an effective form) the bounds that the previous sections imply for each prime power divisor of $Q_0$, according to which of the $\mf{q}_i$ they divide.\\
(I) If $p| \mf{q}_1$, Theorem \ref{thm:boundCorrCompl} yields
\begin{align}\label{eq:MpCh}
M_p(C,\mbf{h}) \ll 3^{2^L}p^{2^{L-1}} \left(2^L \sqrt{p} + p1_{p\mid C} 1_{\mc{T}_p}(\mbf{h})\right).
\end{align}
By Lemma \ref{lem:combo}, if $\mbf{h} \in \mc{T}_p$ then $p|\prod_i h_i$, so we may conclude that
\begin{align} \label{eq:mfq1}
M_p(C,\mbf{h}) \ll_L p^{2^{L-1}+1}\left(p^{-1/2} + 1_{p|C}1_{p|\prod_i h_i}\right),
\end{align}
which implies the bound in the first case. \\
(II) Suppose $p^{\nu}||\mf{q}_2$ and $p^{\nu-1} \nmid C$ then as $|T| \leq 2^L$. With a view towards applying Theorem \ref{thm:expsumPrimPow}, recall the definition
$$
\rho = 1_{p \leq 3\cdot 2^{L-1}-1} + \left \lceil \frac{\log(20 \cdot 2^{3L})}{\log p}\right\rceil,
$$
Since $p > 2^{3L+5}$, we have $\rho = 1$, and so
$$
2|T|^{-2}(\llf \nu 2^{-|T|}\rrf - \rho) = 2^{1-L}(\llf 2^{-2^L}\nu\rrf - 1) = r(\nu).
$$
In light of \eqref{eq:pAdAVCond}, Theorem \ref{thm:expsumPrimPow} yields
\begin{align*} 
M_{p^{\nu}}(C,\mbf{h}) \ll_L p^{\nu(2^{L-1}+1)-\nu/2^{2^L}}. 
\end{align*}
If, in addition $p^{\nu-1}|C$, then Lemma \ref{lem:combo} implies that
\begin{align}\label{eq:mfq2}
M_{p^{\nu}}(C,\mbf{h}) \ll_L \max_{a \pmod{p^r}} p^{\nu(2^{L-1}+1)}\left(p^{-\nu/2^{2L}} 
+1_{p|\prod_i h_i}\left(p^{-1/2} + 1_{C = 0}\right)\right),
\end{align}
and the bound in the second case is proved. \\
(III) Suppose $p^{\nu} || \mf{q}_3$ then $3 < p \leq 2^{3L+5}$. We thus have the uniform bound
$$
\rho \leq 1 + \left\lceil \frac{\log(2^{3L+5})}{\log 5}\right \rceil \leq 1 + 2(L+2).
$$
As in case (II), we have
\begin{align}\label{eq:mfq3}
M_{p^{\nu}}(C,\mbf{h}) \ll_L 
p^{\nu(2^{L-1}+1)}\left(p^{-\nu/2^{2L}} 
+ 1_{p|\prod_i h_i}\left(p^{-1/2} + 1_{C = 0}\right)\right)
\end{align}
where we used the fact that since $Z$ is chosen suitably, 
$$
2^{1-2L}\left(\llf \nu 2^{-2^L}\rrf-2L-5\right) \geq \nu 2^{-2L-2^L} \geq r(\nu)/2
$$ 
whenever $L \geq 2$. \\
(IV) If $p^{\nu}||\mf{q}_4$ then, again by Theorem \ref{thm:expsumPrimPow},
\begin{align}\label{eq:mfq4}
M_{p^{\nu}}(C,h) \ll_{\e,L} p^{\nu(2^{L-1}+1)}\left(p^{-\frac{1}{\nu^2(\nu-1)^2}+\e} + 1_{p^{n-1}|C}\left(p^{-1/2} + 1_{C = 0}\right) 1_{p|\prod_i h_i}\right),
\end{align}
as required. \\
(V) For all of the primes powers $p^{\nu}||\mf{q}_5$ the bound given is trivial.
\end{proof}

The following simple lemma allows us to bound the number of tuples $\mbf{h}$ where \eqref{eq:pAdAVCond} fails.
\begin{lem}\label{lem:ctrlHI}
Let $c \in (0,2^{-2^L}]$, and let $d|Q_0$. Assume that $K/Q_j \geq Q_0^{2c}$ for all $1 \leq j \leq L$.  Then the number of tuples $\mbf{h} \in\mb{Z}^L$ with $1 \leq |h_j| \leq K/Q_j$ for all $j$, such that for each $p^{\nu}||d$
$$
\min_{\substack{I,J \subseteq \{1,\ldots,L\} \\ H_I \neq H_J}} |H_I-H_J|_p < p^{- c\nu}
$$
is $\ll 2^{L} \tau(d)^{2L-1} d^{-c}K^L/(Q_1\cdots Q_L)$.
\end{lem}
\begin{proof}
Write $d = \prod_{1 \leq i \leq m} p_i^{\nu_i}$, so that $m = \omega(d)$. For each $1 \leq i \leq m$, choose a pair of disjoint subsets $I_i,J_i \subseteq \{1,\ldots,L\}$, such that $I_i\cup J_i \neq \emptyset$. We may define a matrix $A_{\mbf{I},\mbf{J}} = \{a_{i,j}\}_{i,j}$ with integer entries via
$$
a_{i,j} := \begin{cases} Q_j &\text{ if $j \in I_i$,} \\ -Q_j &\text{ if $j \in J_i$,} \\ 0 &\text{ otherwise.}\end{cases}
$$
By composing $A_{\mbf{I},\mbf{J}}$ with projections, we may view it as a homomorphism $A_{\mbf{I},\mbf{J}}: \mb{Z}^L \ra \prod_{1 \leq i \leq m} (\mb{Z}/p_i^{\lceil c\nu_i \rceil}\mb{Z})$, such that for each $1\leq i \leq m$ the $i$th entry of $A_{\mbf{I},\mbf{J}}\mbf{h}$ is
$$
(A_{\mbf{I},\mbf{J}}\mbf{h})_i := \sum_{1 \leq j \leq L} a_{i,j} h_j \pmod{p_i^{\lceil c\nu_i \rceil}} = \sum_{l \in I_i} Q_lh_l - \sum_{l \in J_i} Q_lh_l \pmod{p_i^{\lceil c\nu_i \rceil}},
$$
whenever $\mbf{h} \in \mb{Z}^L$. Note that $\text{ker}(A_{\mbf{I},\mbf{J}})$ is a lattice in $\mb{Z}^L$ with covolume $\tilde{d} := \prod_{1 \leq i \leq m} p_i^{\lceil c\nu_i\rceil} \geq d^c$, and trivially $A_{\mbf{I},\mbf{J}}$ is injective on the quotient $\mb{Z}^L/\text{ker}(A_{\mbf{I},\mbf{J}})$, with the unique reduced zero class  being $\mbf{0}$. By lattice periodicity, (and $\tilde{d}\leq Q_0^{2c} \leq K/Q_j$ for each $j$) we deduce that
\begin{equation}\label{eq:lattBd}
|\{\mbf{h} \in \text{ker}(A_{\mbf{I},\mbf{J}}) : 1 \leq |h_j| \leq K/Q_j \forall 1 \leq j \leq L\}| \ll \tilde{d}^{-1}(2K)^L/(Q_1\cdots Q_L) \ll d^{-c} (2K)^L/(Q_1\cdots Q_L).
\end{equation}
Given these remarks, we may prove the lemma as follows. Let $\mc{H}_d$ be the set of tuples $\mbf{h}$ described in the statement of the lemma. Given $\mbf{h} \in \mc{H}_d$, for each prime $p^{\nu}||d$ there are distinct subsets $I_p = I_p(\mbf{h})$, $J_p = J_p(\mbf{h})$ of $\{1,\ldots,L\}$ such that $|H_{I_p}-H_{J_p}|_p \leq p^{-\lceil c\nu\rceil}$. Now, put $I_p' := I_p \bk (I_p \cap J_p)$ and $J_p' := J_p \bk (I_p \cap J_p)$, so that $I_p' \cup J_p' \neq \emptyset$, $I_p'$ and $J_p'$ are disjoint, and moreover $H_{I_p}-H_{J_p} = H_{I_p'}-H_{J_p'}$. We thus have, for each $p^{\nu}||d$,
$$
\sum_{l \in I_p'} Q_l h_l - \sum_{l \in J_p'} Q_lh_l \equiv 0 \pmod{p^{\lceil c\nu \rceil}}.
$$
By the previous definitions, it follows that $\mbf{h} \in \text{ker}(A_{\{I_p'\}_p,\{J_p'\}_p})$. Hence, we obtain the upper bound
$$
|\mc{H}_d| \leq \sum_{\substack{\mbf{I},\mbf{J} \subseteq \{1,\ldots,L\}^m \\ I_i \cap J_i = \emptyset \forall i \\ I_i \cup J_i \neq \emptyset \forall i}} |\{\mbf{h} \in \text{ker}(A_{\mbf{I},\mbf{J}}) : 1 \leq |h_j| \leq K/Q_j \forall 1 \leq j \leq L\}|.
$$
The number of pairs of tuples of sets $\mbf{I},\mbf{J}$ is $\leq \binom{2^L}{2}^m \leq 2^{(2L-1)\omega(d)} \leq \tau(d)^{2L-1}$, so by \eqref{eq:lattBd} we obtain
$$
|\mc{H}_d| \ll 2^L\tau(d)^{2L-1}d^{-c}K^L/(Q_1\cdots Q_L),
$$
as claimed.
\end{proof}


\begin{proof}[Proof of Proposition \ref{prop:ThEst}]
i) Assume $Q_0,\ldots,Q_L$ are all coprime and squarefree. Case (I) of Proposition \ref{prop:bdSumm} gives 
$$
M_p(C,\mbf{h}) \ll_L p^{2^{L-1}+1/2}\min\{(p,C)^{1/2},(p,\prod_i h_i)^{1/2}\},
$$ 
for each $p\mid Q_0$. Combining this with Lemma \ref{lem:toPrime}, we obtain that
\begin{align*}
  T(\mbf{h}) &= \sum_{k \in J(\mbf{h})}\prod_{I \subseteq \{1,\ldots,L\}} \mc{C}^{|I|} K_2\left(b,k+\sum_{i \in I} Q_ih_i; Q_0\right)\ll \sum_{C \in \mb{Z}/(Q_0\mb{Z})} \min\left\{\frac{|J(\mbf{h})|}{Q_0}, \frac{1}{Q_0\|C/Q_0\|}\right\} \prod_{p\mid Q_0} M_p(C,\mbf{h}) \\
             &\ll_{\e,L} Q_0^{2^{L-1}+1/2+\e} \sum_{C \in \mb{Z}/(Q_0\mb{Z})} \min\left\{\frac{|J(\mbf{h})|}{Q_0}, \frac{1}{Q_0\|C/Q_0\|}\right\} \prod_{p\mid q_0} \min\{(p,C)^{1/2},(p,\prod_i h_i)^{1/2}\} \\
             &\ll_{\e} Q_0^{2^{L-1}+1/2+\e}\sum_{d\mid Q_0 \atop d < Q_0} \frac{(d,\prod_i h_i)^{1/2}}{d} \asum_{C' \pmod*{Q_0/d}} \frac{1}{C'}  + \frac{K}{Q_0} (Q_0,\prod_i h_i)^{1/2} \\
             &\ll_{\e} (1+K/Q_0) Q_0^{2^{L-1}+1/2+\e}(Q_0,\prod_i h_i)^{1/2},
\end{align*}
using $|J(\mbf{h})| \leq K$ and the obvious inequality $(d,\prod_i h_i) \leq (Q_0, \prod_i h_i)$ for any $d|Q_0$.
%
%
%
Employing $(Q_0,\prod_i h_i)^{1/2} \leq \prod_i (Q_0,h_i)^{1/2}$ and summing in $h_1,\ldots,h_L$ yields
\begin{align*}
&\sum_{1 \leq |h_1| \leq K/Q_1} \cdots \sum_{1 \leq |h_L| \leq K/Q_L} |T(\mbf{h})| \\
&\ll_{\e,L} (1+K/Q_0) Q_0^{2^{L-1}+1/2+\e} \prod_{1 \leq j \leq L} \quad \sum_{1 \leq |h_j| \leq K/Q_j} (Q_0,h_j)^{1/2} \\
&\ll (1+K/Q_0) Q_0^{2^{L-1}+1/2+\e}\sum_{e_1,\ldots,e_L\mid Q_0} (e_1\cdots e_L)^{1/2} \quad \prod_{1 \leq j \leq L} \quad\sum_{1 \leq |h_j'| \leq K/(e_jQ_j) \atop (h_j',Q_0) = 1} 1 \\
&\ll_L (1+K/Q_0) Q_0^{2^{L-1}+1/2+\e} \frac{K^L}{Q_1\cdots Q_L} \left(\sum_{e\mid Q_0} e^{-1/2}\right)^L \\
&\ll_{\e,L} (1+K/Q_0) Q_0^{2^{L-1}+1/2+\e}\frac{K^L}{Q_1\cdots Q_L} = (1+K/Q_0) Q_0^{2^{L-1}+3/2+\e}K^L/Q,
\end{align*}
using $Q/Q_0 = Q_1\cdots Q_L$ in the last line. This implies i).\\
ii) As above, we would like to estimate
$$
\mc{T} := \sum_{1 \leq |h_1| \leq K/Q_1} \cdots \sum_{1 \leq |h_L| \leq K/Q_L} \sum_{C \pmod{Q_0}} \min\left\{\frac{K}{Q_0},\frac{1}{Q_0\|C/Q_0\|}\right\} \prod_{p^{\nu}||Q_0} M_{p^{\nu}}(C,\mbf{h}).
$$
Here, we are assuming that $(Q_0,6) = 1$. We factor $Q_0 = \mf{q}_1\mf{q}_2\mf{q}_3\mf{q}_4\mf{q}_5$, where 
$$
\mf{q}_1 = \prod_{p||Q_0} p, \ \ \mf{q}_2 = \prod_{p^{\nu}||Q_0 \atop p>Y,\nu > Z} p^{\nu}, \ \ \mf{q}_3 = \prod_{p^{\nu}||Q_0 \atop 3 < p \leq Y, \nu > Z} p^{\nu}, \mf{q}_4 = \prod_{p^{\nu}||Q_0 \atop p > Y, 2 \leq \nu \leq Z} p^{\nu} \text{ and } \mf{q}_5 = \prod_{p^{\nu}||Q_0 \atop p \leq Y, 2 \leq \nu \leq Z} p^{\nu}.
$$
By construction, $\mf{q}_5 \ll_L 1$, so that
$$
\prod_{p^{\nu}||\mf{q}_5} M_{p^{\nu}}(C,\mbf{h}) \ll \mf{q}_5^{(N+M+2)/2} \ll_L 1.
$$
We focus next on the prime power divisors of $Q_0' := Q_0/\mf{q}_5$. As above, define $r(\nu) := \llf 2^{1-2L}\llf \nu 2^{-2^L-2L}\rrf \rrf$. 
For each $1 \leq i \leq 4$ and each $p^{\nu}||\mf{q}_i$, we write
$$
M_{p^{\nu}}(C,\mbf{h}) = p^{(2^{L-1}+1)\nu}\left(M_{p^{\nu},1}(C,\mbf{h}) + M_{p^{\nu},2} + M_{p^{\nu},3}(\mbf{h})\right),
$$
where we have defined 
\begin{align*}
M_{p^{\nu},1}(C,\mbf{h}) &= \begin{cases} 1_{p^{\nu-1}|C} 1_{p|\prod_j h_j}\left(p^{-1/2} + 1_{p^{\nu}|C}\right) &\text{ if $i = 2,3,4$} \\ 1_{p|C}1_{p|\prod_j h_j} &\text{ if $i = 1$,} \end{cases}
\end{align*}
as well as
\begin{align*}
M_{p^{\nu},2} =  \begin{cases} p^{-1/2} &\text{ if $i = 1$} \\ p^{-\nu/2^{2^L}} &\text{ if $i = 2,3$} \\ p^{-\nu/Z^5+\e} &\text{ if $i = 4$} \end{cases} , \ \ \ \ \ 
M_{p^{\nu},3}(\mbf{h}) = \begin{cases} 0 &\text{ if $i = 1,4$} \\ 1 &\text{ if $i = 2,3$ and \eqref{eq:pAdAVCond} fails.}
\end{cases} 
\end{align*}
Note that $M_{p^{\nu},1}$ depends only on $(C,Q_0') = \prod_{1 \leq i \leq 4} (C,\mf{q}_i)$. We factor $C= C'd$ and $d = d_1d_2d_3d_4d_5$, where $d_j = (C,\mf{q}_j)$ for each $1 \leq j \leq 5$, and get
\begin{align}
\mc{T} &\ll_L \sum_{d|Q_0 \atop d=d_1d_2d_3d_4d_5} \left(\asum_{C' \pmod{Q_0/d}} \min\left\{\frac{K}{Q_0},\frac{1}{Q_0\|C'd/Q_0\|}\right\}\right) \nonumber\\
&\cdot \sum_{1 \leq |h_1| \leq K/Q_1} \cdots \sum_{1 \leq |h_L| \leq K/Q_L} \prod_{1 \leq i \leq 4} \prod_{p^{\nu}||\mbf{q}_i} p^{(2^{L-1}+1)\nu}\left(M_{p^{\nu},1}(d_i,\mbf{h}) + M_{p^{\nu},2} + M_{p^{\nu},3}(\mbf{h})\right) \nonumber\\
&\ll_{\e,L} Q_0^{2^{L-1}+1}\ssum_{e_1f_1 = \mf{q}_1} f_1^{-1/2} \ssum_{e_4f_4 = \mf{q}_4} f_4^{-Z^{-5}+\e}\ssum_{e_2f_2g_2 = \mf{q}_2} f_2^{-2^{-2^L}} \ssum_{e_3f_3g_3 = \mf{q}_3} f_3^{-2^{-2^L}} \nonumber\\ 
&\cdot \sum_{\substack{d|Q_0 \\ d = d_1d_2d_3d_4d_5 \\ p^{\nu}||e_j \Rightarrow p^{\nu-1}|d_j \\ 1 \leq j \leq 4}} \left(\asum_{C' \pmod{Q_0/d}} \min\left\{K/Q_0,\frac{1}{Q_0\|C'd/Q_0\|}\right\}\right) \left(\prod_{1 \leq j \leq 4} \prod_{p|e_j} \left(p^{-1/2}1_{j \neq 1} + 1_{\nu_p(e_j) = \nu_p(d_j)}\right)\right) \nonumber\\
&\cdot 
\mathop{\sum_{1 \leq |h_1| \leq K/Q_1} \cdots \sum_{1 \leq |h_L| \leq K/Q_L}}_{\substack{p|g_2g_3 \Rightarrow \exists j : M_{p^{\nu},3} \neq 0 \\ \text{rad}(e_1e_2e_3e_4)|\prod_j h_j}} 1, \label{eq:hsumWithConds}
\end{align}
where the summation symbol $\ssum_{uvw = x}$ indicates that $(u,v) = (v,w) = (w,u) = 1$ (and analogously $\ssum_{uv = x}$ indicates that $(u,v) = 1$), so that if $p|u$, say, then $\nu_p(u) = \nu_p(x)$. \\  
By dropping the constraint on $e_1,e_2,e_3,e_4$, the innermost sum over the tuples $\mbf{h}$ can be bounded above by
$$
|\{\mbf{h} \in \mb{Z}^L : 1 \leq |h_j| \leq K/Q_j \forall j, p|g_2g_3 \Rightarrow \min_{\substack{I,J \subseteq \{1,\ldots,L\} \\ H_I \neq H_J}} |H_I-H_J| < p^{-r(\nu)/2} \}|.
$$
If we define 
$$
\delta' := \min\left\{2^{-2^L},Z^{-5},\min_{\nu > Z} \frac{\llf r(\nu)/2 \rrf}{\nu}\right\}.
$$
then by applying Lemma \ref{lem:ctrlHI} with $c = \delta'$ and $d = g_2g_3$ (which is a divisor of $Q_0$), and using the crude bound $\tau(d)^{2L-1} \ll_{\e,L} d^{\e}$ we may bound this cardinality by
$$
\ll_{\e,L} X^{\e} \frac{K^L}{(g_2g_3)^{\delta'} Q_1\cdots Q_L} = X^{\e} \frac{K^LQ_0}{(g_2g_3)^{\delta'} Q}.
$$
Inserting this into our earlier upper bound for $\mc{T}$, we obtain
%
\begin{align*}
\mc{T} &\ll_{\e,L} \frac{X^{\e}K^LQ_0^{2^{L-1}+2}}{Q} \ssum_{e_1f_1 = \mf{q}_1 \atop e_4f_4 = \mf{q}_4} \frac{1}{f_1^{1/2}f_4^{Z^{-5}}} \ssum_{e_2f_2g_2 = \mf{q}_2 \atop e_3f_3g_3 = \mf{q}_3} \frac{1}{(f_2f_3)^{2^{-2^L}}(g_2g_3)^{\delta'}}\\ 
&\cdot \sum_{\substack{d|Q_0 \\ d= e_1d_2d_3d_4d_5 \\ \nu_p(d_j) \geq \nu_p(e_j) - 1 \\ 2 \leq j \leq 4}} \left(\asum_{C' \pmod{Q_0/d}} \min\left\{K/Q_0,\frac{1}{Q_0\|C'd/Q_0\|}\right\}\right) \left(\prod_{2 \leq j \leq 4} \prod_{p|e_j} \left(p^{-1/2} + 1_{\nu_p(e_j) = \nu_p(d_j)}\right)\right).
\end{align*}
We parametrize $d = e_1d_2d_3d_4d_5$ with $d_j = e_jD_j$, where $D_j|\text{rad}(e_j)$, for each $2 \leq j \leq 4$. Using $\delta' \leq 2^{-2^L}$ and $\delta' \leq Z^{-5}$, we find
\begin{align*}
\mc{T} &\ll_{\e,L} \frac{X^{\e}K^LQ_0^{2^{L-1}+2}}{Q} \ssum_{e_1f_1 = \mf{q}_1 \atop e_4f_4 = \mf{q}_4} \frac{1}{f_1^{1/2}f_4^{\delta'}} \ssum_{e_2f_2g_2 = \mf{q}_2 \atop e_3f_3g_3 = \mf{q}_3} \frac{1}{(f_2f_3)^{\delta'}(g_2g_3)^{\delta'}}\\ 
&\cdot \sum_{D_j|\text{rad}(e_j) \atop 2 \leq j \leq 4} (D_2D_3D_4)^{-1/2} \left(\asum_{C' \pmod{Q_0/(e_1e_2D_2e_3D_3e_4D_4)}} \min\left\{K/Q_0,\frac{1}{Q_0\|C'e_1e_2D_2e_3D_3e_4D_4/Q_0\|}\right\}\right).
\end{align*}
If we bound the sum over $C'$ trivially, we obtain $\ll \frac{\log Q_0}{e_1e_2e_3e_4D_2D_3D_4}$, which finally leads to
\begin{align*}
\mc{T} &\ll_{\e,L} \frac{X^{\e} K^LQ_0^{2^{L-1}+2}}{Q} \ssum_{e_1f_1 = \mf{q}_1 \atop e_4f_4 = \mf{q}_4} \frac{1}{(f_1f_4)^{\delta'}e_1e_4} \ssum_{e_2f_2g_2 = \mf{q}_2 \atop e_3f_3g_3 = \mf{q}_3} \frac{1}{(f_2f_3)^{\delta'}(g_2g_3)^{\delta'}e_2e_3} \\
&\ll_{\e} \frac{X^{\e} K^LQ_0^{2^{L-1}+2}}{Q} (\mf{q}_1\mf{q}_2\mf{q}_3\mf{q}_4)^{-\delta'} \ll_L X^{\e}K^LQ_0^{2^{L-1}+2-\delta'}/Q,
\end{align*} 
again using $\mf{q}_5 \ll_L 1$ and $Q_0 = \mf{q}_1 \mf{q}_2\mf{q}_3\mf{q}_4 \mf{q}_5$. This proves claim ii).
\end{proof}

\section{Proof of Theorems \ref{thm:Smooth} and \ref{thm:notSfree}}
\noindent This section is devoted to the proofs of Theorems \ref{thm:Smooth} and \ref{thm:notSfree}. \\
Let $\e > 0$ be sufficiently small and let $0 < \eta < 1/522$. Let $X$ be sufficiently large in terms of $\e$, and let $X^{2/3-\e} < q \leq X^{3/4+1/1044-\e}$ be $X^{\eta}$-smooth. We wish to show that there is a $\delta = \delta(\e,\eta) > 0$ such that
$$
\Delta_{\mu^2}(X;q,a) = \sum_{n \leq X \atop n \equiv a \pmod{q}} \mu^2(n) - \frac{1}{\phi(q/(a,q))} \sum_{n \leq X \atop (n,q) = (a,q)} \mu^2(n) \ll X^{1-\delta}/q,
$$
for any residue class $a \pmod{q}$ with $(a,q) \leq X^{\e}$, with further size constraints on $q$ depending on whether or not $q$ is squarefree.  \\
Replacing $X,q$ and $a$ by $X/(a,q)$, $q/(q,a)$ and $a/(a,q)$ if necessary, we may reduce our problem to the case in which $a$ is coprime to $q$: indeed, $q/(q,a)$ is still $X^{\eta}$-smooth, and as $(a,q) \leq X^{\e}$ the bound above is implied by the case $(a,q) = 1$ (upon taking $\delta$ slightly smaller). In what follows, we will focus solely on when the residue class $a$ is coprime to $q$. \\
We begin by summarizing the analysis in Section \ref{sec:setupEst}, which is valid when $(a,q) = 1$. Let $\delta > 0$ be fixed, to be selected later. We fixed a scale $V_0 \geq X^{\delta+\e}$, and found a second scale $\max\{X^{1-\delta-\e}/(qV_0),V_0\} \leq V_1 \ll X^{1/2}$ and an interval $I(V_1) = (V_1,V_1 + V_1/V_0]$ such that (see \eqref{eq:impWeil})
\begin{align}
\Delta_{\mu^2}(X;q,a) &\ll_{\e} X^{\e}\frac{V_0V_1}{\tilde{q}} \sum_{f|\tilde{q} \atop f \leq Z} \asum_{k \pmod*{\tilde{q}/f}} \sum_{m \pmod*{\tilde{q}/f} \atop m \neq 0} \frac{\kappa(m;ka,\tilde{q}/f)}{km} + V_0\left( 1 + X^{\e}V_0\tilde{q}^{1/2}Z^{-3/2}\right)  \nonumber\\
&+X^{\e}\left(\frac{X}{qV_0} + \frac{X}{\tilde{q}V_1} + \left(\frac{X}{Z\tilde{q}}\right)^{1/2} + KV_1^2\left(\frac{Z}{\tilde{q}}\right)^{3/2}\right), \label{eq:DeltaMuSumm}
\end{align}
for any divisor $\tilde{q}$ of $q$ and any $Z \geq 1$. We recall here that, given $Q \geq 1$ and $K\geq 1$, we have set
$$
\kappa(M,N;Q) := \max_{1 \leq R \leq K} \left|\sum_{K(M-1) < r \leq K(M-1) + R} e_Q(-rV_1) K_2(N,r;Q)\right|,
$$
for each $M \geq 1$ and $N \in \mb{Z}$ coprime to $Q$. \\
To proceed, we need to be able to choose $\tilde{q}$ of suitable size in order to obtain the desired $O_{\e}(X^{1-\delta}/q)$ bound. The following lemma will be useful in this vein, especially in order to apply the results of the last few sections. 
\begin{lem} \label{lem:factSmooth}
Let $\eta > 0$ and suppose $X \geq 3$ is large enough relative to $\eta$. Let $q \in \mb{N}$ be $X^{\eta}$-smooth. \\
a) If $0 < v < 1$ then there is a divisor $q'$ of $q$ such that $q' \in (q^v,X^{\eta}q^v]$. \\
b) Assume furthermore that $q \in (X,2X]$ is $X^{\eta}$-ultrasmooth. Let $k \geq 4$, and let $u_1,\ldots,u_k \in (\eta,1-\eta)$ with $u_1 + \cdots + u_k = 1$. Then we can find $q_1,\ldots,q_k$, mutually coprime with $(q_k,6) = 1$, such that $q_1 \in (X^{u_1},X^{u_1+\eta}2^{\nu_2(q)}3^{\nu_3(q)}]$, $q_2 \in (X^{u_2},X^{u_2+\eta}]$, $q_j \in (X^{u_j},X^{u_j+\eta}]$ for each $3 \leq j \leq k-1$ and $q_k \in (X^{u_k-(k-1)\eta}2^{-\nu_2(q)}3^{-\nu_3(q)},2X^{u_k}]$ such that $q = q_1\cdots q_k$. 
\end{lem}
\begin{proof}
a) Enumerate the prime factors of $q$ in ascending order as $p_1 \leq p_2 \leq \cdots \leq p_{\omega(q)}$. Let $r$ be chosen maximally such that $p_1\cdots p_r \leq q^v$. Then $q' := p_1\cdots p_rp_{r+1} > q^v$ by maximality, and moreover $q' \leq q^v p_{r+1} \leq q^vX^{\eta}$. This establishes the first claim. \\
b) Arguing similarly as in a), order the prime power factors of $q$ as $p_1^{\alpha_1} < \cdots < p_{\omega(q)}^{\alpha_{\omega(q)}} \leq X^{\eta}$. Pick $1 \leq N_1 \leq N$ to be the minimal integer such that $\prod_{1 \leq j \leq N_1} p_j^{\alpha_j} > X^{u_1}$, and set $q_1 := \prod_{1 \leq j \leq N_1} p_j^{\alpha_j}$. By minimality, $q_1/p_{N_1}^{\alpha_{N_1}} \leq X^{u_1}$, whence $q_1 \in (X^{u_1} , X^{u_1+\eta}]$. \\
We set $X' := X/q_1 \in [X^{1-u_1-\eta},X^{1-u_1})$, and $q' := q/q_1$. As $u_2 > 0$, $N_1 < N$. We then select $N_1 < N_2 \leq N$ such that $q_2 := \prod_{N_1 < j \leq N_2} p_j^{\alpha_j} > X^{u_2}$, so that $q_2 \in (X^{u_2},X^{u_2+\eta}]$. \\
Repeating this process $k-1$ times, we obtain integers $q_1,\ldots,q_{k-1}$ such that $q_j \in (X^{u_j},X^{u_{j}+\eta}]$. We replace $q_1$ by $q_12^{\nu_2(q)}3^{\nu_3(q)}$ if the powers of 2 and 3 dividing $q$ are not already divisors of $q_1,\ldots,q_{k-1}$. The factors $q_j$ are mutually coprime by construction. Putting $q_k := q/(q_1\cdots q_{k-1})$, the above construction and $q \in (X,2X]$ forces $(q_k,6) = 1$ and
$$
q_k \in (X^{1-u_1-\ldots-u_{k-1}-(k-1)\eta}2^{-\nu_2(q)}3^{-\nu_3(q)},2X^{1-u_1-\ldots-u_{k-1}}],
$$ 
which implies the claim since $u_k = 1-u_1-\ldots - u_{k-1}$ by definition.
\end{proof}
\subsection{First Result: $q \leq X^{3/4-\e}$} \label{subsec:3of4}
We begin by considering the easier range $X^{2/3-\e} < q \leq X^{3/4-\e}$, in which we need not assume that $q$ is squarefree. Let $\eta/12 \leq \delta < 1/50$. By Lemma \ref{lem:factSmooth} a), we can choose a divisor $\tilde{q}$ of $q$ with $\tilde{q} \in (X^{1/2 + 12\delta}, X^{1/2+12\delta+\eta}]$. Take $V_0 = K = X^{\delta+\e}$ and $Z= X^{10\delta}$. Using Lemma \ref{lem:WeilK2}, we can bound
$$
\max_{q'\mid\tilde{q}} \max_{1 \leq k' \leq q' \atop (k',q') = 1}  \max_{1 \leq m \leq q'/K} \kappa(m;k'a,q') \ll K \tilde{q}^{1/2+\e} \ll X^{\delta+\e} X^{1/4 + 6\delta + \eta/2}.
$$
From \eqref{eq:DeltaMuSumm} and the above parameter choices, as well as the lower bound $V_1 \geq X^{1-\delta-\e}/(qV_0) \geq X^{1/4-2\delta-3\e}$ that we may assume (see \eqref{eq:assumpV1} above), we obtain
\begin{align*}
\Delta_{\mu^2}(X;q,a) &\ll_{\e} \frac{X^{2\delta+2\e + 1/2}}{\tilde{q}^{1/2}}+X^{2\delta + 3\e-15\delta}\tilde{q}^{1/2} + X^{\delta+\e} + \frac{X^{1-\delta}}{q} + \frac{X^{1+\e}}{\tilde{q}V_1} \\
&+ X^{\delta+\e}X^{1/2(1-10\delta)}\tilde{q}^{-1/2}+ X^{1+\delta + \e + 15\delta}\tilde{q}^{-3/2} \\
&\ll_{\e} X^{1/4-4\delta + 2\e} + X^{1/4-7\delta + \eta/2 +\e} + X^{1/4-\delta + \e} + \frac{X^{1-\delta}}{q} +X^{1/4-10\delta + 2\e} \\
&+ X^{1/4-10\delta + 4\e} + X^{1/4-2\delta + \e} \\
&\ll X^{1-\delta}/q,
\end{align*}
%
%
%
%
%
%
%
using the fact that $\eta \leq 12\delta$ in the last line. This implies the claim for $q \leq X^{3/4-\e}$. 
\begin{rem}
The above proof shows that we can obtain power-savings in the range $q \leq X^{3/4-\e}$ for $q$ that are $X^{\eta}$-smooth with any $\eta < 6/25$, since in this case the choice $\delta = \eta/12$ is admissible.
\end{rem}
\subsection{Squarefree $X^{3/4-\e} < q \leq X^{3/4+1/1044-\e}$}
In this range of the modulus $q$ we will invoke Propositions \ref{prop:qvdC} and \ref{prop:ThEst} i), and to this end we need a suitable divisor $\tilde{q}$ of $q$ \emph{as well as} a suitable factorization for $\tilde{q}$. We begin this subsection by outlining the parameter choices needed to this end. \\
Let $L \geq 2$. Fix $0 < \delta < 1/10$ to be chosen later, and let $\theta := 3/4+\lambda$, where $\eta/2 < \lambda < \min\{1/20,1/(2L)\}$. We will determine constraints on $\lambda$ momentarily, from which we will conclude that any $\lambda < 1/1044$ will be admissible. \\
Set $\gamma := 2\delta +\lambda + \e$, for $\e > 0$ small, and put
$$
\sg := \frac{1}{L} + \frac{2(2^{L+2}+ L)}{L}\gamma,
$$
assuming that $\gamma$ and $L$ are chosen so that this is $< 1/4+\lambda$. Also, put $K := \llf X^{\sg/2-\lambda}\rrf$. 
Suppose that $q \in (X^{\theta},2X^{\theta}]$ is squarefree and $X^{\eta}$-smooth. As $\sg < 1/4 + \lambda$, by Lemma \ref{lem:factSmooth} a) we may choose $\tilde{q} \in (X^{1/2 + \sg},X^{1/2+\sg+\eta}]$.  \\
Of course, $\tilde{q}$ is also $X^{\eta}$-smooth. Applying Lemma \ref{lem:factSmooth} b), we can find a factorization $\tilde{q} = \tilde{q}_0 \tilde{q}_1\cdots \tilde{q}_L$, such that
\begin{align*}
\tilde{q}_{L-j+1} &\in (X^{\sg/2-(2^j+1) \gamma - \eta}, X^{\sg/2-(2^j+1)\gamma}] \text{ for all $1 \leq j \leq L-1$,} \\
\tilde{q}_1 &\in (X^{\sg/2-(2^L+1)\gamma-\eta},6X^{\sg/2-(2L+1)\gamma}], \\
\tilde{q}_0 &\in [X^{\sg - (2^{L+1}+2)\gamma}/6, 2X^{\sg-(2^{L+1}+2)\gamma + L\eta}).
\end{align*}
This is indeed possible since $\tilde{q} \in (X^{1/2+\sg},X^{1/2+\sg+\eta}]$ and
\begin{align*}
&\sg-(2^{L+1}+2)\gamma + L\eta + \sum_{1 \leq j \leq L} \left(\frac{\sg}{2} -(2^j+1) \gamma - \eta\right) = (L/2+1)\sg - (2^{L+2}+ L)\gamma \\
&= \frac{1}{2} + \sg + \frac{L}{2}\left(\sg - \frac{1}{L} - \frac{2(2^{L+2}+ L)}{L}\gamma\right) = \frac{1}{2}+\sg.
\end{align*}
Note that, by construction, we have $K \geq \max\{\tilde{q}_1,\ldots,\tilde{q}_L\}$, and moreover
\begin{align} \label{eq:tilde_qjBd}
\max\left\{\left(\frac{\tilde{q}_0^{1/2}}{K}\right)^{2^{-L}}, \left(\frac{\tilde{q}_{L-j+1}}{K}\right)^{2^{-j}}\right\} \ll X^{-\gamma} \text{ for all $1 \leq j \leq L$.}
\end{align}
In addition, since $\gamma \leq \sg/2$ then
\begin{equation}\label{eq:tilde_q0Term}
\tilde{q}_0^{-2^{-(L+1)}} \ll X^{-2^{-(L+1)}\sg + (1+2^{-L})\gamma} \ll X^{-\gamma}.
\end{equation}
Next, select $V_0 = X^{\delta+\e}$ and $Z = X^{\frac{2}{3}(\sg/2+4\delta+2\lambda-2\e)}$, which satisfies $Z \leq \tilde{q}/K$. We apply all of these parameter choices in \eqref{eq:DeltaMuSumm}, as well as the inequalities
\begin{enumerate}[(i)]
\item $\eta/2 < \lambda < 1/20$ and $\delta < 1/10$, 
\item $X^{1-\delta}/q \geq \frac{1}{2}X^{1/4-\delta-\lambda}$,
\item $\sg > 10\gamma > 10(\delta+\lambda)$, and
\item $X^{1-\delta-\e}/(qV_0) \leq V_1 \ll X^{1/2}$,
\end{enumerate}
to obtain
\begin{align} \label{eq:DeltaMuX}
&\Delta_{\mu^2}(X;q,a) \ll_{\e} X^{\frac{1}{2} + \delta + 2\e}\tilde{q}^{-1} \max_{f|\tilde{q} \atop f \leq Z} \asum_{k \pmod*{\tilde{q}/f}} \sum_{m \pmod*{\tilde{q}/f} \atop m \neq 0} \frac{|\kappa(m;ka,\tilde{q}/f)|}{km} + X^{2\delta - (\sg/2+4\delta + 2\lambda-2\e) + 3\e}\tilde{q}^{1/2} \nonumber\\
&+ \frac{X^{1-\delta}}{q} + \frac{X^{1+\e}}{\tilde{q}V_1} + X^{1/2-\frac{1}{3}(\sg/2+4\delta + 2\lambda-2\e)+ \e}\tilde{q}^{-1/2} + X^{1+\sg/2+(\sg/2+4\delta+2\lambda-2\e)-\lambda+\e}\tilde{q}^{-3/2} \nonumber \\
&= X^{-\sg + \delta + 2\e} \max_{f|\tilde{q} \atop f \leq Z} \asum_{k \pmod*{\tilde{q}/f}} \sum_{m \pmod*{\tilde{q}/f} \atop m \neq 0} \frac{|\kappa(m;ka,\tilde{q}/f)|}{km} + \frac{X^{1-\delta}}{q}\cdot X^{- \delta-\lambda +\eta/2+ 5\e}  \nonumber\\
&+ \frac{X^{1-\delta}}{q}\left(1 + X^{-\sg+3\delta+2\lambda+2\e} + X^{-\frac{2\sg}{3}-\delta/3 + \lambda/3+5\e/3} + X^{-\sg/2+5\delta+\lambda+\eta-\e}\right) \nonumber\\
&\ll_{\e} X^{-\sg + \delta + 2\e} \max_{f|\tilde{q} \atop f \leq Z} \asum_{k \pmod*{\tilde{q}/f}} \sum_{m \pmod*{\tilde{q}/f} \atop m \neq 0} \frac{|\kappa(m;ka,\tilde{q}/f)|}{km} + \frac{X^{1-\delta}}{q},
\end{align}
provided that $\e > 0$ is sufficiently small. \\
Fix $f_0 \mid \tilde{q}$ with $f_0 \leq Z$ that maximizes the sum over $f$ in \eqref{eq:DeltaMuSumm}, and let $q' := \tilde{q}/f_0$. We can factor $q' = q_0' \cdots q_L'$, as in Proposition \ref{prop:qvdC} by writing $\tilde{q} = \tilde{q}_0\cdots \tilde{q}_L$ and setting $q_j' := \tilde{q}_j/(\tilde{q}_j,f_0)$, for each $j$. 
Putting $Q = q'$ and recalling that $K \geq \max_{1 \leq j \leq L} \tilde{q}_j$, we may combine Proposition \ref{prop:ThEst} i) with Proposition \ref{prop:qvdC} to get that for each $1 \leq m \leq q'-1$ and each $1 \leq k \leq q'-1$ coprime to $q'$,
\begin{align*}
  |\kappa(m;ka,q')| &\ll_{\e,L} (q')^{1/2+\e}K\left(\sum_{1 \leq j \leq L} \left(\frac{q'_{L-j+1}}{K}\right)^{2^{-j}}+ \left(\frac{q'(q')^{\e}(q_0')^{2^{L-1}+3/2}}{K^{L+1}(q_0')^{2^{L-1}+1}} \frac{(2K)^L}{q'}(K/q'_0+1)\right)^{2^{-L}}\right) \\
&\ll_L (\tilde{q}/f_0)^{1/2+\e} K\left(\sum_{1 \leq j \leq L} \left(\frac{\tilde{q}_{L-j+1}}{(f_0,\tilde{q}_{L-j+1})K}\right)^{2^{-j}} + \left(\frac{\sqrt{\tilde{q}_0/(\tilde{q}_0,f_0)}}{K}\right)^{2^{-L}} + (\tilde{q}_0/(\tilde{q}_0,f_0))^{-2^{-L-1}}\right) \\
&\ll K\tilde{q}^{1/2}X^{-\gamma} \ll X^{\frac{1}{4} + \sg + \eta/2-\gamma-\lambda},
\end{align*}
using $f_0^{-1/2+\e}(\tilde{q}_0,f_0)^{2^{-L-1}} \ll_{\e} 1$ whenever $L \geq 2$. Inserting this bound into \eqref{eq:DeltaMuX} and summing $m$ and $k$, when $\e$ is sufficiently small we obtain
\begin{align*}
\Delta_{\mu^2}(X;q,a) &\ll_{\e} X^{-\sg + \delta + 3\e} \cdot X^{\frac{1}{4}+\sg+\eta/2-\gamma-\lambda} + \frac{X^{1-\delta}}{q} \\
&=  \frac{X^{1-\delta}}{q}\left(1+ X^{-\lambda+\eta/2+ 3\e}\right) \ll \frac{X^{1-\delta}}{q}.
\end{align*}
It remains to show that any $\lambda < 1/1044$ is admissible. With the parameter choices made earlier, we have assumed that 
$$
\sg = \frac{1}{L} + \frac{2(2^{L+2}+ L)\gamma}{L} < \frac{1}{4} + \lambda < \frac{1}{4} + \gamma,
$$
which forces $L \geq 5$, and 
$$
\gamma \leq \frac{L/4-1}{2^{L+3}+L}.
$$
Since the bounds are decreasing in $L$, we deduce by setting $L = 5$ that
$$
\lambda < \gamma \leq \frac{1}{1044}.
$$
Furthermore, as $\gamma = 2\delta + \lambda$, we may always choose $\delta$ small enough so that any $\lambda > 1/1044 -\e'$ is obtainable, for any $\e' > 0$. We thus deduce that for $\eta > 0$ sufficiently small (in particular, smaller than $2\lambda$) we can find $\delta = \delta(\eta,\e)$ such that if $q \leq X^{3/4+1/1044-\e}$ is smooth and squarefree then 
$$
\Delta_{\mu^2}(X;q,a) \ll_{\e} X^{1-\delta}/q,
$$
for any residue class $a$ modulo $q$.
%

\subsection{Non-squarefree $q > X^{3/4-\e}$}
The proof follows similar lines to that of Theorem \ref{thm:Smooth}, but invoking Proposition \ref{prop:ThEst} ii), rather than i). The choice of parameters can be rigged up similarly as in the previous proof, save that the factors $\tilde{q}_0,\ldots,\tilde{q}_L$ must be chosen to satisfy $K \geq  \max\{\tilde{q}_1,\ldots,\tilde{q}_L\}$, and also
$$
\max\left\{\left(\frac{\tilde{q}_0^{1-\delta'}}{K}\right)^{2^{-L}},\left(\frac{\tilde{q}_{L-j+1}}{K}\right)^{2^{-j}}\right\} \ll X^{-\gamma},
$$
in analogy to \eqref{eq:tilde_qjBd} above, which amounts to replacing the condition $\tilde{q}_0^{-2^{-(L+1)}} \ll X^{-\gamma}$ by $\tilde{q}_0^{-(1-\delta')2^{-L}} \ll X^{\gamma}$. We also must choose $\tilde{q}_0$ to be coprime to $2$ and $3$ in order to apply the results of Section \ref{sec:corPrimePow}, but by the $X^{\eta}$-ultrasmooth condition this can be done at a cost of $X^{2\eta}$ in precision in the choice of $\tilde{q}_1$ (as is done explicitly in Lemma \ref{lem:factSmooth} b). \\
Finally, in order to apply Lemma \ref{lem:ctrlHI} we must assume $K/\tilde{q}_j \geq \tilde{q}_0^{\delta'}$ for all $1 \leq j \leq L$, where $\delta' = \delta'(L) \in (0,2^{-2^L}]$ arises in Proposition \ref{prop:ThEst}. This can be assured in light of the choice $\sg$ from the previous subsection: up to $X^{\eta}$ factors (where $\eta$ can be chosen as small as desired), we have, uniformly in $1 \leq j \leq L$,
$$
K/\tilde{q}_j \geq \frac{1}{2} X^{2^j \gamma + 2\delta} \geq \frac{1}{2}X^{4\gamma + 2\delta},
$$
with $\gamma = 2\delta + \lambda$, whereas
$$
\tilde{q}_0^{\delta'} \leq 2 X^{2\sg \delta' + \delta'\eta L} \leq X^{\delta'(2/L + \eta L) + (2^{L+3}+2L)\gamma\delta'/L},
$$
so that as $\delta' \leq 2^{-2^L}$ it suffices, for instance, to have $\gamma + \delta > \delta'(1/L + \eta L)$. We leave to the interested reader the determination of an explicit choice of $\lambda$ and $\delta$ (both of which necessarily depend on $\delta'$, which could be reduced if necessary) in which the range $q \leq X^{3/4 + \lambda}$ is admissible with power savings $X^{1-\delta}/q$.
%
%

\begin{proof}[Proof of Corollary \ref{cor:posProp}]
Put $Y := X^{196/261-\e}$, and set $u := (\log Y)/(\eta\log X)$. It suffices to show that 
$$
|\{q \leq Y : P^+(q) \leq X^{\eta}, \mu^2(q) = 1\}| \gg_{\eta} Y.
$$
Of course, we have
\begin{align*}
\sum_{q \leq Y} \mu^2(q)1_{P^+(q) \leq X^{\eta}} = \sum_{d \leq Y^{1/2}} \mu(d)1_{P^+(d) \leq X^{\eta}} \sum_{m \leq Y/d^2} 1_{P^+(m) \leq X^{\eta}}.
\end{align*}
Put $D := (\log X)^{1/2}$. Then we trivially bound
$$
\sum_{D < d \leq Y^{1/2}} \mu(d)1_{P^+(d) \leq X^{\eta}} \sum_{m \leq Y/d^2} 1_{P^+(m) \leq X^{\eta}} \ll Y\sum_{d > D} d^{-2} \ll Y(\log X)^{-1/2}.
$$
On the other hand, standard estimates for smooth numbers (see e.g. Theorem III.5.8 of \cite{Ten}) yield
\begin{align*}
&\sum_{d \leq D} \mu(d)1_{P^+(d) \leq X^{\eta}} \sum_{m \leq Y/d^2} 1_{P^+(m) \leq X^{\eta}} \\
&= Y\sum_{d \leq D} \frac{\mu(d)\rho(\log(Y/d^2)/\eta \log X)}{d^2}1_{P^+(d) \leq X^{\eta}} + O_{\eta}(Y D/\log X)
\\
&= \frac{6}{\pi^2}\rho(u) Y + O_{\eta}\left(Y\sum_{d \leq D} \frac{|\rho(u) - \rho(u-v_d)|}{d^2} + Y(\log X)^{-1/2}\right),
\end{align*}
where $\rho$ denotes the Dickman function, and for each $d \leq D$ we set $v_d := \frac{2(\log d)}{\eta \log X}$. As
$$
w\rho'(w) = -\rho(w-1)
$$
for $w > 1$, we observe that for each $d \leq D$,
\begin{align*}
|\rho(u-v_d)-\rho(u)| &= \left|\int_{u-v_d}^u \rho'(t) dt\right| \leq v_d \max_{0 \leq t \leq v_d} |\rho'(u-t)| \\
&\ll_{\eta} \frac{\log\log X}{\log X} \frac{\rho(u-2)}{u} \ll_{\eta} \frac{\log\log X}{\log X},
\end{align*}
provided $X$ is large enough in terms of $\eta$. We thus deduce that
$$
|\{q \leq Y : P^+(q) \leq X^{\eta}, \mu^2(q) = 1\}| = \frac{6}{\pi^2}\rho(u)Y \left(1+ O_{\eta}(1/\sqrt{\log X})\right) \gg_{\eta} Y,
$$
as $u \ll_{\eta} 1$, and the claim follows.
\end{proof}

\bibliography{SmoothSfree}
\bibliographystyle{plain}
\end{document}